\documentclass[11pt]{amsart}

\usepackage{enumerate,comment}
\usepackage{physics,esdiff}
\usepackage{mathtools}
\mathtoolsset{showonlyrefs}

\newcommand{\rn}{\mathbb R^n}
\newcommand{\s}{\mathbb{S}}
\newcommand{\sn}{\s^{n-1}}
\newcommand{\kno}{\mathcal K^n_o}
\newcommand{\kne}{\mathcal K^n_e}

\newcommand{\supp}{\operatorname{supp}}

\usepackage[bookmarks=true,
bookmarksnumbered=true, breaklinks=true,
pdfstartview=FitH, hyperfigures=false,
plainpages=false, naturalnames=true,
colorlinks=true,
pdfpagelabels]{hyperref}
\hypersetup{
  pdftitle={Lp Weighted Minkowski Problem},
  pdfauthor={Langharst, Liu and Tang}
}

\calclayout

\newtheorem{theorem}{Theorem}[section]
\newtheorem{lemma}[theorem]{Lemma}
\newtheorem{corollary}[theorem]{Corollary}
\newtheorem{proposition}[theorem]{Proposition}

\theoremstyle{definition}
\newtheorem{definition}[theorem]{Definition}

\theoremstyle{remark}

\numberwithin{equation}{section}

\newcommand{\hess}[1]{\operatorname{Hess}(#1)}
\newcommand{\R}{\mathbb{R}} 

\newcommand{\conbod}{\mathcal{K}^n}

\newcommand{\vol}{\text{\rm Vol}}

\title[The Weighted $L^p$ Minkowski Problem]{The Weighted $L^p$ Minkowski Problem}

\author{Dylan Langharst}
\address{Department of Mathematical Sciences \\ Carnegie Mellon University \\ Pittsburgh PA 15213 \\ USA}
\email{dlanghar@andrew.cmu.edu}

\author{Jiaqian Liu}
\address{School of Mathematics and Statistics \\ Henan University \\ Jinming Avenue \\ 475001, Kaifeng
\\
China}
\email{liujiaqian@henu.edu.cn}

\author{Shengyu Tang}
\address{Institute of Mathematics \\ Hunan University \\
Lushan S Road \\
410082, Changsha \\ China}
\email{tsy@hnu.edu.cn}



\begin{document}
\begin{abstract}
The Minkowski problem in convex geometry concerns showing that a given Borel measure on the unit sphere is, up to perhaps a constant, some type of surface area measure of a convex body. Two types of Minkowski problems in particular are an active area of research: $L^p$ Minkowski problems, introduced by Lutwak and (Lutwak, Yang, and Zhang), and weighted Minkowski problems, introduced by Livshyts. For the latter, the Gaussian Minkowski problem, whose primary investigators were (Huang, Xi and Zhao), is the most prevalent. In this work, we consider weighted surface area in the $L^p$ setting. We propose a framework going beyond the Gaussian setting by focusing on rotationally invariant measures, mirroring the recent development of the Gardner-Zvavitch inequality for rotationally invariant, log-concave measures. Our results include existence for all $p \in \R$ (with symmetry assumptions in certain instances).
We also have uniqueness for $p \geq 1$ under a concavity assumption. Finally, we obtain results in the so-called \textit{small mass regime} using degree theory, as instigated in the Gaussian case by (Huang, Xi and Zhao).
\end{abstract}

\maketitle

\section{Introduction}
\subsection{Motivation}
Given a finite Borel measure $\mu$ on the unit sphere $\sn$ in the $n$-dimensional Euclidean space $\R^n$, one may ask: does there exist a unique (up to translations) convex body $K$ (compact, convex set with non-empty interior) such that $dS_K=d\mu$? Here, $S_K$ is the surface area measure of $K$, which is obtained by pushing the Hausdorff measure on $\partial K$, the boundary of $K$, to $\sn$, i.e.,  for every Borel $A \subset \sn,$ $$S_K(A)=\mathcal{H}^{n-1}(n^{-1}_K(A)),$$ where $\mathcal{H}^{n-1}$ is the $(n-1)$-dimensional Hausdorff measure and $n_K:\partial K \rightarrow \sn$ is the Gauss map, which associates an element $y$ of $\partial K$ with its outer unit normal.

Minkowski's existence theorem \cite[p. 455]{Sh1} shows that if $\mu$ satisfies the following two conditions, then the answer is yes:
\begin{enumerate}
    \item The measure $\mu$ is not concentrated on any great hemisphere, that is
    \label{con_1}
    $$\int_{\sn}\langle \theta, \xi \rangle_+d\mu(\xi) > 0 \quad \text{for all } \theta\in\sn.$$
    \item The measure is centered, that is
    $$\int_{\sn}\xi d\mu(\xi)=0.$$
    \label{con_2}
\end{enumerate}
The concept of the surface area measure of a convex body has seen many extensions over the past century, initiated by a series of papers by Lutwak \cite{LE93,LE96} in the `90s establishing the modern $L^p$ Brunn-Minkowski theory. It will be convenient to denote by $\kno$ the set of all convex bodies containing the origin in their interior, and by $\kne$ the set of all (origin)-symmetric convex bodies. A set $K$ is origin symmetric if $K=-K$. For brevity, we may simply say symmetric.

The question of showing that a given data measure on the sphere is the appropriate surface area measure of a convex body is a \textit{Minkowski problem}. In many instances, one must introduce a constant into the Minkowski problem; in fact, this is more standard than not in the weighted setting (also in the Orlicz setting, see \cite{HLYZ18}). We say that a Minkowski problem is \textit{constant-free} if there is no constant. In certain instances, only \textit{even} data measures $\nu$ on the sphere are considered when solving a Minkowski problem, in which case the convex body whose relevant area measure fits the data will be symmetric.  

The $L^p$ Minkowski problem, concerning the $L^p$ surface area measure $S_{K,p},$ is perhaps the most notable example of a Minkowski problem. Letting the support function of a convex body $K$ be $h_K(x)=\sup_{y\in K}\langle x, y\rangle$, we have $dS_{K,p}=h_K^{1-p}dS_K$ for $K\in\kno$. Existence in the even case was completely solved \cite{LE93,LYZ04,LYZ06,BLYZ13} for $p\geq 0$ (note that, for $p=n$, one must introduce a constant). In fact, there is also uniqueness for $p\geq 1, p\neq n$ (uniqueness when $p=0$ is still an open question, known as the famed Log-Minkowski problem). For $p>1, p\neq n$, in the non-even case, existence and uniqueness were settled in \cite{HLYZ05,CW06}. For $p\in (0,1)$, existence was settled in \cite{CLZ17}. The regime $p<1$, especially uniqueness and non-uniqueness (there are examples of non-uniqueness), is still an active area of research. See, e.g., \cite{AS02,AS03,MI13,HL14,GZ14,GZ15_2,GZ15,JLW15,HLW16,BH16,BHZ16,BH17,BT17,CLZ17,CLZ19,QL19,BBCY19,BBC20,CHLL20,CFL23,EM24,IM24}. The reader is recommended to read the recent survey by B\"or\"oczky \cite{KB23} for a detailed history of the $L^p$ Minkowski problem.

In this work, we will focus on a particular type of Minkowski problem. Given a Borel measure $\mu$ on a metric measure space $(\mathcal{R},d)$, the Minkowski content, or $\mu$-weighted surface area, of a measurable set $A\subset \mathcal{R}$ is given by
\begin{equation}
    \mu^+(\partial A)=\liminf_{\epsilon \to 0}\frac{\mu(A_\epsilon)-\mu(A)}{\epsilon},
    \label{eq:minkowski_content}
\end{equation}
where, as usual, $A_\epsilon = \{x\in\mathcal{R}:d(x,A)\leq \epsilon\}$. In this work, $\mathcal{R}$ will always be $\R^n$ for some $n$ and the metric $d$ will always be the Euclidean distance $|\cdot|$. In this case, $A_\epsilon=A+\epsilon B_2^n.$ Here, $K+L=\{x+y:x\in K,y\in L\}$ is the so-called Minkowski sum of Borel sets $K$ and $L$. 

In our work, we will require that $\mu$ satisfies an {\it isoperimetric inequality}. Such inequalities control the Minkowski content $\mu^+(\partial A)$ of $A$ by $\mu(A)$. Given a probability measure $\mu$ on $\R^n$, its \textit{isoperimetric constant} $h(\mu)$ is the largest constant $h$ such that
\begin{equation}
\label{eq:iso_constant}
\mu^+(\partial A) \geq h\min\{\mu(A),1-\mu(A)\}.\end{equation}
The existence of $h(\mu)$, with bounds, when $\mu$ is log-concave was shown by Kannan, Lov\'{a}sz, and Simonovits \cite{KLS95}, and improved by Bobkov \cite{BG07}. Inequalities of the form \eqref{eq:iso_constant} are called \textit{Cheeger's inequality}. Given a Borel measure $\mu$ on $\R^n$, its \textit{isoperimetric function} $I_\mu$ is given by 
\begin{equation}
    I_\mu(t)=\inf\{\mu^+(\partial A):\mu(A)=t, \;\; t\in (0,\mu(\R^n))\}.
    \label{eq:I_mu}
\end{equation}
Therefore, it is the largest function $I$ such that, for every Borel set $A\subset\R^n$ with $\mu(A)\in (0,\mu(\R^n)),$
\begin{equation} \mu^+(\partial A) \geq I(\mu(A))
\label{eq:Iso_eq}
\end{equation}
However, the infimum in \eqref{eq:I_mu} may be zero in general; we say $\mu$ ``has'' an isoperimetric function if $I_\mu >0$. Please, see the excellent textbook by Ledoux \cite[Page 23]{LM01}. For a concrete example, \eqref{eq:iso_constant} shows $I_\mu$ exists when $\mu$ is a log-concave measure with integrable density. If $B$ is so that $\mu^+(\partial B)=I_\mu(\mu(B))$, then $B$ is said to be an extremal set. In the setting of Borel measures on $\R^n$, the isoperimetric function is essentially only known for the Gaussian and Lebesgue measures. In practice, one usually establishes that a function $I$ satisfies \eqref{eq:Iso_eq} (i.e., a \textit{bound} for $I_\mu$ is shown).

A natural question is a formula for \eqref{eq:minkowski_content}. We will follow the terminology by Livshyts \cite{GAL19}: for a convex body $K$ and a Borel measure $\mu$ on the boundary of $K$ of the form $d\mu(y)=\phi(y)d\mathcal{H}^{n-1}(y)$, the \textit{weighted surface area of $K$ with respect to $\mu$} is defined by
\begin{equation}
    \label{eq:surface_mu}
    S^{\mu}_{K}(E)=\int_{n_K^{-1}(E)}\phi(y)d\mathcal{H}^{n-1}(y)
\end{equation}
for every Borel set $E \subset \sn.$ 

We briefly mention how to define $S^\mu_K$ when $\mu$ is a Borel measure on $\R^n$. Essentially, the problem boils down to determining a canonical method to select how the measure behaves on the boundary of each convex body. We will always work with measures that have density. We recall that a measure $\mu$ on $\R^n$ is said to have density if it is absolutely continuous with respect to the Lebesgue measure, i.e., $d\mu(x)=\phi(x)dx$ for a nonnegative, locally integrable function $\phi$. Under the minor technical assumption that $\mu$ has continuous density, there is no issue. For a given fixed convex body $K$, the minimal required assumption is that $\partial K$, up to a set of $(n-1)$-dimensional Hausdorff measure zero, is included in the Lebesgue set of the density of $\mu$. We can make this precise.

 Under the assumption that $K$ is a convex body and that $\mu$ is a Borel measure on $\R^n$ with density $\phi$ whose Lebesgue set contains $\partial K$, one has that the $\liminf$ in \eqref{eq:minkowski_content} is a limit and 
$$\mu^+(\partial K)=\lim_{\epsilon \to 0}\frac{\mu(K+\epsilon B_2^n)-\mu(K)}{\epsilon}=\int_{\sn}dS^\mu_K(u).$$
If we additionally assume that $K\in\kno$, its weighted $L^p$ surface area is then, for $E\subset\sn$ Borel,
\begin{equation}
\label{eq:l_pSura_weighted}
    S^{\mu}_{K,p}(E) = \int_{n^{-1}_K(E)}\langle y, n_K(y) \rangle^{1-p}\phi(y)d\mathcal{H}^{n-1}(y),
\end{equation}
i.e., one has $dS^\mu_{K,p}=h_K^{1-p}dS^\mu_K$ and $S_{K,1}^\mu=S^\mu_K.$ We outline in Section~\ref{sec:surface_area} formulas for more general versions of \eqref{eq:minkowski_content} that we will need throughout this work. Before listing our results, it is necessary to contextualize our results with respect to the broader literature.

\subsection{Weighted Minkowski Problems} In this paper, we establish a framework of weighted $L^p$ Minkowski problems. Previously, the case of homogeneous measures was considered by Livshyts \cite{GAL19} and Wu \cite{DW17}. The class of measures with radially decreasing densities was considered by Kryvonos and Langharst \cite{KL23}. The Borel measure on $\R^n$ most pertinent to our investigations is the Gaussian measure
$$d\gamma_n(x)=\frac{1}{(2\pi)^{\frac{n}{2}}}  e^{-\frac{|x|^2}{2}} dx.$$
Huang, Xi, and Zhao \cite{HXZ21} first established existence for the even weighted Minkowski problem for the Gaussian measure. Due to the lack of homogeneity, the constant in the problem cannot be removed. This is made precise by the result from K. Ball \cite{Ball93}, who showed that $\gamma^+_n(\partial K) \leq 4 n^\frac{1}{4}$, illustrating the necessity of the constant. Feng, Liu and Xu \cite{FLX23} later removed the even assumption. Liu \cite{JL22} first considered the $L^p$ Gaussian Minkowski problem, $p > 0$, without any symmetry assumptions.  Recently, Feng, Hu and Xu \cite{FHX23} established existence for the even $L^p$ Gaussian Minkowski problem when $p\leq 0$. 

Like many fields of study, the Gaussian measure is a core measure in the theory, and the first to be characterized within the given framework. However, when a field matures, one must inevitably broaden one's horizons and consider other classes of measures. Liu and Tang \cite{LS24} took a step in this direction; motivated by \cite{LYZ05_2,LYZ07,LLYZ12}, they solved the weighted $L^p$ Minkowski problem for $p\in\mathbb{R}\setminus\{0\}$ with weights that were $s$-concave approximations of a Gaussian, i.e., they considered probability measures $\mu$ whose densities were proportional to $(1-s|x|^\alpha)_+^{\beta}$ for various $s,\alpha$ and $\beta$. Hu \cite{JH25} then considered the $p=0$ case. In this work, we study the weighted $L^p$ Minkowski problem for a large class of Borel measures on $\R^n$.
\begin{definition}
A Borel measure $\mu$ on $\R^n$ is said to be {\it rotationally invariant} if, for every Borel set $A\subset \R^n$ and $T\in O(n)$, one has $\mu(TA)=\mu(A)$.
\end{definition}
When $p<0$, we will require that the density of our measures satisfy the following growth regularity. For a set $E\subset \R^n$, we denote by $C^+(E)$ the set of nonnegative, continuous functions $E$.
\begin{definition}
    Fix $p\neq 0$. We say $\phi\in C^+(\R^n\setminus\{0\})$ has fast enough radial decay, or has property \textbf{(D)}$_p$, if $\phi\neq 0$ and, for every $\theta\in\sn$, one has
    $$\lim_{t\to\infty}t^{n-p}\phi(t\theta)=0.$$
\end{definition}
We note that every finite, log-concave measure satisfies Property~\textbf{(D)}$_p$ for all $p$. Thus, there are a plethora of examples.  Now that we have introduced the ambient measures we will consider, we list some technical definitions we need to state our results. Firstly, we set $\R_+=[0,\infty)$ and $$S^\mu_{p}(K):=S^\mu_{K,p}(\sn)$$
when $K\in\kno$. When $p=1$, this is precisely $\mu^+(\partial K).$ 

We also have special terminology used for the $p=0$ case. We recall that B\"{o}r\"{o}czky, Lutwak, Yang, and Zhang \cite{BLYZ13} defined the subspace concentration condition in the context of Minkowski problems; the phenomenon of subspace concentration also appeared in the work of Klartag \cite{BK10}.
\begin{definition}
    Let $\nu$ be a finite Borel measure on $\sn$. The measure $\nu$ is said to satisfy {\it the subspace concentration inequality} if for every subspace $\xi$ of $\R^n$ such that $0<\text{dim}\, \xi <n,$
    \begin{equation}
    \label{eq:sub}
    \nu(\xi \cap \sn) \leq \frac{1}{n}\nu(\sn)\text{dim}\, \xi.
    \end{equation}
    If \eqref{eq:sub} is always strict, then $\nu$ is said to satisfy the {\it strict subspace concentration inequality}.
\end{definition}

With this in hand, we list our first result, which is when the measure on the sphere is even. We emphasize beforehand that, in the $p\geq 0$ cases, $\mu$ has only locally integrable density, while, in the $p < 0$ case, $\mu$ is a finite Borel measure (i.e., $\mu(\R^n)<\infty$). We recall that a function $\psi\in L^1_{\operatorname{loc}}(E)$ if and only if for every bounded Borel set $A\subset E$, one has $\int_A \psi(x) dx <\infty$. Notice that, if $\psi$ is a function on $(0,\infty)$, the function $\psi\left(|\,\cdot\,|\right)$ is locally integrable on $\R^n$ if and only if $\int_{0}^{R}\psi(t)t^{n-1}dt<\infty$ for every $R>0$ from polar coordinate integration. 

\begin{theorem}
\label{t:first_sym}
Let $n\in \mathbb N$, let $p\in \R$ and let $\psi\in C^+(\R_+\setminus\{0\})$ satisfy
\begin{enumerate}
    \item if $p>0$: $\psi\in L^1_{\operatorname{loc}}(\R_+\setminus\{0\})$ and $\int_{0}^{R}\psi(t)t^{n-1}dt<\infty$ for every $R>0$;
    \item if $p=0$: $\psi\in  L^1_{\operatorname{loc}}(\R_+\setminus\{0\})$ and $\int_0^\infty\psi(t)^qt^{n-1}dt<\infty$ for some $q\in [1,\infty)$;
    \item if $p<0$: $\psi\in C^+(\R_+)\cap L^1(\R_+)$ is decreasing, has $\int_0^\infty\psi(t)t^{n-1}dt<\infty$, and satisfies Property \textbf{(D)}$_p$.
\end{enumerate}
Let $\mu$ be the rotationally invariant measure on $\R^n$ with density $\psi(|\cdot|)$. Fix $a \in (0,\mu(\R^n))$. Suppose $\nu$ is an even, finite Borel measure on the sphere which,
 \begin{enumerate}
     \item if $p>0$: is not concentrated on any great hemisphere;
     \item if $p=0$: satisfies the strict subspace concentration inequality;
     \item if $p<0$: is not concentrated on any great subsphere.
 \end{enumerate}
 Then, there exists a symmetric convex body $K=K(a)$ such that $\mu(K)=a$ and 
    $$\nu=\frac{\nu(\sn)}{S^\mu_p(K)}S^\mu_{K,p}.$$
Also, for $p\in \R\setminus\{0\}$, if the density of $\mu$ is positive on $\partial K$, then the converse direction holds as well.
\end{theorem}
Our next result removes the evenness assumption on $\nu$. To achieve this, we require that $\mu$ be finite for all $p$. For ease of presentation, we shall henceforth normalize, without loss of generality, all finite Borel measures on $\R^n$ to probability, i.e., we may assume that $\mu(\R^n)=1$ when $\mu$ is finite.
 
\begin{theorem}
    \label{t:first_finite}
     Let $p\in \R$, let $n\in\mathbb N,$ and let $\psi\in C^+(\R_+\setminus\{0\})$ be such that $\int_{0}^{\infty}\psi(t)t^{n-1}dt<\infty$. Furthermore, normalize $\psi$ so that the rotationally invariant measure $\mu$ on $\R^n$ with density $\psi\left(|\,\cdot\,|\right)$ is probability.
     \begin{enumerate}
         \item If $p\geq 0$, fix $a\in [\frac{1}{2},1)$.
         \item If $p<0$, suppose additionally that $\psi\in C^+(\R_+)$ is decreasing and satisfies Property \textbf{(D)}$_p$. Then, take $a\in (\frac{1}{2},1)$.
     \end{enumerate}
     Suppose $\nu$ is a finite Borel measure on the sphere which,
 \begin{enumerate}
     \item if $p>0$: is not concentrated on any great hemisphere;
     \item if $p=0$: satisfies the strict subspace concentration inequality;
     \item if $p<0$: is not concentrated on any great subsphere.
 \end{enumerate}
   Then, there exists $K=K(a)$ in $\kno$ such that $\mu(K)=a$ and 
    $$\nu=\frac{\nu(\sn)}{S^\mu_p(K)}S^\mu_{K,p}.$$ 
    If $\nu$ is even, then $K$ can be taken to be symmetric. For $p\in \R\setminus\{0\}$, if the density of $\mu$ is positive on $\partial K$, then the converse direction holds as well. 
\end{theorem}

Please see Section~\ref{sec:constants} (for $p>0$), Section~\ref{sec:p_zero} (for $p=0$) and Section~\ref{sec:neg_p} (for $p<0$) to find the proofs of each case of Theorems~\ref{t:first_sym} and \ref{t:first_finite}. Concerning Theorem~\ref{t:first_finite}, the proof of each case uses in a critical way that $\mu$ is finite. Thus, that theorem says nothing about non-finite, rotationally invariant measures (e.g., volume). This is not surprising; we essentially used $\mu$ being finite to circumvent any requirement that $\nu$ has its barycenter at the origin. On the other hand, the $p=0$ case of Theorem~\ref{t:first_sym} holds for the Lebesgue measure as well; at a key step, we use H\"older's inequality to control $\mu$-measures of convex bodies by their volumes, a step that is trivially unnecessary if $\mu$ is Lebesgue. However, for Lebesgue measure, assuming the \textit{strict} subspace concentration inequality is more restrictive than necessary. It is well known that for volume, when $\nu$ is even, the weaker \textit{subspace concentration condition} is necessary and sufficient \cite{BLYZ13}. Recall that $\nu$ is said to satisfy the subspace concentration condition precisely when equality occurring in \eqref{eq:sub} for a subspace $\xi$ implies that there exists a subspace $\xi^\prime$ complementary to $\xi$ so that equality occurs in \eqref{eq:sub} for $\xi^\prime$.

The reader may wonder why we restrict our attention to the strict subspace concentration inequality for general measures $\mu$, rather than pursuing this full condition. The reason is that the equality cases behave differently for non-Lebesgue measures. For example, the even Minkowski problem for the Gaussian cone volume measure, roughly corresponding to the $p=0$ case of Theorem~\ref{t:first_sym}, was recently completed by Hu \cite{JH24}. In that work, she constructed a family of bodies that yields asymptotic equality in the strict subspace concentration inequality when $\nu$ is the Gaussian cone volume measure. A similar phenomenon was observed in the so-called chord log-Minkowski problem \cite{LXYZ24}. For this reason, we did not pursue the full subspace concentration condition in this work, as it seems to be a volume-esque phenomenon.

\subsection{Existence Results in the small mass regime}
Intuitively, the Gaussian measure  behaves like the Lebesgue measure on convex bodies containing the origin with an asymptotically small measure. Huang, Xi and Zhao initiated a program on analyzing the constant-free, even Gaussian Minkowski problem by requiring $K$ to have such a small measure. Feng, Liu and Xu \cite{FLX23} removed the symmetry assumptions.  We will call this program the \textit{small mass regime}. We remark here that the small mass regime in the planar case for the Gaussian measure was thoroughly investigated by Chen, Hu, Liu and Zhao in the beautiful work \cite{CHLZ23}. Liu \cite{JL22} showed the existence of a solution to the constant-free, even, $L^p$ Gaussian Minkowski problem, $p\geq 1$ in the small mass regime. Feng, Hu and Xu \cite{FHX23} later removed the assumption of symmetry. Tang \cite{ST24} recently showed, when $1\leq p < n$, the existence of a solution to the constant-free, even, $L^p$ Gaussian Minkowski problem in the small mass regime, with a twist: the body $K$ that solves the problem satisfies $\gamma_n(K) <1/2$, opposite to the results mentioned above.

Our next set of results uses degree theory to work in the small mass regime. We will make use of the isoperimetric function $I_\mu$ from \eqref{eq:Iso_eq}. However, when $p\geq n$, the property \textbf{(D)}$_p$ is automatically implied by the other hypotheses on $\psi$, and we can actually avoid the use of the isoperimetric function. This is not surprising; the associated Monge-Amp\`ere equation becomes much more pliable to classical PDE techniques. 

Moreover, we will need to impose an additional assumption on $\mu$.  We recall that a function $\phi:\R^n\to \R_+$ is said to be radially decreasing if, for every $t\in (0,1)$ and $x\in \R^n$, $\phi(tx)\geq \phi(x)$.   

\begin{definition}
    We say a measure $\mu$ on $\R^n$ is {\it radially decreasing} if it has a locally integrable, radially decreasing density that is continuous on $\R^n\setminus\{0\}$.
\end{definition}
We remark that a special case of \cite[Theorem 1.2]{KL23} by Kryvonos and Langharst is the even, weighted $L^p$ Minkowski problem (with a constant) for all finite, radially decreasing measures on $\R^n$. The results mentioned in Theorems~\ref{t:first_sym} and \ref{t:first_finite} are more exhaustive. Even with radial decay, our results in the small mass regime for $p\in [1,n)$ require another technical property, which we call \textbf{(S)}$_p$. We postpone its definition to Section~\ref{sec:neg_p} (see Definition~\ref{def:Sp}). It would be unnatural to state our results at the level of this technical assumption. 

Instead, we introduce the following natural class of Borel measures on $\R^n$ which satisfy this property. In particular, we will be mirroring the recent development of the Gardner-Zvavitch inequality: as conjectured by Gardner and Zvavitch \cite{GZ10} (with the necessity of symmetry shown in \cite{PT13}) and resolved by Eskenazis and Moschidis \cite{EM21} (using the tools established in \cite{KL22}), $\gamma_n$ is $(1/n)-$concave over the class $\kne$ (see Section~\ref{sec:concavity} to recall the definition of concavity for measures). Not too long afterwards, this was extended to a larger set of Borel measures.  We define the following class of Borel measures on $\R^n$: 
\begin{equation}
\label{eq:CER}
\begin{split}
\mathcal{M}_n^\infty=\bigg\{\mu:\frac{d\mu}{dx}&=e^{-V(|x|)}, \;\;\text{where } V:\R_+\setminus\{0\}\to(-\infty,\infty] \mbox{ is increasing, satisfies, for every}
\\
&R>0,\,\int_{0}^{R}e^{-V(t)}t^{n-1}dt<\infty, \text{ and the function } t\mapsto V(e^t) \text{ is convex}.\bigg\}
\end{split}
\end{equation}
Cordero-Erausquin and Rotem \cite{CER23} extended the result by Eskenazis and Moschidis to every measure $\mu\in\mathcal{M}_n^\infty,$ i.e., every $\mu\in\mathcal{M}_n^\infty$ is $(1/n)$-concave over the class $\kne$. The class $\mathcal{M}_n^\infty$ then contains every rotationally invariant, log-concave measure. But it also contains more. For example, it contains the Cauchy-type measures given by $d\mu_{q,b}(x)=(1+|x|^b)^{-q}dx$ for $q,b \geq 0$. Another group of examples includes: if $\mu_0\in \mathcal{M}_n^\infty$, then so is also $d\mu_q(x)=|x|^{-q}d\mu_0(x)$ for $q\geq 0$. Here, $\mu_0$ can even be taken as the Lebesgue measure. Clearly, every $\mu\in\mathcal{M}_n^\infty$ is radially decreasing. More pertinently, we show in Proposition~\ref{p:uni_3} that Property \textbf{(S)}$_p$ holds for $\mu\in\mathcal{M}_n^\infty$.

Our previous results apply to the class $\mathcal{M}_n^\infty$. Indeed, the regime $p\geq 0$ of Theorem~\ref{t:first_sym} applies to all of $\mathcal{M}_n^\infty$. As for the results requiring finite measures, we define the following subclass of probability measures
\[
\mathcal{M}_n=\{\mu\in \mathcal{M}_n^\infty:\mu(\R^n)=1\}.
\]
Then, the $p\geq 0$ regimes of Theorem~\ref{t:first_finite} applies to $\mathcal{M}_n$. For those members of $\mathcal{M}_n$ whose densities satisfy Property~\textbf{(D)}$_p$, we may apply the $p<0$ case of Theorems~\ref{t:first_sym} and \ref{t:first_finite}. Note that we only need to explicitly impose Property~\textbf{(D)}$_p$ for measures in $\mathcal{M}^\infty_n$ that are not necessarily log-concave.  Consider, for example, the measure $\mu$ with density $\phi(x)=|x|^{-q}$ for some $q>0$ to be determined. Clearly, $\mu\in\mathcal{M}^\infty_n$. To satisfy condition \textbf{(D)}$_p$, we need $n-p < q.$ But, to be integrable near the origin, we need $q < n$. Thus, as long as $p$ and $q$ satisfy $p > n-q > 0$, $\mu$ will be a finite measure that satisfies \textbf{(D)}$_p$.

The following theorem is an application of the more general, but technical, Theorems~\ref{degree theory}, \ref{degree theory_2} and \ref{continuity method} and Lemma~\ref{l:lp_iso}. Here, $\kappa_n$ denotes the volume of the Euclidean unit ball in $\R^n$.

\begin{theorem}\label{t:small_mass_best}
		Fix $p \geq 1$. Let $\mu\in\mathcal{M}_n$ with density $\psi(|\cdot|)$ such that $\psi \in C^1(\R_+)$. Let $\nu$ be a finite Borel measure on $\sn$, not concentrated on any great hemisphere.
  \begin{enumerate}

  \item If $p>n$, then, there exists $K\in\kno$ such that $S^\mu_{K,p}=\nu$.
  \item If $p=n$, and if $\nu$ additionally satisfies $\nu(\sn)\leq \psi(0)n\kappa_n$, then there exists $K\in\kno$ such that $S^{\mu}_{K,n}=\nu$.
    \item If $\psi$ satisfies Property \textbf{(D)}$_p$ and if $\mu$ has an isoperimetric function $I_\mu$: then,  
  \begin{enumerate}
  \item  If $1\leq p<n$, then there exists $a>0$ depending only on $\mu$ sufficiently close to $1$, such that if $\nu$ is even and $$\nu(\sn) < \left(n\min\left\{aI_{\mu}(a)^\frac{p}{1-p}, \left(1-a\right)I_\mu(1-a)^\frac{p}{1-p}\right\}\right)^{1-p},$$
      then, there exist $K_1,K_2\in\kne$ such that $\mu(K_1)>a, \mu(K_2)<1-a,$ and
        $$S^\mu_{K_1,p} = \nu = S^\mu_{K_2,p}.$$
  \item 
        If $1\leq p<\infty$ and if $\nu$ satisfies $$\nu(\sn) <\left(\frac{n}{2}\right)^{1-p}I_\mu\left(\frac{1}{2}\right)^{p},$$ then there exists $K\in\kno$ such that $\mu(K)\geq \frac{1}{2}$ and $\nu = S^\mu_{K,p}.$
    \end{enumerate}
    \end{enumerate}
	\end{theorem}

\subsection{Concavity of Measures}
\label{sec:concavity}
For our uniqueness results to make sense, we must first discuss concavity of measures, in particular the so-called $F$-concave measures explored in \cite{GAL19, FLMZ23, FLMZ26,LMNZ17}. We say a Borel measure $\mu$ is $F$-concave, where $F$ is an invertible, (strictly) monotonic, continuous function, if there exists a collection $\mathcal{C}$ of Borel sets with finite $\mu$-measure such that, for every $K,L\in\mathcal{C}$ and every $0<\lambda <1$, one has 
 \begin{equation}
	\label{eq:concave}
	\mu\left((1-\lambda)K+\lambda L\right)\geq F^{-1}\left((1-\lambda) F(\mu(K)) +\lambda F(\mu(L))\right).
	\end{equation}
For a fixed $s\in[-\infty,\infty)$, a measure $\mu$ is $s$-concave over $\mathcal{C}$ if \eqref{eq:concave} holds for $F(x)=x^s$, i.e., for every $\lambda\in(0,1)$ and Borel sets $K,L\in\mathcal{C}$ one has
$$\mu((1-\lambda)K + \lambda L) \geq \left((1-\lambda)\mu(K)^s+\lambda \mu(L)^s\right)^\frac{1}{s}.$$
The case $s=0$ is log-concavity: 
$$\mu((1-\lambda)K + \lambda L) \geq \mu(K)^{1-\lambda}\mu(L)^\lambda.$$
When $s=-\infty$, this should be read as $\mu((1-\lambda)K+\lambda L) \geq \min\{\mu(K),\mu(L)\}$. A function $f:\R^n\to [0,\infty]$ is said to be $\kappa$-concave, $\kappa\in \R$, if, for every $\lambda \in (0,1),$
$$f((1-\lambda)x+\lambda y) \geq \left((1-\lambda)f(x)^\kappa + \lambda f(y)^\kappa\right)^\frac{1}{\kappa}$$
for every $x,y$ such that $0< f(x)f(y) < \infty$. The case $\kappa=0$ means $f$ is $\log$-concave, $\kappa=\infty$ means $f$ is constant.

In addition to the aforementioned result by Cordero-Erausquin and Rotem \cite{CER23}, showing that every $\mu\in\mathcal{M}_n^\infty$ is $(1/n)$-concave over $\kne$, Aishwarya and Rotem \cite{AR26} showed that, if $V$ is a convex, $q$-homogeneous function and $\mu$ is a Borel measure such that $d\mu(x)=Ce^{-V(x)}dx,$ for $C>0$, then $\mu$ is $\frac{q-1}{qn}$-concave over $\kno$. When $V(x)=\frac{1}{2}|x|^2$, then one obtains that $\gamma_n$ is $\frac{1}{2n}$-concave over $\kno$, which was previously shown by Kolesnikov and Livshyts \cite{KL21}.

Recall that a Borel measure is Radon if it is locally finite and (inner) regular. Borell's \cite{Bor73} classification states that a Radon measure $\mu$ is an $s$-concave measure (with $\mathcal{C}$ all Borel subsets of $\R^n$) if and only if it has density with respect to the Lebesgue measure that is $\kappa=s/(1-ns)$-concave (so, if $s=\frac{1}{n}$, the measure is a multiple of the Lebesgue measure, and, if $s > 1/n$, the density is zero a.e.). Henceforth, we say such measures are $s$-concave in the sense of Borell. For brevity, we may simply say $s$-concave (without referencing some class $\mathcal{C}$). 

Recall that a measure is said to be $\alpha$-homogeneous if, for a Borel set $A$, one has $\mu(tA)=t^\alpha\mu(A)$, $\alpha,t>0$, for all $t$ such that $tA$ is in the support of $\mu$. The Lebesgue measure, which we denote as $\vol_n$, is an $n$-homogeneous, $\frac{1}{n}$-concave measure. To illustrate the use of $s$-concave measures in Minkowski problems, Livshyts \cite{GAL19} considered and solved the constant-free, even Minkowski problem for a measure that is $\alpha$-homogeneous, $\frac{1}{\alpha}$-concave, in the sense of Borell, when $\alpha\geq n$; the result by Livshyts was extended to the $L^p$ case, $p\geq n,$ by Wu \cite{DW17}. Kryvonos and Langharst \cite{KL23} then did all $p\geq 1$.

For another example, the Gaussian measure $\gamma_n$ is $\log$-concave. In fact, $\gamma_n$ is concave in the sense of \eqref{eq:concave} with $F$ being the inverse function of $\gamma_1((-\infty,x))$, the so-called Ehrhard inequality \cite{EHR1,EHR2,Lat96,Bor03}; when $\mathcal{C}$ is taken to be $\conbod,$ then there is equality in Ehrhard's inequality if and only if the two bodies under consideration are identical (the Ehrhard inequality still holds with $\mathcal{C}$ taken to be all Borel sets, but then equality conditions are still open). Using the Ehrhard inequality in the case of convex bodies, and the additional assumption that $\gamma_n(K)\geq 1/2$, Huang, Xi and Zhao \cite{HXZ21} obtained uniqueness in the even Gaussian Minkowski problem (this is a special case of Proposition~\ref{p:uni_0_1} below).

\subsubsection{Uniqueness Results}

For our uniqueness results, we need to refer to Firey's $L^p$ summation \cite{Firey62}, denoted by $+_p$. Note that $p=1$ is the usual Minkowski sum. We say $\mu$ is $L^p$ $F$-concave if it satisfies \eqref{eq:concave} with Minkowski summation replaced by $L^p$ summation. Note that $L^p$ means, of the form $(1-\lambda)\cdot K+_p \lambda \cdot L$ for $\lambda \in (0,1)$ and $K,L\in\kno$, are increasing with respect to set-inclusion as $p$ increases. Thus, if $\mu$ is $F$-concave, then it is also $L^p$ $F$-concave (with the same $F$) for $p>1$.  If $\mu$ is shown to be $L^p$ $F$-concave from this method (i.e., if one already knows $\mu$ is $F$-concave, and then uses monotonicity of $L^p$ means to obtain $\mu$ is $L^p$ $F$-concave), then equality in the $L^p$ version of \eqref{eq:concave}, $p>1$, yields $K=L$, solely from equality forcing the $L^p$ mean to be independent of $p$. 

However, this does not use the $L^p$ summation in a ``smart" way. For example, Roysdon and Xing \cite{RX21} showed, extending on the volume case by Lutwak \cite{LE93}, that, if $\mu$ is $s$-concave (in the sense of Borell), $s>0$, then it is $(sp)$-concave with respect to $L^p$ summation for $p\geq 1$. As for equality conditions:  we recall that Dubuc \cite{SD77} showed, as elaborated on by Milman and Rotem \cite{MR14}, the following: let $\mu$ be $s$-concave (in the sense of Borell). If $K$ and $L$ are Borel sets such that, for some $\lambda\in (0,1)$ one has \begin{equation}
\label{eq:dubuc}
\mu((1-\lambda)K + \lambda L)^s = (1-\lambda)\mu(K)^s+\lambda \mu(L)^s,\end{equation}
then, $K=aL+b$ for some $a>0$ and $b\in \R^n$. In \eqref{eq:dubuc}, if one replaces $+$ with $+_p$, and $s$ with $ps$, then we lose the possibility of translation, i.e., equality holds in the $(sp)$-concavity of an $s$-concave measure with respect to $L^p$ summation if and only if the two bodies are dilates.

As far as we are aware, the method of obtaining $L^p$ $F$-concavity using monotonicity and the results by Roysdon-Xing and Lutwak for $s$-concave measures are the only examples of measures that are $L^p$ $F$-concave for $p>1$ (we save discussion of $p<1$ for the appendix). Thus, in our results below, the equality conditions for $p>1$ should be read as $K=tL$. Using the $L^p$ version of \eqref{eq:concave}, we recall the following from \cite[Proposition 6.16]{KL23}.

\begin{proposition}
\label{p:uni_0_1}
Fix $p \geq 1$. Let $\mu$, a Borel measure with continuous density on $\R^n\setminus\{0\}$, be $L^p$ $F$-concave over $\mathcal{C}\subset \kno$ such that $F$ is differentiable. Suppose $K,L\in\mathcal{C}$ are such that $S^{\mu}_{K,p}$ and $S^{\mu}_{L,p}$ are finite Borel measures on $\sn$ and $$S^{\mu}_{K,p}=S^{\mu}_{L,p}.$$ Then
$$\frac{F(\mu(L))-F(\mu(K))}{F^{\prime}(\mu(K))}\leq \frac{F(\mu(L))-F(\mu(K))}{F^{\prime}(\mu(L))}.$$
Furthermore, there is equality if and only if there is equality in the $L^p$ version of \eqref{eq:concave}.
\end{proposition}

As we have seen in our existence results, we were able to pin the measure of our convex bodies. Therefore, uniqueness follows as an immediate corollary of the above.
\begin{theorem}
\label{t:uni_0_1}
Fix $p \geq 1$. Let $\mu$, a Borel measure with continuous density on $\R^n\setminus\{0\}$, be $L^p$ $F$-concave over $\mathcal{C}\subset \kno$ such that $F$ is differentiable. Suppose $K,L\in\mathcal{C}$ are such that $\mu(K)=\mu(L)$, $S^{\mu}_{K,p}$ and $S^{\mu}_{L,p}$ are finite Borel measures on $\sn$ and $S^{\mu}_{K,p}=S^{\mu}_{L,p}.$ Then, there is equality in the $L^p$ version of \eqref{eq:concave}.
\end{theorem}

Let us return to the case of $s$-concave measures from \eqref{eq:dubuc}, and let $s>0$. If $\mu$ is $s$-concave, then it is also $L^p$ $s$-concave (due to monotonicity of $L^p$ means) and $L^p$ $(sp)$-concave (Roysdon-Xing). Suppose that we have equality in either of these $L^p$ concavities. Then, since we also know that $\mu(K)=\mu(L)$, $K=L$. Consequently, ``equality in \eqref{eq:concave}" in Theorem~\ref{t:uni_0_1}, in the case of $s$-concave measures for $p>1$, should be read as $K=L$. For $p=1$, one needs extra information to remove the translation from the equality conditions of \eqref{eq:dubuc} first, for instance, by assuming that both bodies are symmetric. We stated Proposition~\ref{p:uni_0_1} and Theorem~\ref{t:uni_0_1} in terms of a class $\mathcal{C}\subset \kno$ to allow us to consider, for example, $\mathcal{C}=\kne$. Thus, we can also apply Theorem~\ref{t:uni_0_1} to the $L^p$ $(1/n)$-concavity of $\mu\in\mathcal{M}_n^\infty$ over $\kne$ for $p\geq 1$. If $p>1$, then we must have $K=L$. However, for $p=1$, outside the case of rotationally invariant measures that are also log-concave (where \eqref{eq:dubuc} and $\mu(K)=\mu(L)$ tell us that $K=L$), the equality conditions of \eqref{eq:concave} in this case are apparently still open.

This paper is organized as follows. In Section~\ref{sec:surface_area}, we give many formulas for variations of surface area measure needed in this work (among other preliminary facts). In Section~\ref{sec:constants}, we prove Theorem~\ref{t:p>n} and the $p>0$ cases of Theorem \ref{t:first_sym} and Theorem~\ref{t:first_finite}, which are results concerning positive ranges for $p$. In Section~\ref{sec:p_zero}, we prove the $p=0$ cases of Theorems~\ref{t:first_sym} and \ref{t:first_finite}, which are the weighted log-Minkowski problems. In Section~\ref{sec:neg_p}, we prove the $p<0$ cases of Theorems~\ref{t:first_sym} and \ref{t:first_finite}, which concern negative $p$. For the constant-free results in the small mass regime, we prove Theorem~\ref{degree theory} in Section~\ref{sec:small_mass_small_p}, and Theorem~\ref{continuity method} in Section~\ref{sec:small_mass_large_p}. Finally, in Appendix \ref{sec:iso}, we list examples of isoperimetric functions, concluding on a conjecture that would yield a sharp isoperimetric inequality for measures in $\mathcal{M}_n^\infty$.

\section{Preliminaries}
\label{sec:surface_area}
In this work, we will be using the radial function of a convex body $K\in\kno$, which is the continuous function on $\R^n\setminus\{0\}$ given by $\rho_K(u)=\sup\{r>0:ru\in K\}$. Using radial functions, we have the volume formula
\begin{equation}
    \vol_n(K)=\frac{1}{n}\int_{\sn}\rho_K(u)^ndu,
    \label{eq:radial}
\end{equation}
where the integration is with respect to the spherical Lebesgue measure (i.e., $du=dS_{B_2^n}$).

We would now like to mention that in \cite{GHX19,GHXY20,LSYY22}, the first two of which pre-date the Gaussian Minkowski problem, Ye and his various collaborators studied the so-called generalized volume setting. The framework introduces the general volume function $G$ on $(0,\infty)\times \sn$. From this, a measure on $\partial K$ for $K\in\kno$ is introduced as a type of pushforward of $G(\rho_K(u),u)$. This framework is very broad; indeed, set \begin{equation}
\label{eq:general_volume}
G(t,u)=\int_{0}^t\phi(ru)r^{n-1}dr\end{equation} for $\phi$ density of a measure $\mu$ on $\rn$ to obtain $S^\mu_K$. They found necessary and sufficient conditions for existence in the associated Minkowski problem with a constant under various assumptions on $\partial_t G$. In particular, their results imply the $L^p$ Gaussian Minkowski problem with a constant for all $p>0$, in the strict sense of the definition of Minkowski problem. Crucially, however, the approach, when translated to $S^\mu_K$ via \eqref{eq:general_volume}, does not allow control of $\mu(K)$. The ability to do so, as first done by Huang, Xi and Zhao, is vital in Gaussian Minkowski problems, and is an important extra ingredient that we strove to keep. Consequently, the results for the Gaussian measure implied by \cite{GHX19,GHXY20,LSYY22} are very similar, but ultimately disjoint, from the majority of the mentioned results on Gaussian Minkowski problems with a constant, and, more pertinently, disjoint from the results we present herein. In addition to being able to control $\mu(K)$, the class of rotationally invariant measures we consider allows us to drop the assumptions on $\phi$ that one would obtain using the aforementioned generalized volume results.

Fix an arbitrary $K\in\kno$. We say a convex body is strictly convex if its boundary does not contain a line-segment. The subgradient of $h_K$ is precisely the set function given by
$$\partial h_K(u)=\{z:\R^n: h_K(y) \geq h_K(u) + \langle y, z-u \rangle \; \forall y\in K\}.$$
Denoting $\nabla$ the usual gradient operator, one has that $\partial h_K(u)$ is a singleton, namely $\nabla h_K(u)$, if and only if $h_K$ is differentiable at $u$.
In general, for $u\in\sn$,
\begin{equation}\partial h_K(u) = F(K,u) = \{y\in K: h_K(u)=\langle u,y \rangle \}
\label{eq:faces}
\end{equation}
where $F(K,u)$ is the face of $K$ with outer-unit normal $u$. For $x\in \R^n\setminus\{0\}$, one has
$$\partial h_K(x) = F\left(K,\frac{x}{|x|}\right).$$
The Gauss map and the support function are related: $n_K$ is invertible at $u\in\sn$ if and only if $h_K$ is differentiable at $u$, in which case $n^{-1}_K(u) = \nabla h_K(u)$ \cite[Corollary 1.7.3]{Sh1}. Hence, $K$ is strictly convex if and only if $h_K\in C^1$ \cite[Page 115]{Sh1}.

The surface area measure is closely related to the so-called Monge-Amp\`ere measure: given a convex function $h$ defined on an open, $d$-dimensional convex set $\Omega$ (equipped with the $d$-dimensional Hausdorff measure), its Monge-Amp\`ere measure is precisely $\mu_{h}(E)=\mathcal{H}^d\left(N_{h}(E)\right),$ where
$$N_{h}(E)=\bigcup_{x \in E} \partial h(x)$$
for a Borel subset $E\subset \Omega$. Here, $\partial h$ is the subgradient of $h$; since we will only use this when $h$ is the support function of a convex body, we do not define the subgradient of an arbitrary convex function. If $h$ is $C^2$, then one obtains the following integral representation:
\begin{equation}
    \label{eq:mae}
    \mu_{h}(E) = \int_{E}\det\hess{h(x)}d\mathcal{H}^{d}(x),
\end{equation}
where \text{Hess} denotes the Hessian on $\R^{d}$. As an example, setting $\Omega=\s^{n-1}$, the surface area measure is then the Monge-Amp\`ere measure of the support function: 
\begin{equation}
\label{eq:sur_mae}
S_K(E)=\mathcal{H}^{n-1}\left(\bigcup_{u \in E} F(K, v)\right)=\mathcal{H}^{n-1}\left(\bigcup_{u \in E} \partial h_K(u)\right)=\mu_{h_K}(E).\end{equation}

If the boundary of $K$ is a $C^2$ manifold with positive Gauss curvature everywhere, then we say that $K$ is $C^2_+$. If $K$ is $C^2_+$, then $h_K\in C^2(\sn)$ \cite[Bottom of Page 115]{Sh1} and $S_K$ is absolutely continuous with respect to the spherical Lebesgue measure: $dS_K(u) = \det(\nabla^2 h_{K}+h_{K}I) du $. Here, $I$ is the $(n-1)\times (n-1)$ identity matrix and $\nabla^2$ is the spherical Hessian. We denote by $\nabla_s$ the spherical Gradient. 
 
Since $K\in\kno$, $\rho_K(\theta)$ is continuous, and so there exists some $\theta_K\in\sn$ such that $\rho_K(\theta_K)$ is maximal. One has that the line segment $[-\rho_K(\theta_K)\theta_K,\rho_K(\theta_K)\theta_K]$ is completely contained in $K$, and yet $K$ is contained in the ball of radius $\rho_K(\theta_K).$
    Let $g(t)=|t|$ if $K$ is symmetric and $g(t)=t_+=\max\{t,0\}$ otherwise. From convexity, one has 
    \begin{equation}\label{support function inequality2}
		h_{K}(u)\geq \rho_{K}(\theta_K)g(\langle\theta_K, u\rangle)
	\end{equation}
    for all $u\in\sn$. Note that, at the point $\theta_K$, $\rho_K(\theta_K)=h(\theta_K)$. Recalling that $h_{K}$ is $1$-homogeneous, we have (by differentiating $h(tu)$ at $t=1$) that, for every $u\in \sn$ such that $h_K$ is differentiable at $u$,
 \begin{equation} 
 h_{K}(u)=\langle \nabla h_{K}(u),u \rangle.
 \label{supp_homo}
 \end{equation}
 Consequently, by the Cauchy-Schwarz inequality, $h_{K}(u)\leq |\nabla h_{K}(u)|$. These two estimates yield \begin{equation}
 \label{eq:supp_max_deriv_com}
 h_{K}(\theta_K)g
 (\langle u,\theta_K\rangle) \leq |\nabla h_{K}(u)|\end{equation} for $K\in\kno$.

It is also well-known that for every $u\in\sn$, there exists a $v$ such that $\rho_{K}(v)v=h_{K}(u)u+\nabla_s h_{K}(u)$. This implies that
\begin{equation}
\label{eq:radial_support}
\rho_{K}^2(v)=h^2_{K}(u)+|\nabla_s h_{K}(u)|^2.\end{equation}
Recall also the fact that, when $K$ is $C^2_+$,
 \begin{equation}
 \label{eq:monge}
 \operatorname{det}\left(\nabla^2 h_{K}(u)+h_{K}(u) I\right)=\frac{\left(h_{K}(u)^2+\left|\nabla_s h_{K}(u)\right|^2\right)^{\frac{n}{2}}}{h_{K}(u)}=\frac{\rho_K(v)^n}{h_K(u)},\end{equation}
 where $v$ and $u$ are related via \eqref{eq:radial_support}.
 Notice then that \eqref{eq:monge} yields
 \begin{equation}
 \label{eq:monge_2}
h_K(u)\operatorname{det}\left(\nabla^2 h_{K}(u)+h_{K}(u) I\right) \leq \max_{v\in\sn}\rho_K(u)^n = \max_{u\in\sn}h_K(u)^n.\end{equation}

For every positive $f\in C(\sn)$, the \textit{Wulff shape} of $f$ is the convex body given by
	\begin{equation}
	    [f]=\{x\in\R^n:\langle x,u\rangle\leq f(u) \; \forall u \; \in \sn\}.
	\end{equation}
	One has that, for $K\in\kno$, $[h_K]=K.$ Since $f$ is positive, $[f]\in\kno$. Furthermore, if $f$ is even, then $[f]$ is symmetric. Next, for $f\in C(\sn)$, Aleksandrov \cite{AL} defined a perturbation of $K\in\conbod$ to be the Wulff shape of the function
	\begin{equation}h_t(u)=h_K(u)+tf(u),
 \label{eq:perturb}
 \end{equation}
	where $t\in (-\delta,\delta)$, $\delta$ small enough so that $h_t$ is positive for all $u$.
 
 It was shown \cite{KL23,HXZ21} that, for almost all $u\in\s^{n-1}$ up to a set of spherical Lebesgue measure zero,
	\begin{equation}
 \label{eq:radial_variation}
 \diff{\rho_{[h_t]}(u)}{t}\bigg|_{t=0}=\lim_{t\to 0}\frac{\rho_{[h_t]}(u)-\rho_{K}(u)}{t}=\frac{f(n_{K}(r_{K}(u)))}{h_K(n_K(r_K(u)))}\rho_K(u).\end{equation} Here, $r_K(u)=\rho_K(u)u$ is the radial map, which is defined almost everywhere. In fact, \eqref{eq:radial_variation} was proven using tools from \cite{HLYZ16}. Let us elaborate:
 In \cite{HLYZ16}, they introduced the following concept: for $h,f\in C(\s^{n-1})$ and some small $\delta$, define a function $h_t:\s^{n-1}\to (0,\infty)$ via
	\begin{equation}
	    \log(h_t(u))=\log(h(u)) +tf(u) + o(t,u)
	    \label{eq:log_fam}
	\end{equation}
	where $o(t,u)/t \rightarrow 0$ as $t\to 0$ for all $u\in\s^{n-1}.$ Then, the Wulff shapes $[h_t]$ are said to be the logarithmic family of Wulff shapes formed by the pair $(h,f).$ One can readily verify that $h_t$ defined by \eqref{eq:perturb} is the logarithmic family of Wulff shapes formed by the pair $(h_K,\frac{f}{h_K})$; then, \eqref{eq:radial_variation} follows from \cite[Lemma 4.3]{HLYZ16}. However, there was actually no need to use this specific family. The same proof yields the following.
 \begin{proposition}
 \label{p:radial_new}
     Let $K\in\kno$ and $f\in C(\s^{n-1})$ be such that $(h_K,f)$ is a logarithmic family. Then, defining $h_t$ via \eqref{eq:log_fam}, we have
     \begin{equation}
 \label{eq:radial_variation_2}
 \diff{\rho_{[h_t]}(u)}{t}\bigg|_{t=0}=\lim_{t\to 0}\frac{\rho_{[h_t]}(u)-\rho_{K}(u)}{t}=f(n_{K}(r_{K}(u)))\rho_K(u).\end{equation}
 In fact, there exist constants $M,\delta>0$ so that, for almost every $u\in\sn$ and every $t\in (-\delta,\delta),$ $$|\rho_{[h_t]}(u)-\rho_{K}(u)| \leq M|t|.$$
 \end{proposition}

 The following result was shown by Langharst and Kryvonos \cite{KL23}, extending on the partial cases by Hosle-Kolesnikov-Livshyts \cite{GAL19,HKL21}, Gardner-Hug-Weil-Xing-Ye \cite{GHX19}, Kolesnikov-Milman \cite{KM18}, Bobkov \cite{BS07} and Aleksandrov \cite{AL}. 
 \begin{proposition}
 \label{p:var}
 Let $\mu$ be a Borel measure on $\R^n$ with locally integrable density $\phi$. Let $K$ be a convex body, such that $\partial K$, up to a set of $(n-1)$-dimensional Hausdorff measure zero, is in the Lebesgue set of $\phi$. Then, for $f\in C(\sn)$, one has that
	\begin{equation}
 \label{eq:var}
 \lim_{t\rightarrow 0}\frac{\mu([h_K+t f])-\mu(K)}{t}=\int_{\sn}f(u)dS^{\mu}_{K}(u).\end{equation}
 \end{proposition}
 The assumption that $\mu$ has locally integrable density can be dropped, i.e., $\mu$ can have singular components, as long as there exists a $\delta$-neighborhood of $\partial K$ not intersecting the region of $\R^n$ assigned singular mass. When $f=h_L$, where $L$ is a compact, convex set, one has $[h_K+\epsilon h_L]=K+\epsilon L$ if $K\in\kno$. 

 Recently, \eqref{eq:var} was re-proven in \cite{FYZZ25} in the case when $K\in\kno$ and $\mu$ has continuous density $\omega$; we now elaborate. Recall that $S^\mu_{K,0}$ is the weighted cone measure, i.e., $dS^\mu_{K,0}=h_KdS^\mu_K$.
 Then, \eqref{eq:var} can be written as, when $K\in\kno$,
 \[\lim_{t\rightarrow 0}\frac{\mu([h_K+t f])-\mu(K)}{t}=\int_{\sn}\frac{f(u)}{h_K(u)}dS^{\mu}_{K,0}(u),\]
 which is what appears in \cite[Lemma 2.1]{FYZZ25} (note in that work $S^{\mu}_{K,0}$ is called the dual $\omega$-Orlicz moment of $K$). The choice to use $\frac{1}{h_K(u)}dS^{\mu}_{K,0}(u)$, and not simply $S^{\mu}_{K}$, forces one to have the origin in the interior of $K$. 
 
  By using $S^\mu_K$ in \eqref{eq:var}, one can merely shift both the measure and the convex body $K$ to remove any necessity for $K$ to contain the origin; that is, the case for a general convex body $K$ actually follows from the case containing the origin. Indeed, for a convex body $K$, let $K-\int_{K}xdx = K^\prime \in \kno.$ For a Borel measure $\mu$ with density $\phi$ whose Lebesgue set contains $\partial K$, let $\mu^\prime$ be the Borel measure with density $\phi^\prime(y) = \phi(y+\int_K x dx)$. Then, notice that $S^{\mu^\prime}_{K^\prime}=S^\mu_K$, $\mu(K)=\mu^\prime(K^\prime)$, and, for $t$ small enough, $[h_{K^\prime}+tf]=[h_K+tf]-\int_K x dx$ implies $\mu^\prime([h_{K^\prime}+tf])=\mu([h_K+tf]).$ 

 Schneider \cite{RS24} recently gave a new, geometric proof of Proposition~\ref{p:var} for when $\mu$ has continuous density. This elegant proof is in contrast to the approach taken in \cite{KL23}, which used the radial variations \eqref{eq:radial_variation} to prove \eqref{eq:var}. However, one advantage of using \eqref{eq:radial_variation} is that, if one uses Proposition~\ref{p:radial_new}, the exact same proof yields the following.
 \begin{theorem}
 \label{t:better_var}
     Let $\mu$ be a Borel measure on $\R^n$ with locally integrable density $\phi$. Let $K\in\kno$ be such that $\partial K$, up to a set of $(n-1)$-dimensional Hausdorff measure zero, is in the Lebesgue set of $\phi$.
     Suppose $f\in C(\sn)$ is such that $(h_K,f)$ is a logarithmic family, and define $h_t$ via \eqref{eq:log_fam}.
     
      Then,  one has
	\begin{equation}
 \label{eq:var_2}
 \lim_{t\rightarrow 0}\frac{\mu([h_t])-\mu(K)}{t}=\int_{\sn}f(u)dS^{\mu}_{K,0}(u).\end{equation}
 \end{theorem}
 Notice that, by taking the Taylor series expansion of $e^x$, a perturbation of $h_K$ of the form $h_Ke^{tf}$ satisfies the hypothesis of Theorem~\ref{t:better_var}. Thus, we obtain the following corollary.
 \begin{corollary}
 \label{cor:deriv_2}
      Let $\mu$ be a Borel measure on $\R^n$ with locally integrable density $\phi$. Let $K\in\kno$ be such that $\partial K$, up to a set of $(n-1)$-dimensional Hausdorff measure zero, is in the Lebesgue set of $\phi$.
     Suppose $f\in C(\sn)$. Then,
      $$\lim_{t\rightarrow 0}\frac{\mu([h_Ke^{tf}])-\mu(K)}{t}=\int_{\sn}f(u)dS^{\mu}_{K,0}(u).$$
 \end{corollary}
 \noindent This extends on the volume case from \cite{BLYZ12} and the Gaussian case from \cite{FHX23,JH24}. 
 
 The $L^p$ surface area was also given a weighted analogue in \cite{KL23}.  The associated variational formula is the same as above, except $[h_K+th_L]$ is replaced with $[(h_K^{p}+th_L)^{1/p}]$ and $dS^\mu_K$ is replaced with $\frac{1}{p}dS^\mu_{K,p}$; here, of course, $K$ must contain the origin. This extends on partial cases by Lutwak \cite{LE93}, Wu \cite{DW17}, and Liu \cite{JL22}.  Note that in \cite{KL23} it is assumed that $p\geq 1$, but this is not used in the proof. 
 \begin{proposition}
 \label{p:deriv}
      Let $p\neq 0$, let $K$ be a convex body in $\R^n$, and let $\mu$ be a Borel measure on $\R^n$ with a density whose Lebesgue set contains $\partial K$. Assume that $S^\mu_{K,p}$ is a finite Borel measure on $\sn$. Then,
      $$\lim_{\epsilon\rightarrow 0}\frac{\mu([(h_K^p+\epsilon f)^\frac{1}{p}])-\mu(K)}{\epsilon}=\frac{1}{p}\int_{\sn}f(u)dS^{\mu}_{K,p}(u).$$
 \end{proposition}

 Let $\mu$ be a Borel measure on $\R^n$. Then, Livshyts \cite{GAL19} introduced the \textit{mixed measure} of Borel sets $K$ and $L$ as 
\begin{equation}\mu(K;L):=\liminf_{\epsilon \to 0}\frac{\mu(K+\epsilon L)-\mu(K)}{\epsilon}.
\label{eq:mixed_measures}
\end{equation}
In the same work, it was shown that, if $\mu$ has continuous density and $K$ and $L$ are convex sets, then the liminf is a limit and
\begin{equation}\mu(K;L)=\int_{\sn}h_L(u)dS^\mu_K(u).
\label{eq:mixed_form}
\end{equation}
In \cite{KL23}, the assumption that $\mu$ has continuous density was weakened to assuming that $\mu$ with a density whose Lebesgue set contains $\partial K$. When $\mu=\gamma_n$, the quantity $\gamma_n(K;L)$ would reappear in \cite{HXZ21}. Notice that $\mu(K;B_2^n)=\mu^+(\partial K)$. Mixed measures were later systematically studied in \cite{FLMZ23,FLMZ26}.

We mention now that Firey's $L^p$ summation is precisely, for $a,b \geq 0$ and $K,L\in\kno,$ $$a\cdot K+_p b\cdot L=[(ah_K^{p}+bh_L^{p})^\frac{1}{p}].$$ The $L^p$ version of mixed measures, where summation was replaced by $L^p$ summation, was introduced in \cite{KL23}: under the same assumptions on $K,L$ and $\mu$ in \eqref{eq:mixed_form}, one has
\begin{equation}
\label{eq:Lp_mixed_measures}
\mu_p(K;L):=\lim_{\epsilon \to 0}\frac{\mu(K+_p\epsilon \cdot L)-\mu(K)}{\epsilon}=\frac{1}{p}\int_{\sn}h_L(u)^pdS^\mu_{K,p}(u),
\end{equation}
whenever $S^\mu_{K,p}$ is a finite Borel measure on the sphere. This extends on the Gaussian case from Liu \cite{JL22} and the volume case by Lutwak \cite{LE91}, Lutwak, Yang, and Zhang \cite{LYZ12}, and B\"{o}r\"{o}czky, Lutwak, Yang, and Zhang \cite{BLYZ13}. Notice that $\mu_p(K;K)=\frac{1}{p}\mu(K;K).$ While it is true that
\begin{equation}
\label{eq:volume}
\vol_n(K)=\frac{1}{n}\int_{\sn}h_K(u)dS_K(u),\end{equation}
this does not hold for measures: $n\mu(K)\neq \mu(K;K)$. The following proposition shown in \cite[Proposition 2.1]{FLMZ23}, relates them.
\begin{proposition}
\label{p:meas_relate}
    Let $\mu$ be a radially decreasing measure on $\R^n$. Then, for every convex body $K\in\kno$ such that the density of $\mu$ is defined on $\partial K$, one has
    $$n\mu(K) \geq \mu(K;K),$$
    with equality if and only if for almost every $y\in \partial K,$ the density of $\mu$ is a constant almost everywhere on $(0,y].$
\end{proposition}

 We conclude this section by obtaining an isoperimetric inequality for $S^\mu_{p}(K)$ when $K\in\kno$ and $\mu$ is radially decreasing. To do this, we need another definition. \begin{definition}
\label{def: Lp_iso}
    Fix $p >0$, and let $\mu$ be a Borel measure on $\R^n$. We say $\mu$ has an $L^p$ isoperimetric function over $\kno$ (or $\kne$) if there exists a nonnegative function $I_{p}$ such that, for $K\in\kno$ (or $\kne$),
    $$S^\mu_p(K) \geq I_p(\mu(K)).$$
\end{definition}
We discuss in Appendix~\ref{sec:iso} existence of $I_p$. There is a plethora of examples, such as log-concave measures.  The result of the following lemma establishes the existence of $I_p$ when $p>1$ for radially decreasing measures $\mu$ that have an isoperimetric function $I_\mu$.

\begin{lemma}
\label{l:lp_iso}
    Let $K\in\kno$ and let $p\geq 1$. Let $\mu$ be a radially decreasing measure on $\R^n$ and $I$ a function satisfying \eqref{eq:Iso_eq} for $\mu$ and $K$. Then, 
    $$S^\mu_{p}(K) \geq (n\mu(K))^{1-p}I(\mu(K))^p.$$
\end{lemma}
\begin{proof}
    By applying Jensen's inequality to the probability measure $\frac{h_K(u)dS^\mu_K(u)}{\mu(K;K)}$ on the sphere, we obtain
    \begin{align*}
        \left(\frac{S^\mu_p(K)}{\mu(K;K)}\right)^{\frac{1}{p}}&=\left(\int_{\sn}h_K(u)^{-p}\frac{h_K(u)dS^\mu_K(u)}{\mu(K;K)}\right)^{\frac{1}{p}}
        \\
        &
        \geq \int_{\sn}h_K(u)^{-1}\frac{h_K(u)dS^\mu_K(u)}{\mu(K;K)}=\frac{\mu^+(\partial K)}{\mu(K;K)}.
    \end{align*}
    Upon re-arrangement, this becomes
    $$S^\mu_p(K)^\frac{1}{p}\geq \mu(K;K)^{\frac{1-p}{p}} \mu^+(\partial K).$$
    Applying Proposition~\ref{p:meas_relate}, we obtain
    $$S^\mu_p(K)^\frac{1}{p}\geq (n\mu(K))^{\frac{1-p}{p}}\mu^+(\partial K).$$
    Using that $\mu^+(\partial K)\geq I(\mu(K))$ by hypothesis and raising both sides to the $p$th power yields the result.
\end{proof}
We note that the case $\mu=\gamma_n$ (and $I=I_{\gamma_n}$) of Lemma~\ref{l:lp_iso} was previously done in \cite{FHX23}.

\section{The case of positive p}
\label{sec:constants}
In this section, we will prove the $p>0$ cases of Theorems~\ref{t:first_sym} and \ref{t:first_finite}. At the end of the section, we also establish another result for when $p>n$, see Theorem~\ref{t:p>n} below. We will utilize the variational technique; the use of this technique can be traced through \cite{HW24,FLX23,FHX23,KL23,HXZ21,GHXY20,HLYZ10,CW06,LYZ06,LYZ04,LE93}. We denote by $C(\sn)$ the set of continuous functions on the sphere. If there is a ``$+$" superscript, then the functions are additionally nonnegative; if there is a ``$e$" subscript, then the functions are also even.

We will work with the following two functionals: the first
\begin{equation}
    \psi_{\nu}(f):=\mu\left([f]\right)-\frac{1}{p}\int_{\sn}f^p(u)d\nu(u); \quad \psi_\nu(K):=\psi_\nu(h_K),
    \label{eq:functional}
\end{equation}
will be taken over $C^+_e(\sn)$, and the second
\begin{equation}
    \Omega_{\nu}(f):=-\frac{1}{p}\int_{\sn}f^p(u)d\nu(u); \quad \Omega_\nu(K):=\Omega_\nu(h_K),
    \label{eq:functional_2}
\end{equation}
can be taken over $C^+(\s^{n-1})$ or $C^+_e(\sn)$. 

We break the proof into four steps:
\begin{enumerate}
    \item First, show that, for $p\neq 0$, any maximizer of \eqref{eq:functional} or \eqref{eq:functional_2} in $C^+_e(\sn)$ is the support function of a symmetric convex body, and any maximizer of \eqref{eq:functional_2} in $C^+(\sn)$ is the support function of a convex body containing the origin in its interior.
    \item Second, show that, when $p\neq 0$, a maximizer $K$ of \eqref{eq:functional} satisfies $\nu=S^\mu_{K,p}$, and a maximizer $K$ of \eqref{eq:functional_2} satisfies $\nu=\frac{\nu(\sn)}{S^\mu_p(K)}S^\mu_{K,p}$.
    \item Third, establish the existence of a maximizer when the measure of the body is pinned and $p$ is positive, that is, to prove the $p>0$ cases of Theorems~\ref{t:first_sym} and \ref{t:first_finite}.
    \item Finally, establish the existence of a maximizer when $p> n$ and there are assumptions of symmetry and Property \textbf{(D)}$_p$, that is, to prove Theorem~\ref{t:p>n}.
\end{enumerate}

\begin{lemma}
\label{l:op_relate}
    Let $\mu$ be a Borel measure on $\R^n$, and fix a Borel measure $\nu$ on $\sn$. Let $K\in\kno$. Then, for a fixed $p\in\R,$ $p\neq 0$:
    \begin{enumerate}
        \item $K$ solves 
        \begin{equation} \sup\{\psi_{\nu}(K):K\in\kne\}
        \label{eq:max_p>n}
        \end{equation} if and only if $h_K$ solves $$\sup\{\psi_{\nu}(f):f\in C_e^+(\sn)\}.$$
        \item For a fixed $a\in (0,\mu(\R^n))$, $K$ solves 
        \begin{equation}
        \label{eq:max_p}
        \sup\{\Omega_{\nu}(K):\mu(K)=a,K\in\kno\}\end{equation} if and only if $h_K$ solves $$\sup\{\Omega_{\nu}(f):\mu([f])=a,f\in C^+(\sn)\}.$$
    \end{enumerate}
     Additionally, in \eqref{eq:max_p}, if $K$ is also assumed to be symmetric in the first optimization problem, then the set $C^+(\sn)$ is replaced by $C_e^+(\sn)$ in the second optimization problem.
\end{lemma}
\begin{proof}

Recall for $f\in C^+(\sn)$, $h_{[f]} (u) \leq f(u)$ point-wise and $[h_{[f]}]=[f]$. Suppose $p>0$. Then, $h_{[f]}^p \leq f^p$ but $-\frac{1}{p} h_{[f]} ^p \geq -\frac{1}{p} f^p$. If $p<0$, then, we instead have $h_{[f]}^p \geq f^p$, but still obtain $-\frac{1}{p} h_{[f]} ^p \geq -\frac{1}{p} f^p$. Therefore, for all $p\neq 0$,
\begin{align*}
    \Omega_{\nu}([f])&= -\frac{1}{p}\int_{\sn}h^p_{[f]}(u)d\nu(u)
    \\
    &\geq -\frac{1}{p}\int_{\sn}f^p(u)d\nu(u)=\Omega_{\nu}(f).
\end{align*}
Thus, a maximizer of $\Omega_\nu$ is the support function of a symmetric convex body. Since one always has $\mu([f])=\mu([h_{[f]}])$, the above also shows $\psi_\nu([f]) \geq \psi_\nu(f) $. This establishes both equivalences.
\end{proof}

\begin{lemma}
\label{l:op_solve}
    Let $K\in\kno$, and let $\mu$ be a Borel measure on $\R^n$ with a density whose Lebesgue set contains $\partial K$. Then, for a fixed $p \in \R, p\neq 0$:
    \begin{enumerate}
        \item If $K$ solves \eqref{eq:max_p>n}, then $\nu=S^{\mu}_{K,p}$.
        \item If $K$ solves \eqref{eq:max_p}, then $\nu=\frac{\nu(\sn)}{S^\mu_p(K)}S^{\mu}_{K,p}.$
    \end{enumerate}
\end{lemma}

\begin{proof}
First suppose that $K$ solves \eqref{eq:max_p>n}. Then, we perturb $K$ by $f\in C^+_e(\sn)$: let $$h_t=(h_K^p+tf)^\frac1p,$$ where $t\in (-\delta,\delta)$, $\delta$ chosen so that $h_t$ is strictly positive on $\sn$. Then, by Proposition~\ref{p:deriv}, we obtain
\begin{align*}
    0 = \frac{d\psi_\nu(h_t)}{dt}\Big|_{t=0} =\frac{1}{p}\int_{\sn}f(u)dS^{\mu}_{K,p}(u) - \frac{1}{p}\int_{\sn}f(u)d\nu(u).
\end{align*}
Since this is true for every $f\in C^+_e(\sn)$, we have by the Riesz representation theorem that $\nu=S^\mu_{K,p}$, thus establishing the first claim.

Next, suppose that $K$ solves \eqref{eq:max_p}. Since perturbing $K$ causes us to break the condition that $\mu(K)=a$, we must modify our direct variational approach with Lagrange multipliers. We define
$$ \Phi(t,\epsilon)=\Omega_{\nu}((h_{K}^p+tf+\epsilon)^{\frac{1}{p}})
$$
and
$$
\Psi(t,\epsilon)=\mu ([(h_{K}^p+tf+\epsilon)^{\frac{1}{p}}]),
$$
where $f\in C^+_e(\sn)$ if $\nu$ is even and one wants to take $K$ symmetric, or $f\in C^+(\sn)$ otherwise. We again set $h_t=(h_K^p+tf)^\frac1p.$ It can be noticed that
$$
\Psi(0,0)=\mu(K)=a,
$$
and, from Proposition~\ref{p:deriv},
$$
\Psi_\epsilon(0,0)=\frac{1}{p}S^\mu_p(K)\neq 0.
$$
Here, $\Psi_\epsilon$ denotes partial differentiation of $\Psi$ in the variable $\epsilon$.
Thus, when $t_0,\epsilon_0>0$ are sufficiently small, we can apply the implicit function theorem on $R=(-t_0,t_0)\times (-\epsilon_0,\epsilon_0)$, to obtain that there exists a function $\xi(t)$ such that $(t,\xi(t))$ is the unique solution of $\Psi(t,\epsilon)=a$ on $R$.

From $\Psi(t,\xi(t))=a$, we get $\Psi_t(0,0)+\Psi_\xi(0,0)\xi'(0)=0$, i.e., $\xi'(0)=\frac{-p\Psi_t(0,0)}{S^\mu_p(K)}$. Recalling by hypothesis $$\Omega_{\nu}(h_K)=\sup\{\Omega_{\nu}(f):\mu([f])=a,f\in C^+(\sn)\},$$
(or $f\in C^+_e(\sn)$ if $\nu$ is even and we wish to take $K$ symmetric), we find that the function $t\mapsto \Phi(t,\xi(t))$ is maximized at $0$. Consequently, we have
\begin{align*}
			0&=\frac{d}{dt}\Phi(t,\xi(t))\big|_{t=0}=\Phi_t(0,0)+\Phi_\xi(0,0)\xi'(0)\\
			&=\Phi_t(0,0)+\frac{\nu(\sn)\Psi_t(0,0)}{S^\mu_p(K)}=\frac{d}{dt}\Omega_{\nu}(h_t)\big|_{t=0}+\frac{\nu(\sn)}{S^\mu_p(K)}\frac{d}{dt}\mu([h_t])\big|_{t=0}\\
            &=-\frac{1}{p}\int_{\sn}f(u)d\nu(u)+\frac{\nu(\sn)}{pS^\mu_p(K)} \int_{\sn}f(u) dS^{\mu}_{K,p}.\\
		\end{align*}
The last-step follows from Proposition~\ref{p:deriv}. Then, by the arbitrariness of $f$, the second claim follows from the Riesz Representation theorem.
\end{proof}

There are a few steps that will be common to all proofs. Therefore, we list these as preparatory propositions. The first is fairly standard; we follow \cite[Lemma 6.3]{HLYZ18}.
\begin{proposition}
    \label{p:not_con} Let $\nu$ be a finite Borel measure on $\sn$ not concentrated on any great hemisphere, let $g(t)=|t|$ or $g(t)=t_+$ and let $p>0$. Then, there exists a constant $C_{\nu}(p)>0$ such that
    $$\int_{\sn}g(\langle\theta,u\rangle)^pd\nu(u) \geq C_{\nu}(p)^p\nu(\sn).$$
    Furthermore, for $p\geq 1$, $C_{\nu}(p)=C_{\nu}$, where $C_\nu\in (0,1]$ is a constant independent of $p$.
\end{proposition}
\begin{proof}
    Observe that the map, for $\theta\in\sn$,
$$\theta\to\int_{\sn}\langle\theta,u\rangle_+d\nu(u)$$
is strictly positive, since $\nu$ is not concentrated on any hemisphere. Since $\nu$ is finite, we can find a constant $C_\nu$ such that
$$\frac{1}{\nu(\sn)}\int_{\sn}|\langle\theta,u\rangle|d\nu(u) \geq \frac{1}{\nu(\sn)}\int_{\sn}\langle\theta,u\rangle_+d\nu(u) \geq C_\nu >0.$$
Furthermore, one has that $C_\nu\leq1$ since $|\langle \theta, u\rangle|\leq 1$. 
Notice via Jensen's inequality that, when $p\geq 1$,
$$\frac{1}{\nu(\sn)}\int_{\sn}|\langle\theta,u\rangle|^pd\nu(u) \geq \frac{1}{\nu(\sn)}\int_{\sn}\langle\theta,u\rangle_+^pd\nu(u) \geq C_\nu^p,$$
and the claim follows. When $p\in (0,1)$, we do not have such nice control over $C_\nu(p)$. We simply use the fact that, since $\nu$ is not concentrated on any great hemisphere, the function $$\theta \mapsto \int_{\sn}\langle \theta,u \rangle_+^pd\nu(u)$$
is continuous and therefore obtains a strictly positive minimum on $\sn$.
\end{proof}

\begin{proposition}
\label{p:converse}
    Let $K\in\kno$. Let $\mu$ be a Borel measure on $\R^n$ with density $\phi$ such that the Lebesgue set of $\phi$ contains $\partial K$ and $\phi > C> 0$ on $\partial K$. Let $p\in\R$ be so that $S^\mu_{K,p}$ is a finite Borel measure. Then, $S^\mu_{K,p}$ is not concentrated on any great hemisphere.
\end{proposition}
\begin{proof}
Let 
\[
c_1 = \begin{cases}
    \min_{u\in\sn}h_K(u) & p < 1,
    \\
    1 & p =1,
    \\
    \max_{u\in\sn}h_K(u) & p >1.
\end{cases}
\]
Notice that $c_1 >0$ by the assumptions on $K$. For all $p\in\R$, we have $h_K^{1-p}\geq c_1^{1-p}$ on $\sn$.
Observe that, for every $\theta\in\sn$
\begin{align*}
    \int_{\sn}\langle \theta,u \rangle_+ dS^{\mu}_{K,p}(u) &= \int_{\partial K}h_K^{1-p}(n_K(u))\langle \theta,n_K(y) \rangle_+\phi(y)d\mathcal{H}^{n-1}(y)
    \\
    &\geq 
    c_1^{1-p}C\int_{\partial K}\langle \theta,n_K(y) \rangle_+d\mathcal{H}^{n-1}(y)
    \\
    &= c_1^{1-p}C\int_{\sn}\langle \theta, u \rangle_+dS_K(u) >0
\end{align*}
    since $S_K$ is not concentrated on any great hemisphere.
\end{proof}

The next two propositions will be used to show non-degeneracy of the limit of the sequence of convex bodies, i.e., that the limiting body has a non-empty interior. The first proposition will be used in the proof of the $p\geq 0$ cases of Theorem~\ref{t:first_sym}. In those proofs, we are able to prove upper-and-lower bounds on sequences solving the optimization problems independently of each other, and so we may assume we have already shown boundedness when establishing the non-degeneracy.

\begin{proposition}
\label{prop:non_degen}
Let $\psi\in C^+(\R_+\setminus\{0\})\cap L^1_{\operatorname{loc}}(\R_+\setminus\{0\})$ be such that $\int_{0}^{R}\psi(t)t^{n-1}dt<\infty$ for every $R>0$. Then, define the rotationally invariant Borel measure $\mu$ on $\R^n$ with density $\psi(|\cdot|)$ and fix $a \in (0,\mu(\R^n))$. Suppose $\{K_\ell\}$ is a bounded sequence of symmetric convex bodies such that, $\mu(K_\ell)=a$ for all $\ell$ and one has $K_\ell$ converges to a symmetric, compact convex set $K$ in the Hausdorff metric as $\ell\to\infty$. Then, $K\in\kne.$
\end{proposition}
\begin{proof}
    We must show that the limiting set $K$ is non-degenerate. By way of contradiction, suppose it is, i.e., that it has an empty interior. Then, we can find a sequence of directions $\theta_\ell$ such that $\theta_\ell\in \operatorname{argmin}_{v\in\sn}h_{K_\ell}(v)$ and, as $\ell\to\infty$, $h_{K_\ell}(\theta_\ell) \to 0$ and $\theta_\ell\to\theta$ for some $\theta\in\sn$. Since each $K_{\ell}$ is symmetric, it is contained in the symmetric slab supporting $K_\ell$ with outer-unit normals $\pm \theta_\ell$, i.e., if we set
    $$W(v)=\{x\in\R^n:|\langle v,x\rangle|\leq 1\},$$
    then $K_\ell \subset W\left(\frac{\theta_\ell}{h_{K_\ell}(\theta_\ell)}\right)=h_{K_\ell}(\theta_\ell)W(\theta_\ell)$. Without loss of generality, we may rotate $K$ and each $K_\ell$ using the rotational invariance of $\mu$ and assume that $\theta_\ell=e_n$ for all $\ell$. Then we have $K_\ell \subset h_{K_\ell}(e_n)W(e_n)$.
    
    But also, since $K$ is bounded and $K_\ell\to K$ by hypothesis, there exists $R=R(K)>0$ such that $K,K_\ell \subset [-R,R]^n$ for $\ell$ large enough; also, $h_K(e_n)<R$ for $\ell$ large enough. Consequently, $K_\ell \subseteq [-R,R]^{n-1}\times[- h_{K_\ell}(e_n), h_{K_\ell}(e_n)].$ 
    
    Next, decompose $\R^n$ into $\R^{n-1}\times e_n\R$, and write $x=(\bar x,x_n)$, with $\bar x \in \R^{n-1}$. Then
    \begin{equation}
    \label{eq:first_a_calc}
    \begin{split}
        a=\mu(K_\ell)&\leq \mu\left([-R,R]^{n-1}\times[- h_{K_\ell}(e_n), h_{K_\ell}(e_n)]\right)
        \\
        &= 2\int_{[-R,R]^{n-1}}\int_{0}^{h_{K_\ell}(e_n)}\psi(|(\bar x,x_n)|)dx_nd\bar x 
        \\
        &=2h_{K_\ell}(e_n) \int_{[-R,R]^{n-1}}\left(\frac{1}{h_{K_\ell}(e_n)}\int_{0}^{h_{K_\ell}(e_n)}\psi(|(\bar x,x_n)|)dx_n\right)d\bar x.
    \end{split}
    \end{equation}
    Abusing notation and keeping $|\,\cdot\,|$ for the Euclidean norm on $\R^{n-1}$, one then deduces that, for every $R>0,$ $\int_{RB_2^{n-1}}\psi(|\bar x|)d\bar x <\infty$. Notice that, by Lebesgue's differentiation theorem,
    $$\lim_{\ell\to\infty} \frac{1}{h_{K_\ell}(e_n)}\int_{0}^{h_{K_\ell}(e_n)}\psi(|(\bar x,x_n)|)dx_n = \psi(|\bar x|).$$
    Thus, by dominated convergence, we have 
    \begin{equation}
    \begin{split}
    \label{eq:limiting_cubes}
    \lim_{\ell\to \infty} \int_{[-R,R]^{n-1}}&\left(\frac{1}{h_{K_\ell}(e_n)}\int_{0}^{h_{K_\ell}(e_n)}\psi(|(\bar x,x_n)|)dx_n\right)d\bar x 
    \\
    &= \int_{[-R,R]^{n-1}}\psi(|\bar x|)d\bar x 
    < \int_{\sqrt{n-1}RB_2^{n-1}}\psi(|x|)d\bar x 
    \\
    &= \left((n-1)\vol_{n-1}(B_2^{n-1})\int_0^{\sqrt{n-1}R} \psi(t)dt\right)\cdot\left(\int_{0}^{\sqrt{n-1} R}t^{n-2}\frac{\psi(t)dt}{\int_{0}^{\sqrt{n-1}R}\psi(t)dt}\right),\end{split}
    \end{equation}
    where the last equality follows from polar coordinates.
    Define the $(d-1)$th moment of $\psi$ on $(0,r)$:
    $$I_d(r):=\int_0^r t^{d-1}\frac{\psi(t)dt}{\int_0^r \psi(t)dt}.$$ 
    We now show that $I_{n-1}(R)$ is finite for every $R<\infty$. We already have from polar coordinates that $\int_{RB_2^n}\psi(|x|)dx < \infty$ is totally equivalent to $I_n(R) < \infty$ for every $R<\infty$. We now show the rudimentary fact that if a function $\psi$ has finite $(n-1)$th moment on $(0,R)$, then it has finite $(n-2)$nd moment on $(0,R)$ ($n\geq 2$). Indeed, we obtain from Jensen's inequality
    \begin{equation}
    \label{eq:I_finite}
    I_{n-1}(R) = \int_0^R \left(t^{n-1}\right)^\frac{n-2}{n-1}\frac{\psi(t)dt}{\int_0^R \psi(t) dt} 
    \leq \left(\frac{\int_0^R \psi(t)t^{n-1}dt}{\int_0^R \psi(t)dt}\right)^\frac{n-2}{n-1} = I_n(R)^\frac{n-2}{n-1}< \infty.\end{equation}

    We now obtain our contradiction. By \eqref{eq:first_a_calc}, we have
    \[
    0<\frac{a}{2}\left(\int_{[-R,R]^{n-1}}\left(\frac{1}{h_{K_\ell}(e_n)}\int_{0}^{h_{K_\ell}(e_n)}\psi(|(\bar x,x_n)|)dx_n\right)d\bar x\right)^{-1}\leq h_{K_\ell}(e_n).
    \]
    Sending $\ell\to\infty$, we have, by \eqref{eq:limiting_cubes},
    \[
    0<\frac{a}{2}\left(\left((n-1)\vol_{n-1}(B_2^{n-1})\int_0^{\sqrt{n-1}R} \psi(t)dt\right)I_{n-1}(\sqrt{n-1}R) \right)^{-1}\leq 0,
    \]
    which is a contradiction. We conclude the proof.
\end{proof}

The next proposition will be used for non-degeneracy in the proofs of the $p\geq 0$ cases of Theorem~\ref{t:first_finite}.
\begin{proposition}
    Let $\mu$ be a rotationally invariant probability measure with locally integrable density continuous on $\R^n\setminus\{0\}$. Fix $a\in [\frac{1}{2},1)$. If $K$ is a compact convex set such that $o \in K$ and $\mu(K)=a$, then $K\in\kno$.
    \label{p:origin_int}
\end{proposition}
\begin{proof}
    By way of contradiction, suppose that the origin is on the boundary of $K$. Then, there exists $v\in\sn$ such that $h_K(v)=0$. Let $H^+=\{x:\langle x,v \rangle \geq 0\}$ and define $H^-$ similarly, but with $\geq$ replaced by $\leq$. Without loss of generality, $K\subset H^+$. Since $\mu$ is rotationally invariant, $\mu(H^+)=\frac{1}{2}$. But, this means $\frac{1}{2} \leq a = \mu(K)<\mu(H^+)=\frac{1}{2}$, a contradiction. Thus, $K$ contains the origin in its interior.
\end{proof}

We are ready to prove our main results for rotationally invariant measures when $p>0$.

\begin{proof}[Proof of the $p>0$ case of Theorems~\ref{t:first_sym} and \ref{t:first_finite}]
    We will work with the functional $\Omega_\nu$. In light of Lemmas~\ref{l:op_relate} and \ref{l:op_solve}, it suffices to show that a solution exists to 
    \begin{enumerate}
        \item \eqref{eq:max_p} for the $p>0$ case of Theorem~\ref{t:first_finite} and
        \item the variant of \eqref{eq:max_p} consisting of $C_e^+(\sn)$ under the assumption $\nu$ is even for the $p>0$ case of Theorem~\ref{t:first_sym} and the extra claim of the $p>0$ case of Theorem~\ref{t:first_finite}, concerning the existence of $K\in\kne$.
    \end{enumerate}
    We note the converse direction in each case follows from Proposition~\ref{p:converse}. As for the forward directions, let $K_\ell$ be a maximising sequence, that is
    $$\lim_{\ell\to\infty}\Omega_\nu(K_\ell) = \sup\{\Omega_{\nu}(K):\mu(K)=a,K\in\kno\}.$$
    We claim that this sequence is bounded. Indeed, suppose not. If $\nu$ is even and we are aiming to show the existence of a symmetric maximizer, replace $\kno$ with $\kne$ in the maximisation problem and let $g(t)=|t|$. Otherwise, let $g(t)=t_+$.
    
    For each $K_\ell$, there exists a $\theta_\ell\in\sn$ such that $\rho_{K_\ell}(\theta_{\ell})$ is maximal. Furthermore, from \eqref{support function inequality2}, one has $h_{K_\ell}(u)\geq \rho_{K_\ell}(\theta_\ell)g(\langle \theta_\ell, u\rangle)$. Notice then that, from the above estimate and Proposition~\ref{p:not_con},
    \begin{align*}
        \Omega_\nu(K_\ell) \leq  -\frac{1}{p} \rho_{K_\ell}(\theta_\ell)^{p}\int_{\sn} g(\langle \theta_\ell, u\rangle)^p d\nu(u) \leq -\frac{C_\nu(p)^p\nu(\sn)}{p} \rho_{K_\ell}(\theta_\ell)^{p}.
    \end{align*}
    Sending $\ell \to \infty$, the above becomes $\lim_{\ell\to\infty}\Omega_\nu(K_\ell) = -\infty$, contradicting $K_\ell$ being a maximising sequence. Therefore, the convex sets $K_\ell$ are uniformly bounded, and so, by Blaschke selection, then, by passing to a subsequence if need be, converge to a compact, convex set $K$ containing the origin (which is symmetric if $\nu$ is even). For the $p>0$ case of Theorem~\ref{t:first_finite}, the claim follows from Proposition~\ref{p:origin_int}. Similarly, for the $p>0$ case of Theorem~\ref{t:first_sym}, the non-degeneracy follows from Proposition~\ref{prop:non_degen}.      
\end{proof} 

We are able to remove the constant in the even weighted $L^p$ Minkowski problems when $p>n$, which generalizes the result by Wang \cite{HW24} for the Gaussian measure, by using the Property~\textbf{(D)}$_p$. 

\begin{theorem}
\label{t:p>n}
    Fix $p>n$. Let $\mu$ be an even Borel measure on $\R^n$ with locally integrable density $\phi$ satisfying Property~\textbf{(D)}$_p$ such that $\lim_{|x|\to 0}\phi(x)\in (0,\infty]$. Suppose $\nu$ is an even, finite Borel measure on the sphere. If $\nu$ is not concentrated on any great hemisphere, then there exists an origin symmetric convex body $K$ such that
    $$\nu=S^\mu_{K,p}.$$
    If $\phi$ is positive on $\partial K$, then the converse direction holds as well.
\end{theorem}
\begin{proof} We will work with the functional $\psi_\nu$. In light of Lemmas~\ref{l:op_relate} and \ref{l:op_solve}, it suffices to show that a solution to \eqref{eq:max_p>n} exists under the assumption $p>n$. Let 
$$\Psi(r,\nu)=\psi_{\nu}(h_{rB_2^n})=\mu(rB_2^n)-\frac{r^p}{p}\nu(\sn).$$
 Let $\phi$ be the density of $\mu$, and write
\begin{align*}
\mu(rB_2^n)=r^n\vol_n(B_2^n)\frac{1}{\vol_n(rB_2^n)}\int_{rB_2^n}\phi(x)dx.
\end{align*}
Consequently,
$$\Psi(r,\nu) = r^n\vol_n(B_2^n)\left(\frac{1}{\vol_n(rB_2^n)}\int_{rB_2^n}\phi(x)dx -r^{p-n}\frac{n}{p}\frac{\nu(\sn)}{\vol_{n-1}(\sn)} \right).$$
Therefore, as $r\to 0^+$, $r^{p-n}\frac{n}{p}\frac{\nu(\sn)}{\vol_{n-1}(\sn)} \to 0$, but, from the Lebesgue differentiation theorem,
$$\frac{1}{\vol_n(rB_2^n)}\int_{rB_2^n}\phi(x)dx \to \lim_{|x|\to 0}\phi(x) >0$$
as $r\to0$. This shows, $\Psi(r) >0$ for $r$ small enough.  
On the other hand, write
$$\Psi(r,\nu) = r^p\left(\frac{\mu(rB_2^n)}{r^p}-\frac{\nu(\sn)}{p}\right). $$
Next, notice Property~\textbf{(D)}$_p$ implies that, for every $x\in\R^n\setminus\{0\}$, 
\begin{equation}\lim_{r\to\infty}r^{n-p}\phi(rx)=0.
\label{eq:prop_d}
\end{equation}
Observe that 
\begin{align*}
\lim_{r\to\infty} \frac{\mu(rB_2^n)}{r^p} = \lim_{r\to\infty} \frac{1}{r^p}\int_{rB_2^n}\phi(x)dx= \lim_{r\to\infty} \int_{B_2^n} r^{n-p}\phi(rx)dx.
\end{align*}
For $r$ large enough, $r^{n-p}\phi(rx)$ is dominated by, say, $1$ for almost all $x \in B_2^n$. Hence, by dominated convergence, we obtain
\begin{align*}
\lim_{r\to\infty} \frac{\mu(rB_2^n)}{r^p}=0.
\end{align*}
Thus, $\Psi(r,\nu)<0$ for $r$ large enough. We therefore let $R(\nu)=\sup\{r>0:\Psi(r,\nu)>0\}$; we have just shown that $R(\nu)\in (0,\infty)$. We emphasize here that we view $R$ as a function from the set of Borel measures on the sphere to $(0,\infty)$.

Arguing like in the proofs of the $p>0$ case of Theorems~\ref{t:first_sym} and \ref{t:first_finite}, fix an arbitrary symmetric convex body $K$; then, from \eqref{support function inequality2}, $h_K(u)\geq \rho_K(\theta_K)| \langle \theta_K, u\rangle|$, where $\theta_K\in\sn$ is so that $\rho_K(\theta_K)$ is maximal, and $K$ is contained in the ball of radius $\rho_K(\theta_K).$ 

We can now directly compute using the above estimate and Proposition~\ref{p:not_con}:
\begin{align*}
\psi_{\nu}(h_K)&\leq \mu(K)-\frac{\rho_K(\theta_K)^p}{p}\int_{\sn}|\langle \theta_K, u\rangle|^pd\nu(u)
\leq \mu(K)-\frac{(\rho_K(\theta_K)C_\nu)^p}{p}\nu(\sn)
\\
&\leq \mu(\rho_K(\theta_K)B_2^n)-\frac{(\rho_K(\theta_K)C_\nu)^p}{p}\nu(\sn)
=\Psi(\rho_K(\theta_K),C_{\nu}^p\nu).
\end{align*}
Therefore, $\psi_\nu(h_K)>0$ implies $\Psi(\rho_K(\theta_K),C_{\nu}^p\nu) >0$, which means $$\rho_K(\theta_K) \leq R(C_{\nu}^p\nu)\leq \max\{R(C_{\nu}^p\nu),R(\nu)\}=:R,$$
and so $K\subseteq RB_2^n$. Formally, we can restrict our search for the maximizer to the set
$$\mathcal{F}=\{K\in\kne: K\subset RB_2^n\}.$$

We have that $\sup\{\psi_{\nu}(h_K):K\in\mathcal{F}\}>0$ by the construction of $R$. Let $\{K_\ell\}\subset \kne$ be a sequence of convex bodies so that $\lim_{\ell\to\infty}\psi_{\nu}(h_{K_\ell})=\sup\{\psi_{\nu}(h_K):K\in\mathcal{F}\}.$ Via Blaschke selection, there exists some $K\in\mathcal{F}$ and a sub-sequence $\{K_{\ell_j}\}\subset\{K_\ell\}$ such that $\lim_{j\to\infty}\psi_{\nu}(h_{K_{\ell_j}})=\psi_{\nu}(h_{K})$, and so $\lim_{j\to\infty}K_{\ell_j}=K$ with respect to the Hausdorff metric. Finally, we see that, from the definition of $\mathcal{F}$,
\begin{align*}
    \mu(K)=\lim_{j\to\infty}\mu(K_{\ell_j})\geq \lim_{j\to\infty}\psi_{\nu}(h_{K_{\ell_j}}) > 0,
\end{align*}
and so $K$ has non-empty interior, and is thus a convex body in the proper sense. The final claim follows from Proposition~\ref{p:converse}.
\end{proof}

As we saw in the above proof, Theorem~\ref{t:p>n} actually only requires the slightly weaker assumption
    $$\lim_{r\to\infty} \int_{B_2^n}\frac{\phi(rx)}{r^{p-n}}dx = 0.$$
We note that any finite measure on $\R^n$ with continuous density, that is strictly positive at $o$, satisfies the hypotheses of Theorem~\ref{t:p>n}.

\section{The case when p is zero}
\label{sec:p_zero}
This section is dedicated to proving the $p=0$ case of Theorems~\ref{t:first_sym} and \ref{t:first_finite}. Our technique can be traced through \cite{BLYZ13,HLYZ16,HXZ21,LU24}, especially \cite{BLYZ13}. We will be working with the entropy functional
\begin{equation}
    \mathcal{E}_\nu(f)=-\frac{1}{\nu(\sn)}\int_{\sn}\log f(u)d\nu(u); \quad \mathcal{E}_\nu(K):=\mathcal{E}_\nu(h_K),
    \label{eq:entropy}
\end{equation}
which will be taken over $C^+(\sn)$ or $C^+_e(\sn)$. We will consider the following optimization problem.

\begin{lemma}
\label{l:op_relate_p0}
    Let $\mu$ be a Borel measure on $\R^n$, and fix a Borel measure $\nu$ on $\sn$. Let $K\in\kno$. Then, for a fixed $a\in (0,\mu(\R^n))$, $K$ solves 
        \begin{equation}
        \label{eq:max_log}
        \sup\{\mathcal{E}_{\nu}(K):\mu(K)=a,K\in\kno\}\end{equation} if and only if $h_K$ solves $$\sup\{\mathcal{E}_{\nu}(f):\mu([f])=a,f\in C^+(\sn)\}.$$
     Additionally, if $K$ is also assumed to be symmetric in the first optimization problem, then the set $C^+(\sn)$ is replaced by $C_e^+(\sn)$ in the second optimization problem.
\end{lemma}
\begin{proof}
One must simply notice that the function $-\log x$ is monotonically decreasing to obtain, since $h_{[f]} (u) \leq f(u)$ point-wise,
\begin{align*}
    \nu(\sn)\mathcal{E}_{\nu}([f])&= -\int_{\sn} \log h_{[f]}(u)d\nu(u)
    \\
    &\geq -\int_{\sn} \log f(u)d\nu(u)=\mathcal{E}_{\nu}(f)\nu(\sn).
\end{align*}
The equivalence follows since $[h_{[f]}]=[f]$.
\end{proof}
The next lemma shows that solving the optimization problem solves our weighted log-Minkowski problem.
\begin{lemma}
\label{l:op_solvep0}
    Let $K\in\kno$, and let $\mu$ be a measure on $\R^n$ with a density whose Lebesgue set contains $\partial K$. If $K$ solves \eqref{eq:max_log}, then $\nu=\frac{\nu(\sn)}{S^\mu_0(K)}S^{\mu}_{K,0}.$
\end{lemma}

\begin{proof}
The proof is similar to the proof of the second claim in Lemma~\ref{l:op_relate}; define 
$$ \Phi^{(2)}(t,\epsilon)=\mathcal{E}_{\nu}(h_{K}e^{tf+\epsilon})
$$
and
$$
\Psi^{(2)}(t,\epsilon)=\mu ([h_{K}e^{tf+\epsilon}]).
$$
It can again be noticed that
$$
\Psi^{(2)}(0,0)=\mu(K)=a,
$$
and
$$
\Psi_\epsilon^{(2)}(0,0)=S^\mu_0(K)\neq 0.
$$
Then, repeat the use of the implicit function theorem and the subsequent computations; note that one will have to use Corollary~\ref{cor:deriv_2} in place of Proposition~\ref{p:deriv}.
\end{proof}
We now prove our theorems for this section.
\begin{proof}[Proof of the $p=0$ case of Theorems~\ref{t:first_sym} and \ref{t:first_finite}]
In light of Lemmas~\ref{l:op_solvep0} and \ref{l:op_relate_p0}, it suffices to show a solution exists to \eqref{eq:max_log}. Let $K_\ell$ be a maximizing sequence, that is
$$\lim_{\ell\to\infty} \mathcal{E}_\nu(K_\ell)= \sup \left\{\mathcal{E}_\nu(K): \mu(K)=a, K \in \kno\right\}$$
for the $p=0$ case of Theorem~\ref{t:first_finite}, the not necessarily even case, and    
    \begin{equation}
    \label{log_min}
    \lim_{\ell\to\infty} \mathcal{E}_\nu(K_\ell)= \sup \left\{\mathcal{E}_\nu(K): \mu(K)=a, K \in \kne\right\}\end{equation}
    for the even case of (the $p=0$ case of) Theorem~\ref{t:first_finite} and (the $p=0$ case of) Theorem~\ref{t:first_sym}.

Let $r(a)$ be such that $\mu(r(a)B_2^n)=a$, and set $B_a = r(a)B_2^n$. Then, observe that $\mathcal{E}_{\nu}(B_a)=-\log r(a).$ Consequently,
\begin{equation}\lim_{\ell\to\infty}\mathcal{E}_{\nu}(K_\ell) \geq -\log r(a).\label{eq:bound}
\end{equation}
Following the approach from \cite[Theorem 6.3]{BLYZ13}, since each $K_\ell$ is non-empty, there exist ellipsoids $E_\ell$ and vectors $c_\ell$ via John's theorem \cite{FJ37} such that
$E_\ell \subset K_\ell \subset c_\ell + b(n)(E_\ell-c_\ell).$ Here, $b(n)=\sqrt{n}$ in the symmetric case (note $c_\ell = o$ in the symmetric case as well) and $b(n)=n$ in the general case. We then denote by $u_{1, \ell}, \ldots, u_{n, \ell} \in \sn$ the principal directions of $E_\ell$ indexed to satisfy $h_{1, \ell} \leq \cdots \leq h_{n, \ell},$ where $h_{i, \ell}=h_{E_\ell}\left(u_{i, \ell}\right),$ for $i=1, \ldots, n$. Defining the cross-polytopes
$$C_\ell=\left[\pm h_{1, \ell} u_{1, \ell}, \ldots, \pm h_{n, \ell} u_{n, \ell}\right],$$
one obtains
$C_\ell \subset E_\ell \subset \sqrt{n}C_\ell.$ Consequently, we deduce $$C_\ell \subset K_\ell \subset c_\ell + b(n)(\sqrt{n}C_\ell - c_\ell).$$

We must have that there exists a sequence of numbers $A_\ell$ such that $\vol_n(C_\ell) \geq A_\ell >0.$ Indeed, if $\vol_n(C_\ell)\to 0$, then $\vol_n(K_\ell)\to 0$. In the case when $\mu$ is finite, i.e., its density $\phi$ is in $L^1(\R^n)$, one obtains from dominated convergence that $\mu(K_\ell) \to 0$.  Similarly, when $ \phi \in L^q(\R^n)$ for some $q \in (1,\infty),$ we obtain from H\"older's inequality
$$\mu(K_\ell) = \int_{\R^n}\chi_{K_\ell}(x)\phi(x)dx \leq \|\phi\|_{L^q(\R^n)}\vol_n(K_\ell)^\frac{q-1}{q},$$
which again goes to zero if $\vol_n(K_\ell)$ does. In either case, this contradicts $\mu(K_\ell)=a$ for all $\ell$. 

Using the formula for the volume of a cross-polytope, we then obtain that
\begin{equation}
\label{supp_prod}
\prod_{i=1}^nh_{i,\ell}=\frac{n!\vol_n(C_\ell)}{2^n} \geq \frac{n!A_\ell}{2^n}\end{equation}
Notice that, with $\widetilde{C_\ell}=\left(\frac{n!A_\ell}{2^n}\right)^{-\frac{1}{n}}C_\ell$, \begin{equation}
\label{eq:entropy_relate}
\mathcal{E}_\nu\left(\widetilde{C_\ell}\right)=-\frac{1}{\nu(\sn)}\int_{\s^{n-1}} \log h_{\widetilde{C_\ell}}(u) d \nu(u)=-\frac{1}{n}\log\left(\frac{2^n}{n!A_\ell}\right)+\mathcal{E}_\nu(C_\ell).\end{equation}

By way of contradiction, suppose the sequence $\{K_\ell\}$ is not bounded. Then, the sequence $\{C_\ell\}$ is not bounded. Therefore, by passing to a subsequence if need be, one has
$\lim_{\ell\to\infty}h_{n,\ell}=\infty.$ One then obtains from \eqref{supp_prod} and \cite[Lemma 6.2]{BLYZ13} that $\{\mathcal{E}_{\nu}(\widetilde{C}_\ell)\}$, and therefore $\{\mathcal{E}_{\nu}(C_\ell)\}$ via \eqref{eq:entropy_relate}, are  not bounded from below. Since $\mathcal{E}_{\nu}(C_\ell) \geq \mathcal{E}_{\nu}(K_\ell)$, this yields $\{\mathcal{E}_{\nu}(K_\ell)\}$ is not bounded from below, contradicting \eqref{eq:bound} for $\ell$ large enough. Thus, we must have that $\{K_\ell\}$ is bounded. Then, appealing to  Propositions~ \ref{p:origin_int} and \ref{prop:non_degen}, for the $p=0$ cases of Theorems~\ref{t:first_sym} and \ref{t:first_finite}, respectively, we have that the sequences are non-degenerate. Our claim then follows from Blaschke selection.
\end{proof}

\section{The case of negative p}
\label{sec:neg_p}
In this section, we consider $p<0$. In contrast to the previous sections, we need an additional hypothesis on the measure. If $\mu$ is radially decreasing and rotationally invariant, then it has a density $\phi$ of the form $\phi(x)=\psi(|x|)$, where $\psi$ is a continuous and decreasing function on $(0,\infty)$. We will need explicitly that $\lim_{t\to 0^+}\psi(t) \in (0,\infty)$; this forces continuity of $\phi$ on all of $\R^n$. 

In this section, we will prove the cases $p<0$ of Theorems~\ref{t:first_sym} and \ref{t:first_finite}. The techniques from our approach can be traced through, e.g., \cite{CW06,CHLL20,CCL21,FHX23}. We start with the following two propositions for non-degeneracy. The first will be used in the proof of the $p<0$ case of Theorem~\ref{t:first_sym}. However, we need non-degeneracy when establishing boundedness on the optimizing sequence. Thus, the content strays from that in Proposition~\ref{prop:non_degen}.

\begin{proposition}
\label{prop:non_degen_2}
Let $n\in \mathbb{N}$ and let $\psi\in C^+(\R_+\setminus\{0\})\cap L^1(\R_+\setminus\{0\})$ be such that $\int_0^\infty\psi(t)t^{n-1}dt<\infty$ and normalized so that the rotationally invariant measure $\mu$ on $\R^n$ with density $\psi(|\cdot|)$ is probability. Fix $a \in (0,1)$ and suppose $\{K_\ell\}$ is a sequence of symmetric convex bodies such that $\mu(K_\ell)=a$ for all $\ell$ and one has $K_\ell$ converges to a symmetric closed convex set $K$ in the Hausdorff metric as $\ell\to\infty$. Then, $K$ is non-degenerate (i.e., has non-empty interior).

\end{proposition}
\begin{proof}
    By way of contradiction, suppose $K$ is degenerate. Then, we can find a sequence of directions $\theta_\ell$ such that $\theta_\ell\in \operatorname{argmin}_{v\in\sn}h_{K_\ell}(v)$ and, as $\ell\to\infty$, $h_{K_\ell}(\theta_\ell) \to 0$ and $\theta_\ell\to\theta$ for some $\theta\in\sn$. Since each $K_{\ell}$ is symmetric, it is contained in the symmetric slab supporting $K_\ell$ with outer-unit normals $\pm \theta_\ell$, i.e., if we set
    $$W(v)=\{x\in\R^n:|\langle v,x\rangle|\leq 1\},$$
    then $K_\ell \subset W\left(\frac{\theta_\ell}{h_{K_\ell}(\theta_\ell)}\right)=h_{K_\ell}(\theta_\ell)W(\theta_\ell)$. Without loss of generality, we may rotate $K$ and each $K_\ell$ using the rotational invariance of $\mu$ and assume that $\theta_\ell=e_n$ for all $\ell$. Then we have $K_\ell \subset h_{K_\ell}(e_n)W(e_n)$. We again decompose $\R^n$ into $\R^{n-1}\times e_n\R$, and write $x=(\bar x,x_n)$, with $\bar x \in \R^{n-1}$. Then,
    \begin{equation}
    \label{eq:a_calc}
    \begin{split}
        a=\mu(K_\ell)&< \mu\left(\R^{n-1}\times[- h_{K_\ell}(e_n), h_{K_\ell}(e_n)]\right)
        \\
        &= 2\int_{\R^{n-1}}\int_{0}^{h_{K_\ell}(e_n)}\psi(|(\bar x,x_n)|)dx_nd\bar x 
        \\
        &=2h_{K_\ell}(e_n) \int_{\R^{n-1}}\left(\frac{1}{h_{K_\ell}(e_n)}\int_{0}^{h_{K_\ell}(e_n)}\psi(|(\bar x,x_n)|)dx_n\right)d\bar x.
    \end{split}
    \end{equation}
    Abusing notation and keeping $|\,\cdot\,|$ for the Euclidean norm on $\R^{n-1}$, one then deduces that $\int_{\R^{n-1}}\psi(|\bar x|)d\bar x <\infty$. Notice that, by Lebesgue's differentiation theorem,
    $$\lim_{\ell\to\infty} \frac{1}{h_{K_\ell}(e_n)}\int_{0}^{h_{K_\ell}(e_n)}\psi(|(\bar x,x_n)|)dx_n = \psi(|\bar x|).$$
    Thus, we obtain
    \begin{equation}
    \label{eq:limiting_slabs}
    \lim_{\ell\to\infty} \int_{\R^{n-1}}\left(\frac{1}{h_{K_\ell}(e_n)}\int_{0}^{h_{K_\ell}(e_n)}\psi(|(\bar x,x_n)|)dx_n\right)d\bar x = \int_{\R^{n-1}}\psi(|\bar x|)d\bar x.\end{equation}
    To justify the limit: first, truncate $\R^{n-1}$ to $RB_2^{n-1}$ for arbitrary $R$. Then, from the slightly stronger version of Lebesgue's differentiation theorem, which we recall asserts that
    $$\lim_{\ell\to\infty} \frac{1}{h_{K_\ell}(e_n)}\int_{0}^{h_{K_\ell}(e_n)}|\psi(|(\bar x,x_n)|)-\psi(|(\bar x,0)|)|dx_n = 0,$$
    one easily finds that \eqref{eq:limiting_slabs} holds (but with $\R^{n-1}$ replaced by $RB_2^{n-1}$). Then, send $R$ to $\infty$ to establish \eqref{eq:limiting_slabs} via monotone convergence. Next, observe from polar coordinates that
   $$\int_{\R^n}\psi(|x|)dx = n\vol_n(B_2^n)\int_0^\infty \psi(t)t^{n-1}dt,$$
    i.e.,  given a function $\psi$ on $(0,\infty)$, the function $\psi(|x|)$ is integrable on $\R^n$ if and only if $\psi$ has finite $(n-1)$th moment on $\R_+$.  We now show the rudimentary fact that if an integrable function $\psi$ on $\R_+$ has finite $(n-1)$th moment, then it has finite $(n-2)$nd moment ($n\geq 2$). Indeed, we obtain from Jensen's inequality
    \begin{align*}\int_0^\infty \psi(t)t^{n-2}dt &= \left(\int_0^\infty \psi(t) dt\right)\cdot \left(\int_0^\infty \left(t^{n-1}\right)^\frac{n-2}{n-1}\frac{\psi(t)dt}{\int_0^\infty \psi(t) dt}\right) 
    \\
    &\leq \left(\int_0^\infty \psi(t) dt\right)\cdot\left(\frac{\int_0^\infty \psi(t)t^{n-1}dt}{\int_0^\infty \psi(t)dt}\right)^\frac{n-2}{n-1} < \infty.\end{align*}

    Consequently, passing to polar coordinates on $\R^{n-1}$, 
    \begin{equation}\label{eq:bounded}\int_{\R^{n-1}}\psi(|\bar x|)d\bar x = (n-1)\vol_{n-1}(B_2^{n-1})\int_{0}^{\infty}\psi(t)t^{n-2}dt <\infty.\end{equation} With this in hand, we can obtain our contradiction. We have from \eqref{eq:a_calc}
    \[
    0<\frac{a}{2}\left(\int_{\R^{n-1}}\left(\frac{1}{h_{K_\ell}(e_n)}\int_{0}^{h_{K_\ell}(e_n)}\psi(|(\bar x,x_n)|)dx_n\right)d\bar x\right)^{-1}<h_{K_\ell}(e_n).
    \]
    Taking the limit as $\ell\to\infty$, we obtain from \eqref{eq:limiting_slabs}
    \[
    0<\frac{a}{2}\left(\int_{\R^{n-1}}\psi(|\bar x|)d\bar x\right)^{-1}\leq 0,
    \]
    where we used \eqref{eq:bounded} to deduce the strictness away from zero. We have reached the desired contradiction, and we conclude the proof.
\end{proof}

The next proposition will be used to prove the $p<0$ case of Theorem~\ref{t:first_finite}.
\begin{proposition}
    Let $\mu$ be a rotationally invariant probability measure with locally integrable density continuous on $\R^n\setminus\{0\}$. Fix $a\in (\frac{1}{2},1)$. If $K$ is a closed convex set so that $o \in K$ and $\mu(K)=a$, then $o\in \text{int}(K)$.
    \label{p:origin_int_2}
\end{proposition}
\begin{proof}
    By way of contradiction, suppose that the origin is on the boundary of $K$. Then, there exists $v\in\sn$ such that $h_K(v)=0$. Let $H^+=\{x:\langle x,v \rangle \geq 0\}$ and define $H^-$ similarly, but with $\geq$ replaced by $\leq$. Without loss of generality, $K\subseteq H^+$. Since $\mu$ is rotationally invariant, $\mu(H^+)=\frac{1}{2}$. But this means $\frac{1}{2} < a = \mu(K)\leq \mu(H^+)=\frac{1}{2}$, a contradiction. Thus, $K$ contains the origin in its interior.
\end{proof}
Unlike Proposition~\ref{p:origin_int}, we do not have boundedness in Proposition~\ref{p:origin_int_2}. Thus, we had to assume $a>\frac{1}{2}$ to prevent $K=H^+$. We next prove the $p<0$ cases of Theorems~\ref{t:first_sym} and \ref{t:first_finite} when the data measure $\nu$ is zero on every great subsphere. Then, we prove the general case via approximation.

\begin{lemma}
\label{l:vanishing}
    Let $\mu$ be a rotationally invariant, radially decreasing probability measure on $\R^n$ with continuous density. Let $p<0$. We consider two cases. 
    \begin{enumerate}
        \item Fix $a \in (0,1)$. Let $\nu$ be an even, finite Borel measure on $\sn$ that vanishes on all great subspheres. Then, there exists a symmetric convex body $K$ with $\mu(K)=a$ solving (the variant of)  \eqref{eq:max_p} (with $\kne$) for the given $\mu$ and $\nu$.
        \item Suppose also that, if $d\mu(x)=\psi(|\cdot|)dx,$ then $\psi\in L^1(\R_+)$. Fix $a \in (\frac{1}{2},1)$. Let $\nu$ be a finite Borel measure on $\sn$ that vanishes on all great subspheres. Then, there exists a convex body $K$ with $\mu(K)=a$ solving \eqref{eq:max_p} for the given $\mu$ and $\nu$.
    \end{enumerate}
\end{lemma}
\begin{proof}
    Let $K_\ell$ be a sequence tending to the supremum, that is
    \begin{equation}
    \label{eq:vanishes}
    \begin{cases}
    \lim_{\ell\to\infty}\Omega_\nu(K_\ell)=\sup\{\Omega_\nu(K): \mu(K)=a,K\in\kne\}>0 & \text{case 1,}
    \\
\lim_{\ell\to\infty}\Omega_\nu(K_\ell)=\sup\{\Omega_\nu(K): \mu(K)=a,K\in\conbod_o\}>0 & \text{case 2.}
    \end{cases}
\end{equation}
    It suffices to show that the sequence $\{K_\ell\}$ has a non-degenerate limit and is contained in a large ball; by Blaschke selection, this would then yield the result that there exists a convex body $K$ such that up to a subsequence, $K_\ell\to K$ and, by construction, $K$ is a convex body with $\mu(K)=a >0$ (which is symmetric for case 1).

    The fact that the limit is non-degenerate follows from Proposition~\ref{prop:non_degen_2} for case 1 and Proposition~\ref{p:origin_int_2} for case 2; we now show that $K_\ell$ is bounded. Again, by way of contradiction, suppose otherwise. From polarity, the sequence $K_\ell^\circ$ is degenerate. That is, by passing to a subsequence if needed, $K_\ell^\circ$ converges to a compact convex set $K^\circ$ and there exists $v_0\in\sn$ such that $h_{K^\circ}(v_0)=0$. Define the set, for every $r\in\sn$, $$\omega_r(v_0)=\{u\in\sn:|\langle u,v_0\rangle| >r\}.$$
    Then, by \cite[Lemma 6.2]{HLYZ18},
    $\rho_{K_\ell^\circ}\to 0$ on $\omega_r(v_0)$. Since we have already shown that $h_{K_\ell}$ is bounded from below, it follows that $\rho_{K^{\circ}_\ell}$ is bounded from above. Consequently, $K_\ell^\circ\subset RB_2^n$ for some $R>0$. Then:
    \begin{align*}
        -p\Omega_\nu(K_\ell) =\int_{\sn}h_{K_\ell}(u)^pd\nu(u) &= \int_{\omega_r(v_0)}\rho_{K_\ell^\circ}(u)^{-p}d\nu(u) + \int_{\sn\setminus\omega_r(v_0)}\rho_{K_\ell^\circ}(u)^{-p}d\nu(u)
        \\
        &\leq \int_{\omega_r(v_0)}\rho_{K_\ell^\circ}(u)^{-p}d\nu(u) + R^{-p}\nu\left(\sn\setminus\omega_r(v_0)\right).
    \end{align*}
    Since $\nu$ vanishes on all great subspheres, we obtain
    $$\lim_{r\to0^+}\nu\left(\sn\setminus\omega_r(v_0)\right) = \nu(\sn\cap v_0^\perp)=0.$$
    We deduce that $\lim_{\ell\to\infty} -p\Omega_\nu(K_\ell) = 0$, contradicting \eqref{eq:vanishes}. 
\end{proof}

To establish our next set of results, we will use the theory of partial differential equations (and, more specifically, Monge-Amp\`ere equations) in conjunction with an approximation argument. Suppose that we have $K\in\kno$ and a Borel measure $\nu$ on $\sn$ such that
$$\nu= S^{\mu}_{K,p}$$
where $\mu$ is a Borel measure on $\R^n$ with density $\phi$ whose Lebesgue set contains $\partial K$. If $\nu$ has strictly positive density $f$ on $\sn$ and $K$ has $C^2$ smooth boundary, then this rewrites as
the following Monge-Amp\`ere equation:
 \begin{equation}\label{eq:Monge-Ampere equation}
		h_{K}^{1-p}\phi(\nabla h_{K})\det(\nabla^2 h_{K}+h_{K}I)=f.
	\end{equation}
 Our first lemma transforms the Minkowski problem on $\sn$ into one on $\R^{n-1}$. We will use radial projections. This technique can be traced through \cite[pgs. 13-17]{PAVY78} and \cite{CW95,HYLJ_21,UJ26,SZ25}.  We use the usual notation that, for $e\in\sn$, $e^\perp=\{x\in\R^n:\langle x,e \rangle=0\}$ is the hyperplane through the origin orthogonal to $e$. Then, $H_e:=e^\perp + e$ is the hyperplane tangential to $\sn$ at $e$. We denote by $\pi_e:e^\perp\to\sn$ the radial projection of $H_e$ to $\sn$, i.e.,
\begin{equation}
\label{eq:radial_proj}
\pi_e(y)=\frac{y+e}{\sqrt{1+|y|^2}}, \quad \text{which yields } \quad \langle\pi_e(y), e\rangle=\frac{1}{\sqrt{1+|y|^2}}.\end{equation}

\begin{proposition}
\label{p:monge_on_perp}
    Fix $K\in\kno$ and let $h=h_K$ solve \eqref{eq:Monge-Ampere equation} for a given nonnegative function $f$ on $\sn$ and a Borel measure $\mu$ on $\R^n$ with continuous density $\phi$. Fix $e\in\sn$ and define $H_K,F:e^\perp\to\R_+$ as
    $$H_K(y)=h_K(e+y) \quad \text{and} \quad F(y)=\left(1+|y|^2\right)^{-\frac{n+p}{2}}f\left(\frac{e+y}{\sqrt{1+|y|^2}}\right).$$
    Then, $H_K$ solves the following Monge-Amp\`ere equation weakly on $e^\perp \approx \R^{n-1}$:
    $$H_K(y)^{1-p}\phi(\nabla H_K(y)-\left(H_K(y) - \langle \nabla H_K(y),e \rangle \right)\cdot e)\det\hess{H_K(y)} = G(y).$$
\end{proposition}

\begin{proof}
 Observe that we can write from \eqref{eq:radial_proj}
    \begin{equation} 
    \label{eq:supp_relate_proj}
    H_K(y)=\sqrt{1+|y|^2}h_K(\pi_e(y)).\end{equation}
    From this representation, we obtain via \eqref{eq:faces} that
    $$\partial H_K(y) = \partial h_K(y+e)|_{y\in e^\perp}=F(K, y+e)|_{y\in e^\perp}=F(K, \pi_e(y))|_{y\in e^\perp}.$$
    Using the 1-homogeneity of the support function, we differentiate \eqref{eq:supp_relate_proj} to obtain
    $$\nabla H_K(y) = \nabla h_K(v)|_{v=\pi_e(y)}.$$
    Therefore, there exists $t\in \R$ such that
    \begin{equation}\nabla h_K(v) = \nabla H_K(y) -te,
    \label{eq:PDE_grad}
    \end{equation}
    with $v=\pi_e(y).$ Then, using \eqref{supp_homo},
    $$h_K(v) = \langle \nabla h_K(v),v \rangle = \langle \nabla H_K(y),v\rangle - t \langle e,v \rangle. $$
    Multiplying by $\sqrt{1+|y|^2}$ and using \eqref{eq:radial_proj} and \eqref{eq:supp_relate_proj}, we obtain
    $$t= \langle \nabla H_K(y),y\rangle - H_K(y).$$
    Consequently, \eqref{eq:PDE_grad} becomes
    \begin{equation}\nabla h_K(v) = \nabla H_K(y) +\left(H_K(y) - \langle \nabla H_K(y),y\rangle\right)\cdot e,
    \label{eq:PDE_grad_2}
    \end{equation}
    From \eqref{eq:sur_mae}, one has
    $$\mathcal{H}^{n-1}\left(\bigcup_{u \in \pi_e(E)}\left(F(K, u)|_{u\in e^{\perp}}\right)\right)=\int_{\pi_e(E)}\langle u, e\rangle dS_K(u).$$
    
    We may re-formulate \eqref{eq:Monge-Ampere equation} as equality in the following: 
    $$dS_K(u) = \frac{1}{\phi(\nabla h_K(u))}h_K(u)^{p-1}f(u)du$$
    (note the fact that $S_K(E)<S_K(\sn)<\infty$ manifestly yields the integrability of the right-hand side over $\sn$). Consequently, using \eqref{eq:mae} and \eqref{eq:sur_mae}, we obtain, for an arbitrary Borel $E\subset e^\perp$,
    \begin{align*}
        \int_{E}\det\hess{H_K(y)}d\mathcal{H}^{n-1}(y)
        & = \int_{\pi_e(E)}\langle u, e\rangle dS_K(u)
        \\
        & = \int_{\pi_e(E)} \frac{\langle u, e\rangle}{\phi(\nabla h_K(u))} h_K(u)^{p-1} f(u)du
        \\
        & = \int_{E} \frac{1}{\phi(\nabla H_K(y) +\left(H_K(y) - \langle \nabla H_K(y),y\rangle\right)\cdot e)} H_K(y)^{p-1} F(y)d\mathcal{H}^{n-1}(y),
    \end{align*}
    where, in the last line, we used that the determinant of the Jacobian of the coordinate transformation $v=\pi_e(y)$ is $(1+|y|^2)^{-\frac{n}{2}}.$ We note that at this point, we have shown weak equality in
    $$\det\hess{H_K(y)}=\frac{1}{\phi(\nabla H_K(y) +\left(H_K(y) - \langle \nabla H_K(y),y\rangle\right)\cdot e)} H_K(y)^{p-1} F(y).$$ The claim follows from this; if we define a measure $M_1$ on $e^\perp$ as the measure with density $\det\hess{H_K(y)}$, and $$dM_2(y)=\frac{1}{\phi(\nabla H_K(y) +\left(H_K(y) - \langle \nabla H_K(y),y\rangle\right)\cdot e)} H_K(y)^{p-1}F(y)dy,$$ then, by the definition of weak equality,
    $$\int_{E} g(y)dM_1(y) = \int_E g(y) dM_2(y)$$
    for every continuous function $g$ and compact $E$. Setting $$g(y)=\phi(\nabla H_K(y) +\left(H_K(y) - \langle \nabla H_K(y),y\rangle\right)\cdot e)H_K(y)^{1-p}$$ yields the claimed equality.
\end{proof}

Having established Proposition~\ref{p:monge_on_perp}, the proof of the following lemma is line-by-line the same as \cite[Theorem 3.1]{FHX23}.
\begin{lemma}
\label{l:regularity}
    Let $\mu$ be a Borel measure on $\R^n$ with continuous density $\phi$. Fix $p\in \R, p\neq 0.$ Suppose a function $f$ on $\sn$ and $K\in\kno$ satisfy $0<c_1 \leq f \leq c_2$ on $\sn$ and  $dS^{\mu}_{K,p}(u)=f(u)du.$ Then:
    \begin{enumerate}
        \item $\partial K$ is $C^1$ and strictly convex, and so $h_K$ is $C^1$ smooth on $\R^n\setminus\{0\}$;
        \item if $f$ is continuous, then the restriction of $h_K$ to $\sn$ is in $C^{1,\alpha}$ for any $\alpha\in(0,1)$;
        \item if $f\in C^\alpha(\sn)$ for any $\alpha\in (0,1)$, then $h_K$ is $C^{2,\alpha}$ on $\sn$.
    \end{enumerate}
\end{lemma}

With the regularity from Lemma \ref{l:regularity} in hand, we are ready to prove the main results of this section.

\begin{proof}[Proof of the $p<0$ case of Theorems~\ref{t:first_sym} and \ref{t:first_finite}]
    We first approximate $\nu$ with a sequence of Borel measures $\nu_\ell$ such that $d\nu_\ell=f_\ell du,$ with $f_\ell\in C^+(\sn)$ (or $f_\ell \in C^+_e(\sn)$ in the case of even $\nu$). We see that $\nu_\ell$ vanishes on all great subspheres.  Then, by Lemma~\ref{l:vanishing} and Lemma~\ref{l:op_solve}, there exists a sequence $\{K_\ell\}\subset \kno (\kne)$ such that $$\nu_\ell=\lambda_\ell S^{\mu}_{K_\ell,p}, \quad \lambda_\ell= \frac{\nu(\sn)}{S^\mu_p(K_\ell)}S^{\mu}_{K_\ell,p}.$$ But actually, we may assume that $0<f_j \in C^\alpha_e(\sn)$ for some $\alpha\in(0,1)$. Then, the following Monge-Amp\`ere equation holds weakly:
    $$h_{K_\ell}^{1-p}\psi(|\nabla h_{K_\ell}|)\det(\nabla^2 h_{K_\ell}+h_{K_\ell}I)=\frac{f_\ell}{\lambda_\ell},$$
    where we used that $\mu$ has a density of the form $\psi(|x|),$ with $\psi$ decreasing and continuous on $\R_+.$
    
    From Lemma~\ref{l:regularity} Item (3), $h_K$ is then $C^{2,\alpha}$ smooth, and consequently each $K_\ell$ is smooth and uniformly convex. We will again show that the sequence is uniformly bounded; via Blaschke selection, the existence of a $K\in\kno$ (or $\kne$) such that $\mu(K)=a$ and $\nu=\lambda S^\mu_{K,p}$ follows.

    The proof of the fact that the sequence is non-degenerate follows from Proposition~\ref{prop:non_degen_2} for the $p<0$ case of Theorem~\ref{t:first_sym} and Proposition~\ref{p:origin_int_2} for the $p<0$ case if Theorem~\ref{t:first_finite}. All that remains is the upper bound. Let $h_{K_\ell}$ obtain its maximum on $\sn$ at $\theta_\ell$. We again let $g(t)=|t|$ if $\nu$ is even and the $K_\ell$ are symmetric and $g(t)=t_+$ otherwise.
    
    Since $\nu$ is not concentrated on any great subsphere, there exists $\epsilon,\delta>0$ such that
	\begin{equation}\label{mu is not concentrated on any subsphere}
		\int_{\{u\in\sn:g(\langle u, \theta_\ell\rangle)>\delta\}}f_\ell(u)du>\epsilon>0
	\end{equation}
	for large enough $\ell$. Combining \eqref{mu is not concentrated on any subsphere}, \eqref{eq:Monge-Ampere equation}, and \eqref{eq:monge_2} yields
 \begin{align*}
		0&<\frac{\epsilon}{\lambda_\ell}<\int_{\{u\in\sn:g(\langle u, \theta_\ell\rangle)>\delta\}}\frac{f_\ell(u)}{\lambda_\ell}du
  \\
  &=\int_{\{u\in\sn:g(\langle u, \theta_\ell\rangle)>\delta\}}h_{K_\ell}^{1-p}(u)\psi(|\nabla h_{K_\ell}(u)|)\det(\nabla^2 h_{K_\ell}+h_{K_\ell}I)du
  \\
  & \leq {h_{K_\ell}(\theta_\ell)^{-p}} \int_{\{u\in\sn:g(\langle u, \theta_\ell\rangle)>\delta\}}\psi(|\nabla h_{K_\ell}(u)|)h_{K_\ell}(u)\det(\nabla^2 h_{K_\ell}+h_{K_\ell}I)du
  \\ 
  &\leq {h_{K_\ell}(\theta_\ell)^{n-p}} \int_{\{u\in\sn:g(\langle u, \theta_\ell\rangle)>\delta\}}\psi(|\nabla h_{K_\ell}(u)|)du.
\end{align*}
If $g(\langle u,\theta_\ell \rangle)>\delta$, then \eqref{eq:supp_max_deriv_com} yields
 \begin{equation}
      |\nabla h_{K_\ell}(u)| >\delta h_{K_\ell}(\theta_\ell).
      \label{eq:support_inequality 3}
 \end{equation}
However, since $\mu$ is radially decreasing, the function $\psi$ is decreasing. Hence, \eqref{eq:support_inequality 3} provides the bound
 \begin{equation}
 \label{eq:support_inequality 4}
 \psi(|\nabla h_{K_\ell}(u)|) \leq \psi(\delta h_{K_\ell}(\theta_\ell))
 \end{equation}
 whenever $u\in\sn$ is so that $g(\langle u,\theta_\ell \rangle)>\delta$. Therefore, we can continue using \eqref{eq:support_inequality 4} to obtain 
\begin{align*}
    0 <& \epsilon < \lambda_\ell {h_{K_\ell}(\theta_\ell)^{n-p}}\psi(\delta h_{K_\ell}(\theta_\ell))\int_{\{u\in\sn:g(\langle u, \theta_\ell\rangle)>\delta\}}du
    \\
&\leq\lambda_\ell{h_{K_\ell}(\theta_\ell)^{n-p}}\psi(\delta h_{K_\ell}(\theta_\ell))\vol_{n-1}(\sn).
	\end{align*}
  But, the final line of the above goes to zero as $\ell\to\infty$ (since $\mu$ satisfies Property~\textbf{(D)}$_p$), a contradiction. The final claim follows from Proposition~\ref{p:converse}.
\end{proof}

\section{The Small Mass Regime}
	In this section, we will analyze smooth solutions to \eqref{eq:Monge-Ampere equation} when $p > -n-1$ by utilizing degree theory; our technique follows largely \cite{HXZ21}, and can be traced through \cite{JL22,FLX23,FHX23} as well. We will impose that our function $\phi$ is radially decreasing and rotationally invariant. Therefore, for $x\in \R^n\setminus\{0\}$, $\phi(x)=\psi(|x|)$ for some continuous, decreasing function $\psi$ on $(0,\infty)$. In which case, \eqref{eq:Monge-Ampere equation} becomes, after a change of notation that will be exclusive to this section,
 \begin{equation}\label{eq:Monge-Ampere equation_2}
		h^{1-p}\psi(|\nabla h|)\det(h_{ij}+h\delta_{ij})=f.
	\end{equation}
  We follow the usual notation that $h_i = \partial_i h$ and $h_{ij}=\partial^2_{ij} h$.
 \begin{definition}
     For a given nonnegative function $f\in C(\sn)$, nonnegative, decreasing $\psi\in C^{1}(\R_+\setminus\{0\})$ and $p\in \R$,
     we say $h$ solves \eqref{eq:Monge-Ampere equation_2} for the triple $(f,p,\psi)$ if
     \begin{enumerate}
     \item $h\in C^{2}(\sn),$
         \item $h$ is nonnegative, convex and $1$-homogeneously extended to $\R^n$, and
         \item \eqref{eq:Monge-Ampere equation_2} holds true.
     \end{enumerate}
 \end{definition}

Our main theorems for this section are as follows. 
\begin{theorem}\label{degree theory}
		Fix $n\in \mathbb N$ and $0 < p<n$. Suppose $\psi\in C^1(\R_+)$ is strictly decreasing in its support, satisfies properties \textbf{(D)}$_p$ and \textbf{(S)}$_p$, is such that $\int_0^\infty\psi(t)t^{n-1}dt<\infty$ and normalized so that the rotationally invariant measure $\mu$ on $\R^n$ with density $\psi(|\cdot|)$ is probability. Suppose additionally that $\mu$ has $L^p$ isoperimetric function $I_{p}$ over $\kne$.
  
  Let $\nu$ be an even, finite Borel measure on $\sn$, not concentrated on any great hemisphere, such that $$\nu(\sn) < \min\{I_{p}(a), I_{p}(1-a)\},$$ where $a$ is a constant depending only on $\mu$ sufficiently close to $1$. In fact, one can take $a=1/2$. Then, there exist $K_1,K_2\in\kne$ such that $\mu(K_1)>a, \mu(K_2)<1-a,$ and
        $$S^\mu_{K_1,p} = \nu = S^\mu_{K_2,p}.$$
	\end{theorem}

To explain the necessity of our assumptions, we will first establish Theorem~\ref{degree theory} in the case when $\nu$ is a multiple of the spherical Lebesgue measure (the so-called isotropic curvature flow problem), in Lemma~\ref{uniqueness} and Proposition~\ref{p:uniqueness} below. We require that the density of $\mu$ is radially decreasing and continuous at zero at this moment. We establish that there are at least two solutions to the isotropic curvature problem (each of which is a centered Euclidean ball). However, we need that there are exactly two solutions; this is precisely Property \textbf{(S)}$_p$, given in Definition~\ref{def:Sp}. Pre-supposing such a fact may seem extreme. However, write $d\mu(x)=\psi(|x|)dx.$ Then, in Proposition~\ref{p:uni_2}, we show that if $\psi$ is concave far enough down its support, then Property \textbf{(S)}$_p$ holds. 

Finally, in Lemma~\ref{C^0 estimate} and Theorem~\ref{degree theory for smooth solution}, we use degree theory to solve Theorem~\ref{degree theory} when $\nu$ has smooth density $f$. The presence of Property \textbf{(D)}$_p$ is present at all steps. If we only care about a single solution, then we can drop the assumption of symmetry and expand the range of $p$.
\begin{theorem}\label{degree theory_2}
		Fix $p\geq 1$. Suppose $\psi\in C^1(\R_+)$ is strictly decreasing on its support, satisfies properties \textbf{(D)}$_p$ and \textbf{(S)}$_p$, is such that $\int_0^\infty\psi(t)t^{n-1}dt<\infty$ and normalized so that the rotationally invariant measure $\mu$ on $\R^n$ with density $\psi(|\cdot|)$ is probability. Suppose additionally that $\mu$ has $L^p$ isoperimetric function $I_{p}$ over $\kno$.
  
  Let $\nu$ be a finite Borel measure on $\sn$, not concentrated on any great hemisphere, such that $$\nu(\sn) < I_{p}\left(\frac{1}{2}\right).$$ Then, there exists $K\in\kno$ such that $\mu(K)\geq \frac{1}{2}$ and $\nu = S^\mu_{K,p}.$ Furthermore, if $\nu$ is even, then $K\in\kne$.
	\end{theorem}

 In the next lemma, we first consider the case when $f$ is a constant; this is the so-called isotropic curvature problem. Recently, Ivaki and Milman \cite{IM23}, in a swooping \textit{triumph}, established the following concerning solutions to isotropic curvature problems.

 \begin{proposition}[Solution to isotropic curvature problems, Theorem 1.3 in \cite{IM23}]
 \label{p:isotropic}
     Suppose $K\in\kne$ and that $\partial K$ is smooth, strictly convex and has strictly positive Gauss curvature $\kappa$.

     Let $\varphi:(0,\infty)\times (0,\infty) \to (0,\infty)$ be $C^1$ with $\partial_1 \varphi \geq 0, \partial_2 \varphi \geq 0$ such that at least one of these inequalities is strict. If $K$ satisfies
     $$\varphi(h,|\nabla h|)\kappa = h^{n+2},$$
     then $K$ is a centered Euclidean ball.
 \end{proposition}
 Recall that $\nabla$, the gradient on $\R^n$, and $\nabla_s$, the gradient on $\sn$, are related via 
 \begin{equation}
 \label{eq:grad_relates}
 \nabla h(u) = \nabla_s h(u) +h(u)u.\end{equation} We note that if $h$ is a centered Euclidean ball, then its support function $h$ is constant on $\sn$. Therefore, we refer to support functions of Euclidean balls as \textit{constant solutions} to \eqref{eq:Monge-Ampere equation_2}. If $h$ is a (nonnegative) constant function on $\sn$, then the above becomes $\nabla h(u) = h\cdot u$, in particular we have $|\nabla h(u)|=h$. Also, we recall that $\det(\nabla^2 h_K + hI)$ is precisely the reciprocal of the Gauss curvature of $K$ (as a function of its outer-unit normals).
 
 \begin{lemma}\label{uniqueness}
 Fix $p>-n-1$ and a nonnegative, decreasing $\psi\in C^{1}(\R_+\setminus\{0\})$. Let $c>0$ and suppose that $h$ solves \eqref{eq:Monge-Ampere equation_2} for the triple $(c,p,\psi)$. If $[h]\in\kne$, then $h$ has to be constant.
  \end{lemma}
\begin{proof}
      Set $\varphi(x,t)=cx^{n+p+1}\psi(t)^{-1}$. Observe that \eqref{eq:Monge-Ampere equation_2} can be re-written as $\varphi(h,|\nabla h|)\kappa=h^{n+2},$ where $\kappa$ denotes the Gauss curvature of $[h]$. Observe also that $\partial_t\varphi(x,t) = cx^{n+p+1}\frac{d}{dt}\psi(t)^{-1} \geq 0$ and $\partial_x \varphi(x,t)= c(n+p+1) x^{n+p}\psi(t)^{-1} >0$. Then, according to Proposition~\ref{p:isotropic}, $h$ is constant.
  \end{proof}

  Recently, Ivaki \cite{MI24} showed the following variant of \cite[Theorem 1.3]{IM23}.
  \begin{proposition}[Solution to isotropic curvature problems, Theorem 1.2 in \cite{MI24}]
 \label{p:isotropic_2}
     Suppose $K\in\kno$ has support function $h$. Suppose also that $\partial K$ is smooth, strictly convex and has strictly positive Gauss curvature $\kappa$.

     Let $\varphi:(0,\infty)\times (0,\infty) \to (0,\infty)$ be $C^1$ with $\partial_1 \varphi \geq 0, \partial_2 \varphi \geq 0$ such that at least one of these inequalities is strict. If $K$ satisfies
     $$\varphi(h,|\nabla h|)\kappa = 1,$$
     then $K$ is a centered Euclidean ball.
     Furthermore, if $\partial_1 \varphi =0$, then the same assumption holds with $K\in\conbod$.
 \end{proposition}
 We can then obtain the following variant of Lemma~\ref{uniqueness}.
 \begin{lemma}
 \label{uniqueness_2}
     Fix $p\geq 1$ and a nonnegative, decreasing $\psi\in C^{1}(\R_+\setminus\{0\})$. Let $c>0$ and suppose that $h$ solves \eqref{eq:Monge-Ampere equation_2} for the triple $(c,p,\psi)$. If $p>1$ and $[h]\in\kno$, or $p=1$ and $[h]\in\conbod$, then $h$ has to be constant.
 \end{lemma}
 \begin{proof}
      Set $\varphi(x,t)=cx^{p-1}\psi(t)^{-1}$, where $p\geq 1$. Then, \eqref{eq:Monge-Ampere equation_2} can be re-written as $\varphi(h,|\nabla h|)\kappa=1,$ where $\kappa$ denotes the Gauss curvature of $[h]$. Also, observe that $\partial_t\varphi(x,t) = cx^{p-1}\frac{d}{dt}\psi(t)^{-1} \geq 0$ and, if $p>1$ $\partial_x \varphi(x,t)= c(p-1) x^{p-2}\psi(t)^{-1} >0$. If $p=1$, then we see $\partial_x \varphi(x,t)= 0$. Then, according to Proposition~\ref{p:isotropic_2}, $h$ is constant. 
  \end{proof}

  \subsection{Small p}
  \label{sec:small_mass_small_p} Now we turn our attention to $p \in (-n-1,n)$. We first show that, under Property \textbf{(D)}$_p$, there exist two constant solutions by making $c$ small enough when $p\in (-n-1,n)$. We recall that the support of a nonnegative function is given by
  $$\supp(f)=\{x:f(x) > 0\}.$$ We define the set 
  \begin{equation}
  \label{eq:critical_points}
  L_p(\psi)=\left\{t\in \supp(\psi): \psi(t) = -\frac{t\psi^\prime(t)}{(n-p)}\right \}.\end{equation}
  Recall that $\psi$ is said to have Property~\textbf{(D)}$_p$ if \[
    \lim\limits_{t\rightarrow\infty}t^{n-p}\psi(t)=0.
		\]
  
  \begin{proposition}\label{p:uniqueness}
      Fix $-n-1 < p <n$. Let $\psi\in C^1(\R_+)$ be a nonnegative, strictly decreasing function satisfying Property~\textbf{(D)}$_p$ such that $\psi \not\equiv 0$.
  
    For $c < \max_{t\in\supp(\psi)} t^{n-p}\psi(t)$, there exist two constant solutions $h_1,h_2$ solving \eqref{eq:Monge-Ampere equation_2} with $f=c$. If $\sup(L_p(\psi)) < \sup(\supp(\psi))$, then $c$ can be made small enough so that there are exactly two constant solutions. 
  \end{proposition}
  \begin{proof} First, notice that if $h$ is a constant, then \eqref{eq:Monge-Ampere equation_2} becomes $h^{n-p}\psi(h)=c$. Therefore, if we set $g(t) = t^{n-p}\psi(t)$, we are looking at the level curve $g(t)=c$. We will analyze the critical points of $g$. By the hypotheses on $\psi$, $g$ is differentiable on $(0,\infty)$. Observe that $g^\prime(t)=t^{n-p-1}((n-p)\psi(t)+t\psi^\prime(t))$ for $t\in (0,\infty)$. Therefore, $t^\star$ is a critical point of $g$ if and only if $t^\star\in L_p(\psi)$.
  
   First, we show that $0$ is not an accumulation point of $L_p(\psi)$. By way of contradiction, suppose it is. Then, there exists a sequence $t_k\in L_p(\psi)$ tending to zero. Observe then that
\begin{equation}\liminf_{t\to 0}\psi(t) \leq \lim_{k\to\infty}\psi(t_k) = -\frac{\psi^\prime(t_k)}{(n-p)}t_k = 0.
\label{eq:limit_bound}
\end{equation}
Indeed, fix a small $\epsilon>0$. Then, since $\psi$ is $C^1$ smooth on $[0,\epsilon]$, its derivative is bounded on that interval. Thus, there exists $M>0$ such that $|\psi^\prime(t)|\leq M$ for every $t\in [0,\epsilon]$. Recalling that $t_k\to 0$, we deduce that, for $\delta\in (0,\epsilon)$, there exists $N\in \mathbb{N}$ such that $t_k<\delta$ for $k>N$. Combining these facts, we have $\left|\frac{\psi^\prime(t_k)}{(n-p)}t_k\right|\leq \frac{M}{(n-p)}\delta$ for $k>N$. We deduce that the limit is zero by sending $\delta\to 0$. 

However, by the definition of the set $L_p(\psi)$, $-\frac{\psi^\prime(t_k)}{(n-p)}t_k=\psi(t_k)$ for all $k$. Consequently, from \eqref{eq:limit_bound} we deduce that $\psi(0)=0.$ But $\psi$ is decreasing and nonnegative; this implies $\psi \equiv 0,$ a contradiction. 

Let $t_1=\inf(L_p(\psi))$. Since $\lim_{t\to 0^+} g(t)=0$, we can at this moment pick any $c\in (0, g(t_1))$ and define $T_1 \in (0,t_1)$ to be so that $c= g(T_1)$. Notice that $g$ is increasing on $(0, t_1)$. Next, we consider two cases. The first case is $\sup(L_p(\psi)) <\sup(\supp(\psi)).$ Then, since $$\lim_{t\to \sup(\supp(\psi))} g(t)=0,$$ we find that $g$ is decreasing on $(\sup(L_p(\psi)),\sup(\supp(\psi)))$. Let $b=\inf_{t\in L_p(\psi)}g(t)$. Then, we lower $c$ if necessary (and thus pick a new $T_1$) so that $c < b$, to obtain $g(T_1)=c < b$ and there exists a unique $T_2 \in (\sup(L_p(\psi)),\sup(\supp(\psi)))$ such that $g(T_2) =c$; additionally, $g(t) \geq b > c$ on $(t_1,\sup(L_p(\psi)))$.
This shows that $T_1$ and $T_2$ are the only $t\in \supp(\psi)$ such that $g(t)=c$. 

Finally, we consider the case where $\sup(L_p(\psi)) = \sup(\supp(\psi)).$ Then, there exists a sequence $\{t_k\}\subset L_p(\psi)$ that tends to $\sup(\supp(\psi))$. Consequently, we have 
\begin{equation}
\label{eq:limit_critical}
0=(n-p)\lim_{k\to\infty}t_k^{n-p}\psi(t_k)=\lim_{k\to\infty}t_k^{n-p+1}(-\psi^\prime(t_k)).\end{equation}
Therefore, there exists $N$ such that, for $k\geq N,$
$$t_k^{n-p+1}(-\psi^\prime(t_k)) < c(n-p);$$
note that we take $N=N(c)$ to be the smallest possible for our given $c$. This rewrites as
$$t_k^{n-p}\psi(t_k) < c.$$
Therefore, $g(t_{N}) < c,$ but $g(t_{N-1}) > c$. By the intermediate value theorem, we pick $T_2$ to be the point such that $g(T_2)=c$ and $t_{N-1} < T_2 < t_{N}$.
  \end{proof}

The condition $\sup(L_p(\psi))< \sup(\supp(\psi))$ seems somewhat mysterious. We give a few sufficient conditions so that this is the case.
\begin{proposition}
\label{p:uni_2}
    Fix $-n-1 <p <n$. Let $\psi\in C^1(\R_+)$ be a nonnegative, strictly decreasing function satisfying Property~\textbf{(D)}$_p$ such that there exists $T>0$ such that $\psi$ is concave on $[T,\sup(\supp(\psi)))$. Then, there exists a small enough $c$ so that there are exactly two constant solutions $h_1,h_2$ that solve \eqref{eq:Monge-Ampere equation_2}.
\end{proposition}
\begin{proof}
Clearly, if the cardinality of $L_p(\psi)$ is finite, then $\sup(L_p(\psi))< \sup(\supp(\psi))$. Suppose that the cardinality of $L_p(\psi)$ is infinite. Then, if $t_1,t_{2}\in L_p(\psi)$, with $t_1 < t_{2}$, one has
$$-\frac{\psi^\prime(t_{2})}{(n-p)}t_{2} = \psi(t_{2}) < \psi(t_{1})= -\frac{\psi^\prime(t_{1})}{(n-p)}t_{1}.$$
Since $\psi$ is decreasing, $-\psi^\prime$ is a positive function. Then, the above re-writes as 
$$- \psi^\prime(t_{2}) < - \psi^\prime(t_{1}) \frac{t_{1}}{t_{2}} < - \psi^\prime(t_{1}). $$
Thus, the fact that $\psi$ is decreasing implies that $\psi^\prime$ is strictly increasing in $L_p(\psi)$. However, by hypothesis, $\psi^\prime$ decreases in $[T,\sup(\supp(\psi)))$. For both facts to be true, we must have $\sup(L_p(\psi)) \leq T$. The claim then follows from Proposition~\ref{p:uniqueness}. 
\end{proof}

Proposition~\ref{p:uni_2} is a bit weak in the sense that many of the measures that we care about are eventually convex. For example, $e^{-t}$ is convex on $\R_+$, and $e^{-t^2/2}$ is concave on $[0,1)$ and convex on $(1,\infty)$. We now show that, if $\psi=e^{-V}$ where $t\mapsto V(e^t)$ is convex (e.g., $V$ convex), $L_p(\psi)$ is a singleton.
\begin{proposition}
    \label{p:uni_3}
    Fix $-n-1 <p <n$. Let $\psi\in C^1(\R_+)$ be a nonnegative, strictly decreasing function satisfying Property~\textbf{(D)}$_p$ such that $\psi=e^{-V}$ with $t\mapsto V(e^t)$ convex. Then, there exists a small enough $c$ so that there are exactly two constant solutions $h_1,h_2$ that solve \eqref{eq:Monge-Ampere equation_2}.
\end{proposition}
\begin{proof}
As shown in the proof of Proposition~\ref{p:uniqueness}, if we set $g(t)=t^{n-p}\psi(t)$, then $\lim_{t\to0}g(t)=\lim_{t\to \infty}g(t)=0$. Thus, $L_p(\psi)$ is non-empty, being the critical points of $g$. We now claim that $L_p(\psi)$ is a singleton. Inserting that $\psi=e^{-V}$, the condition $t\in L_p(\psi)$ can be written as
    \begin{equation}
    \label{eq:log_con_lppsi}
    \frac{(n-p)}{t}=V^\prime(t).\end{equation}
    We take a moment to consider the illuminating case where $\psi$ is log-concave, and so $V$ is convex. Then, $V^\prime$ is increasing. However, the function $(n-p)/t$ is decreasing; these two functions can agree at only a point. 

    In our more general case, with $V(e^t)$ convex, we have that $(V(e^t))^\prime$ is increasing. Also, from the chain rule, $(V(e^t))^\prime = V^\prime(e^t)e^t$. Therefore, \eqref{eq:log_con_lppsi} becomes, from the chain rule and a change of variable,
    $$(n-p)=(V(e^t))^\prime,$$
    which ipso facto can only be true at one point. The claim then follows from Proposition~\ref{p:uniqueness}.
\end{proof}
We note that Proposition~\ref{p:uniqueness} yields a procedure to determine the number of solutions given the value of $c$ given $\psi$ and knowing $p\in (-n-1,n)$. For example, the Gaussian case. We outline this in a proposition.
  \begin{proposition}
      Suppose $h\in C^2$ solves the equation
      \[e^{-|\nabla h|^2/2}h^{1-p}\det(h_{ij}+h\delta_{ij})=c(2\pi)^{\frac{n}{2}}.\]
      If $p=1$, or $p>-n-1$ and $h$ is even and positive, or $p>1$ and $h$ is positive, then $h$ is constant, i.e., $[h]$ must be a centered Euclidean ball. Furthermore, we can characterize the number of solutions. If $p>n$, then for every $c>0$, there is exactly one solution. If $p=n,$ there is exactly one solution for $c\in (0,(2\pi)^{-n/2})$ and no solution for $c>(2\pi)^{-n/2}$. Finally, suppose $p\in (-n-1,n)$. Let $C(n,p)=(2\pi)^{-n/2}(n-p)^\frac{n-p}{2}e^{-(n-p)/2}$.
      \begin{enumerate}
          \item If $c< C(n,p)$, there are exactly two solutions.
          \item If $c=C(n,p)$, there is exactly one solution.
          \item If $c>C(n,p)$, there are no solutions.
      \end{enumerate}
  \end{proposition}
  \begin{proof}
      According to the proof of Proposition~\ref{p:uniqueness}, the number of solutions is determined by the number of solutions to $g(t)=c$, with $g(t)=(2\pi)^{-n/2}t^{n-p}e^{-t^2/2}$. Observe that, if $p>n$, then $\lim_{t\to 0^+}g(t)=\infty$ and $g(t)$ is strictly decreasing to $0$ on $(0,\infty)$ (with $g$ having no critical points). Thus, $g$ is a bijection from $(0,\infty)$ to $(0,\infty)$, and so $g(t)=c$ has a unique solution for every $c>0$. If $p=n$, then $g$ is a bijection from $\R_+$ to $(0,(2\pi)^{-n/2}]$, and the claim follows.

      For $p\in (-n-1,n)$, we examine the number of critical points of $g$, which is determined by solving
      $$\psi(t) = -\frac{t\psi^\prime(t)}{(n-p)}$$
      with
      $\psi(t)=(2\pi)^{-n/2}e^{-t^2/2}$. Then,  the formula reduces to $t^2 = (n-p)$. Thus, the critical point of $g$ is at $t=\sqrt{n-p}$. Therefore, we see that $g$ increases on $(0,\sqrt{n-p})$ from $0$ to its maximum $g=C(n,p)$, and then decreases on $(\sqrt{n-p},\infty)$ to $0$. 

      Thus, for every $T\in (0,\sqrt{n-p}),$ $g(t)=g(T)$ has exactly two solutions $T_1, T_2$ which satisfy $T_1 < \sqrt{n-p} <T_2$. Note that this also shows that if $c>g(\sqrt{n-p})(=C(n,p)),$ then there are no solutions. The fact that these are the only possible solutions is from Lemmas~\ref{uniqueness} and \ref{uniqueness_2}.
  \end{proof}
  \noindent This extends the result by Chen, Hu, Liu and Zhao \cite{CHLZ23} in the case $p=1,n=2$.
 
  In the next lemma, we analyze the case when $f$ is $\alpha$-H\"older. We recall the H\"older norm of a function $u$ over a set $U$ is given by
 $$\|u\|_{C^{2,\alpha}(U)}=\sum\limits_{|\beta|\leq 2}\sup\limits_{x\in U}|D^\beta u(x)|+\sum\limits_{|\beta|=k}[D^\beta u]_{C^{\alpha}(U)} \text{ and }[D^\beta u]_{C^{\alpha}(U)}:=\sup\limits_{x,y\in U,x\neq y}\left\{\frac{|u(x)-u(y)|}{|x-y|^\alpha}\right\}.$$
  
	\begin{lemma}\label{C^0 estimate}
		Let $f\in C_e^\alpha(\sn)$ be such that  $\frac{1}{\tau}<f<\tau$ for some positive constant $\tau$. Let $0<p<n$ and let $\psi\in C^1(\R_+\setminus\{0\})$ be a function that is strictly decreasing and satisfies Property~\textbf{(D)}$_p$. Suppose $h$ is an even function that solves \eqref{eq:Monge-Ampere equation_2} for the triple $(f,p,\psi)$.

		Then, there exists a constant $\tau'>0$, depending only on $\tau$ and $\psi$, such that
		\[
		\frac{1}{\tau'}<h<\tau'
		\quad \text{ and } \quad
    \|h\|_{C^{2,\alpha}}<\tau'.
  \]
	\end{lemma}
\begin{proof}
	We first show that $h$ is bounded from above. Let $\{u_k\}$ be a sequence in $\sn$ such that $\lim_{k\to\infty} h({u_k})=\sup_{u\in\sn} h(u)$. Then from equation \eqref{eq:Monge-Ampere equation_2},
	\begin{align}\label{upper bound}
		\frac{1}{\tau}&<f(u_k)=h(u_k)^{1-p}\psi(|\nabla h(u_k)|)\det(h_{ij}+h\delta_{ij})|_{u_k}\notag\\
		& = h(u_k)^{n-p}\psi(|\nabla h(u_k)|) \frac{\det(h_{ij}+h\delta_{ij})|_{u_k}}{h(u_k)^{n-1}} \leq \frac{\det(h_{ij}+h\delta_{ij})|_{u_k}}{h(u_k)^{n-1}} h(u_k)^{n-p}\psi(h(u_k)),
	\end{align}
 where, in the last step, we used $\psi$ was decreasing and \eqref{supp_homo}. If $h$ is not bounded from above, then $h(u_k)\rightarrow\infty$, but, from \eqref{eq:monge}, $\frac{\det(h_{ij}+h\delta_{ij})|_{u_k}}{h(u_k)^{n-1}} \to 1$ and  $\lim\limits_{k\to \infty}h(u_k)^{n-p}\psi(h(u_k))=0$; this contradicts \eqref{upper bound}. Thus $h_{\max}<C_0$ for some positive constant $C_0$ that depends only on $\tau$ and $\psi$.
    
    Next, we will show that $h$ also has a positive lower bound. We first establish that
    \begin{equation}\label{hmax lower bound}
    	h_{\max}\geq C_1
    \end{equation} 
    for some constant $C_1>0$ depending only on $\tau$ and $\psi$. Indeed, let $u_0$ be a direction so that $h(u_0)=\sup\limits_{u\in\sn} h(u)=h_{\text{max}}$. Then, from \eqref{eq:monge_2} and \eqref{eq:Monge-Ampere equation_2}, the following holds:
    \begin{align*}
       h(u_0)^{n-1}\geq\det(h_{ij}+h\delta_{ij})|_{u_0}&=f(u_0)h(u_0)^{p-1}\frac{1}{\psi(|\nabla h(u_0)|)}
       \\
       & \geq  f(u_0)h(u_0)^{p-1}\frac{1}{\psi( h(u_0))}
       \geq\frac{1}{\tau}h(u_0)^{p-1}\cdot \min_{u\in \sn}\psi(h(u))^{-1}.\\
    \end{align*}
 Since $p<n$, the claim follows with $$C_1:=\left(\min_{u\in \sn}(\tau\psi(h(u)))^{-1}\right)^{\frac{1}{n-p}}.$$ Next, using \eqref{eq:Monge-Ampere equation_2}, we have
    \begin{align*}
    	\frac{1}{n}h\det(h_{ij}+h\delta_{ij})&=\frac{1}{n\psi(|\nabla h|)}h^pf\geq \frac{1}{n\psi(h)}h^pf \geq \min_{u\in \sn}(n\tau\psi(h(u)))^{-1}h^p.
    \end{align*}
     We obtain from \eqref{eq:volume} that the total integral of the left-hand side on $\sn$ is the volume of $[h]$. Therefore, by setting $C_2 = \min_{u\in \sn}(n\tau\psi(h(u)))^{-1},$ we have
    \begin{equation}\label{total integral}
    	\vol_n([h])\geq C_2\int_{\sn}h^p(u)du.
    \end{equation}
  
    Combining \eqref{hmax lower bound}, \eqref{total integral}, and \eqref{support function inequality2}, we have
    \begin{equation}\label{lower bound of volume}
    	\vol_n([h])\geq C_2 h_{\max}^p\int_{\sn}|\langle u,u_0\rangle|^pdu\geq C_1^pC_2\int_{\sn}|\langle u,u_0\rangle|^pdu=:C_3.
    \end{equation}
    Also note the truth
    \[
    [h]\subset(h_{\max}B_2^n)\cap\{x\in\rn:|\langle x, u_0 \rangle|\leq h_{\min}\}.
    \]
    This implies
    \begin{equation}\label{upper bound of volume}
    	\vol_n([h])\leq2^n h_{\max}^{n-1}h_{\min}<2^nC_0^{n-1} h_{\min}.
    \end{equation}
    Combining \eqref{lower bound of volume} and \eqref{upper bound of volume}, $h_{\min}$ is such that
    \[
    h_{min}>\frac{1}{2^nC_0^{n-1}C_3}.
    \]
    We then define $\tau^\prime$ so that $\tau^\prime > C_0$ and $0 < \frac{1}{\tau^\prime} < \frac{1}{2^nC_0^{n-1}C_3}.$
    Finally, by Lemma \ref{l:regularity}, the $C^{2,\alpha}$ a priori estimate is also established.
\end{proof}
We will recall from Proposition~\ref{p:uniqueness} that, for $c$ small enough, we can find two constant solutions to the isotropic problem \eqref{eq:Monge-Ampere equation_2} with $f=c$. We gave examples in Propositions~\ref{p:uni_2} and \ref{p:uni_3} when there are exactly two solutions; including when the measure $\mu$ with density $\psi(|\cdot|)$ is in $\mathcal{M}_n^\infty$. We will need the existence of exactly two solutions for the following theorem. Therefore, we make it an explicit property.
\begin{definition}
\label{def:Sp}
    Fix $0 < p <n.$ Let $\psi\in C^1(\R_+)$ be nonnegative and strictly decreasing on its support. We say $\psi$ has property \textbf{(S)}$_p$ if $\sup(L_p(\psi)) < \sup(\supp(\psi))$, where $L_p(\psi)$ is given by \eqref{eq:critical_points}.
\end{definition}

\begin{lemma}
\label{l:prop_S_two}
    Fix $0 < p <n.$ Suppose $\psi\in C^1(\R_+)$ is strictly decreasing on its support, satisfies properties \textbf{(D)}$_p$ and \textbf{(S)}$_p$, is such that $\int_0^\infty\psi(t)t^{n-1}dt<\infty$ and normalized so that the rotationally invariant measure $\mu$ on $\R^n$ with density $\psi(|\cdot|)$ is probability. Then, there exists $C,a,h_1,h_2>0$ such that, if $c<C$, $h_1$ and $h_2$ are the unique solutions to \eqref{eq:Monge-Ampere equation_2} with $f=c$ and the centered Euclidean balls $[h_1]$ and $[h_2]$ satisfy $$\mu([h_1]) > a \quad \text{and} \quad \mu([h_2]) < 1-a.$$
    Furthermore, $a$ can be made independent of $c,h_1$ and $h_2$.
 Additionally, $C$ can be small enough so that $$\mu([h_2]) >\frac{1}{2} > \mu([h_1]).$$
\end{lemma}
\begin{proof}
    Recall the isotropic equation with respect to \eqref{eq:Monge-Ampere equation_2}, i.e.,
	\begin{equation}\label{isotropic Monge-Ampere equation}
			h^{1-p}\psi(|\nabla h|)\det(h_{ij}+h\delta_{ij})=c.
	\end{equation}
    From Proposition~\ref{p:uniqueness}, there exist two constant solutions to \eqref{isotropic Monge-Ampere equation}, say $h_1$ and $h_2$, for $$0<c < C:=\max_{t\in \supp(\psi)}t^{n-p}\psi(t);$$ by Property \textbf{(S)}$_p$, $c$ can be made small enough so that these are the only two. Recall that $c$ is related to $h_1$ and $h_2$ via $h_1^{n-p}\psi(h_1)=c=h_2^{n-p}\psi(h_2)$. From Lemma~\ref{uniqueness}, these solutions, which are the support functions of centered Euclidean balls, are the only solutions among positive, even $C^2(\sn)$ convex functions. Set $h_3=\inf L_p(\psi) > 0$. We then have $$0< h_1 < h_3 < h_2.$$
    Observe that $[h_1]\subset [h_3] \subset [h_2]$. Thus,
    $$\mu([h_1]) < \mu([h_3]) < \mu([h_2]) < 1.$$ But also, $\mu([h_2]) = 1 - \mu (\R^n\setminus [h_2]).$
   Naively, one can set, say, $$a=\frac{99}{100}\min\{\mu([h_1]),\mu (\R^n\setminus [h_2])\}$$ to obtain $\mu([h_1]) > a$ and $\mu([h_2]) <1-a.$ However, this bound depends on $h_1,h_2,$ and $c$.
   Instead, set $$a=\min\left\{\mu([h_3]), \mu (\R^n\setminus [h_3])\right\}.$$ Then, $$\mu([h_2]) > \mu([h_3]) \geq a$$ and $$\mu([h_1]) < \mu([h_3]) =  1 - \mu (\R^n\setminus [h_3]) \leq 1 - a.$$
   For the final claim, we recall that it was shown $h_1 \in (0,h_3)$ and $h_2\in (h_4,h_5)$, where $h_4 = \sup L_p(\psi)$ and $h_5=\sup(\supp(\psi))$. On $(0,h_5),$ the function $h\mapsto \mu([h])$ is increasing from $0$ to $1$. Thus, there exists $h_0$ and $h_6$ such that $\mu([h]) < \frac{1}{2}$ if we pick $h\in (0,h_0)$ and $\mu([h]) > \frac{1}{2}$ if we pick $h\in (h_6,h_5)$. Consequently, we first set $C=\min\{h_0^{n-p}\psi(h_0),h_3^{n-p}\psi(h_3),h_4^{n-p}\psi(h_4),h_6^{n-p}\psi(h_6)\}$ and pick any $c\in (0,C)$ to obtain the corresponding, unique, $h_1$ and $h_2$, that, by construction, satisfy $\mu([h_1]) < \frac{1}{2}<\mu([h_2])$.
\end{proof}

For the reader unfamiliar with degree theory, we recommend the work by Li \cite{Li89} for a thorough review. Recall $I_p$ isoperimetric functions, which we introduced in Definition~\ref{def: Lp_iso}.

	\begin{theorem}\label{degree theory for smooth solution}
		Fix $0 < p<n$. Suppose $\psi\in C^1(\R_+)$ is strictly decreasing on its support, satisfies properties \textbf{(D)}$_p$ and \textbf{(S)}$_p$, is such that $\int_0^\infty\psi(t)t^{n-1}dt<\infty$ and normalized so that the rotationally invariant measure $\mu$ on $\R^n$ with density $\psi(|\cdot|)$ is probability. Suppose additionally that $\mu$ has $L^p$ isoperimetric function $I_{p}$ over $\kne$.
  
  Let $f\in C_e^\alpha(\sn)$ be such that $\frac{1}{\tau}<f<\tau$ for some positive constant $\tau$ and $$\|f\|_{L^1}<f_a:=\min\{I_{p}(a), I_p(1-a)\},$$ where $a$ is a constant depending only on $\mu$ sufficiently close to $1$. Then, there exist $K_1,K_2\in\kne$ such that $\mu(K_1)>a$ and $\mu(K_2)<1-a$ and both $h_{K_1}$ and $h_{K_2}$ solve \eqref{eq:Monge-Ampere equation_2} for the triple $(f,p,\psi)$. Furthermore, one can pick $a=1/2$.
	\end{theorem}
\begin{proof}
	We will work by using a degree theoretic argument. Firstly, we consider the isotropic equation \eqref{isotropic Monge-Ampere equation}, which is \eqref{eq:Monge-Ampere equation_2} when $f=c$, a positive constant. By Lemma~\ref{l:prop_S_two}, if we take $0<c < C:=\max_{t\in \supp(\psi)}t^{n-p}\psi(t),$ then we guarantee the existence of two solutions to the isotropic equation \eqref{isotropic Monge-Ampere equation}. We also obtain a constant $a\in (0,1)$, independent of $c$, bounding the measures of the Wulff shapes of these solutions. Since $a$ is independent of $c$, we can lower $c$ if need be so that $\|c\|_{L^1}<f_a$ and the two solutions to \eqref{isotropic Monge-Ampere equation} with this $c$, denoted say $h_1$ and $h_2$, are the only two. Observe that $[h_1]$ and $[h_2]$ are centered Euclidean balls satisfying $\mu([h_1])>a$ and $\mu([h_2])<1-a$.  In fact, we saw we can lower $c$ even more, if it is necessary, and change $a$ to the more concrete value $a=1/2$. If one decides to do so, continue to denote by $h_1$ and $h_2$ the two unique solutions to \eqref{isotropic Monge-Ampere equation}. 
 
 We denote by $\triangle_{\sn}$ the spherical Laplacian. We recall that $\triangle_{\sn}$ has a discrete spectrum. Therefore, by making $c$ smaller if necessary, we have that the linearized operator of the equation \eqref{isotropic Monge-Ampere equation}, i.e.,
 \begin{equation}L_{h_i}(g)={h_i}^{n-2}\left(\triangle_{\sn}g+\left((n-p)+h_i\cdot\frac{\psi'(h_i)}{\psi(h_i)}\right)g\right)
    \label{eq:linear_op}
    \end{equation} is invertible for $i=1,2$. 
    
    	Define a family of operators $F_t:C^{2,\alpha}(\sn)\rightarrow C^\alpha(\sn)$ by
    \begin{equation*}
    	F_t(h)=\det(\nabla^2 h+hI)-h^{p-1}\frac{1}{\psi(|\nabla h|)}f_t,
    \end{equation*}
    where $f_t=(1-t)c+tf$ for some $t\in[0,1]$. Since $f\in C_e^\alpha(\sn)$, there exists a constant $\tau>0$ such that $\frac{1}{\tau}<f,c<\tau$ and $\|f\|_{C^\alpha}<\tau$. Then for each $t\in[0,1]$, $f_t$ has the same bound as $f$, i.e., $\frac{1}{\tau}<f_t<\tau$, $\|f_t\|_{L^1}<f_a$, and $\|f_t\|_{C^\alpha}<\tau$. We then let $\tau'>0$ be the constant in Lemma \ref{C^0 estimate}.
    
    For ease, set $\mathcal{S}_e= C^{2,\alpha}(\sn)\cap \{h\in C_e^+(\sn): h\text{ is convex}\}$ and define the sets $O_1,O_2\subset \mathcal{S}_e$ by
    \[
    O_1=\left\{h\in \mathcal{S}_e: \frac{1}{\tau'}<h<\tau', \|h\|_{C^{2,\alpha}}<\tau', \mu([h])>a\right\}
    \]
    and
    \[
    O_2=\left\{h\in \mathcal{S}_e: \frac{1}{\tau'}<h<\tau', \|h\|_{C^{2,\alpha}}<\tau', \mu([h])<1-a\right\}.
    \]
    Under the $\|\cdot\|_{C^{2,\alpha}}$ norm, it is easy to see that each $O_k$ is open and bounded. We claim that
    \begin{equation}\label{well define condition for degree}
    	\partial O_k\cap F_t^{-1}(0)=\emptyset
    \end{equation}
    for $k=1,2$ and $t\in[0,1]$. Indeed, if $h^{(k)}\in \partial O_k\cap F_t^{-1}(0)$, i.e., $F_t(h^{(k)})= 0$, then $h^{(k)}$ solves
    \begin{equation}\label{monge ampere equation for t}
    	(h^{(k)})^{1-p}\psi(|\nabla h^{(k)}|)\det(h^{(k)}_{ij}+h^{(k)}\delta_{ij})=f_t,
    \end{equation}
    $\mu([h^{(1)}])\geq a$ and $\mu([h^{(2)}])\leq1-a$. Since $h^{(k)}\in\partial O_k$, it must achieve equality in at least one of the conditions defining $O_k$. However, Lemma \ref{C^0 estimate} guarantees that $$\frac{1}{\tau'}<h^{(k)}<\tau' \quad \text{and} \quad\|h^{(k)}\|_{C^{2,\alpha}}<\tau'$$ are strictly satisfied. Therefore, the only way $h^{(k)}$ can lie on the boundary $\partial O_k$ is if it satisfies the measure condition with equality, i.e., $\mu([h^{(1)}]) = a$ (or $\mu([h^{(2)}]) = 1-a$). Next, by hypothesis, we have
    \[
    \|f_t\|_{L^1}=S_p^{\mu}([h^{(1)}])\geq I_p(a)>0 \text{ and } \|f_t\|_{L^1}=S_p^{\mu}([h^{(2)}])\geq I_p(1-a)>0
    \]
    which contradicts the condition $\|f\|_{L^1}<f_a$. This proves \eqref{well define condition for degree}. Then,  by \cite[Proposition 2.2]{Li89}, we obtain
    \[
    \deg(F_0,O_k,0)=\deg(F_1,O_k,0)\neq0,
    \]
    i.e., the degree $\deg(F_t,O_k,0)$ is well-defined for $t\in[0,1]$ and does not depend on $t$. We now show that $\deg(F_1,O_k,0) \neq 0$ by showing that $\deg(F_0,O_k,0)\neq 0$.
    
    For this, we will use our linear operator from \eqref{eq:linear_op}. Recalling that it is invertible, we can use \cite[Proposition 2.3]{Li89} and Property \textbf{(S)}$_p$ to obtain
    \[
    \deg(F_0,O_1,0)=\deg(L_{h_1},O_1,0)\neq0 \quad \text{and} \quad \deg(F_0,O_2,0)=\deg(L_{h_2},O_2,0)\neq0,
    \]
    where the last inequalities in each follow from \cite[Proposition 2.4]{Li89}. This implies the existence of $h^{(k)}\in O_k$ such that $F_1(h^{(k)})=0$. The claim follows.
\end{proof}

Theorem~\ref{degree theory} then follows from Theorem~\ref{degree theory for smooth solution} via approximation arguments, like in the proof of the $p<0$ case of Theorems~\ref{t:first_sym} and \ref{t:first_finite}.

\subsection{Non-symmetric Case}
In this section, we restrict the number of solutions to drop the assumption of symmetry. Recalling that Proposition~\ref{p:uniqueness} yields two solutions to the isotropic equation \eqref{isotropic Monge-Ampere equation}, and \textbf{(S)}$_p$ yields exactly two solutions via Lemma~\ref{l:prop_S_two}, we will take only the ``larger" solution in order to drop symmetry.
\begin{lemma}\label{C^0 estimate for big volume}
		Let $n\in \mathbb N, p\in\R$ and let $\psi\in C^1(\R_+\setminus\{0\})$ be a function that is strictly decreasing and satisfies Property~\textbf{(D)}$_p$, such that $\int_0^\infty\psi(t)t^{n-1}dt<\infty.$ Normalize $\psi$ so that the rotationally invariant measure $\mu$ on $\R^n$ with density $\psi(|\cdot|)$ is probability. 
        
        Let $f\in C^\alpha(\sn)$ be such that  $\frac{1}{\tau}<f<\tau$ for some positive constant $\tau$. Suppose $h$ is a function with $\mu([h])\geq\frac{1}{2}$ that solves \eqref{eq:Monge-Ampere equation_2} for the triple $(f,p,\psi)$. Then, there exists a constant $\tau'>0$, depending only on $\tau$ and $\psi$, such that
		\[
		\frac{1}{\tau'}<h<\tau'
		\quad \text{ and } \quad
    \|h\|_{C^{2,\alpha}}<\tau'.
  \]
	\end{lemma}
\begin{proof}
	We first note that we have $h_{\max}<C_0$ for some positive constant $C_0$ which depends only on $\tau$ and $\psi$, as the proof is the same as the first part of Lemma~\ref{C^0 estimate}. Next, from Proposition~\ref{p:origin_int}, we then obtain $[h]\in\kno$. This yields a positive lower bound on $h$, namely $h \geq \min_{u\in\sn} h(u)>0$. However, we claim we can find a uniform bound away from zero depending only on $\psi$ and $\tau$. Suppose not, i.e.,
 \[\inf\limits_{K\in\mathcal{Q}}\left\{\min\limits_{u\in\sn}h_K(u)\right\}=0,
 \]
 where $$\mathcal{Q}:=\left\{K\in\kno: \mu(K)\geq\frac{1}{2} \text{ and } h_K \text{ solves \eqref{eq:Monge-Ampere equation_2} for the triple $(f,p,\psi)$}\right\}.$$ Then for any $\epsilon>0$, there exists a $K_\epsilon\in\kno$ and $u_\epsilon\in\sn$ such that 
 $$h_{K_\epsilon}(u_\epsilon):=\min\limits_{u\in\sn}h_{K_\epsilon}(u)<\epsilon.$$ 
 Thus, $K_\epsilon\subset (H_\epsilon+\epsilon u_\epsilon)\cap RB_2^n$, where $R$ is a finite, positive constant since $\max_{u\in\sn}h_{K_\epsilon}(u)<C_0$ and
 $$H_\epsilon =\{x\in\rn: \langle x, u_\epsilon\rangle \leq0\}.$$
 Next, observe that
\[
\frac{1}{2}=\mu(H_\epsilon)=\mu(H_\epsilon\cap RB_2^n)+\mu(H_\epsilon\setminus RB_2^n),
\]
where $\mu(H_\epsilon\setminus RB_2^n) = c > 0$ for some constant $c$ that is strictly positive and independent of $\epsilon$ due to the rotational invariance of $\mu$ and the fact that $u_\epsilon$ is a unit vector.

But observe that 
\begin{align*}
\mu(K_\epsilon) &\leq\mu(\{x\in\rn: \langle x,u_\epsilon\rangle \leq\epsilon\}\cap RB_2^n) \\
&= \mu(H_\epsilon\cap RB_2^n) + \mu(\{x\in\rn: 0 < \langle x,u_\epsilon\rangle \leq\epsilon\}\cap RB_2^n).
\end{align*}
For $\epsilon$ small enough, the measure of this thin slice is strictly less than $c$. Thus,
\[
\mu(K_\epsilon) < \mu(H_\epsilon\cap RB_2^n) + c = \frac{1}{2}.
\]
This contradicts the assumption that $\mu(K_\epsilon) \geq \frac{1}{2}$. Consequently, there exists a constant that depends only on $\tau$ and $\psi$ bounding $\inf\limits_{K\in\mathcal{Q}}\left\{\min\limits_{u\in\sn}h_K(u)\right\}$ away from zero.
 \end{proof}
 
 Having established Lemma~\ref{C^0 estimate for big volume}, the proof of the following theorem is exactly the same as Theorem~\ref{degree theory for smooth solution}, and is omitted.  We simply note that one uses Lemma~\ref{uniqueness_2} for the uniqueness of the isotropic solutions, and we also have the necessary analogue of Lemma~\ref{l:prop_S_two} when $p\geq n$, which means specifically we can find a unique solution to the isotropic problem \eqref{isotropic Monge-Ampere equation} satisfying $\mu([h])\geq 1/2$ when $c$ is small enough. In fact, the properties \textbf{(D)}$_p$ and \textbf{(S)}$_p$ are actually unnecessary. This follows from the function $g(t)=t^{n-p}\psi(t)$ being always decreasing when $p\geq n$ and $\psi$ is decreasing, and the fact that $g(t)$ is a bijection from $(0,\infty)$ to $(0,\infty)$ for $p >n$ and from $[0,\infty)$ to $(0,\psi(0)]$ when $p=n$.
 \begin{theorem}
 \label{degree theory for smooth solution_no_sym}
     Fix $p\geq 1$. Suppose $\psi\in C^1(\R_+)$ is strictly decreasing on its support, satisfies properties \textbf{(D)}$_p$ and \textbf{(S)}$_p$, is such that $\int_0^\infty\psi(t)t^{n-1}dt<\infty$ and normalized so that the rotationally invariant measure $\mu$ on $\R^n$ with density $\psi(|\cdot|)$ is probability. Suppose additionally that $\mu$ has $L^p$ isoperimetric function $I_{p}$ over $\kno$.
  
  Let $f\in C^\alpha(\sn)$ be such that $\frac{1}{\tau}<f<\tau$ for some positive constant $\tau$ (if $p=n$, additionally assume that $\tau < \psi(0)$). Suppose also that $$\|f\|_{L^1}<f_a:=I_{p}\left(\frac{1}{2}\right).$$ Then, there exists $K\in\kno$ such that $\mu(K)\geq \frac{1}{2}$ and $h_{K}$ solves \eqref{eq:Monge-Ampere equation_2} for the triple $(f,p,\psi)$. Furthermore, if $f\in C_e^\alpha(\sn),$ then $K\in \kne$.
 \end{theorem}
 Theorem~\ref{degree theory_2} then follows from Theorem~\ref{degree theory for smooth solution_no_sym} by approximation. As we show in the next section, we can use a completely different method to study the range $p\geq n$ and remove the norm assumption on $f$. 

\subsection{Large p}
\label{sec:small_mass_large_p}
In this subsection, we repeat many of the efforts from the previous subsection, but when $p \geq n$. We are able to drop the symmetry assumption in this case. We no longer need to assume Property~\textbf{(D)}$_p$, as, for $p\geq n$, the fact that $\psi$ is strictly decreasing implies that this property is true. We first establish the uniqueness.

\begin{lemma}\label{Uniqueness for p>n}Fix $p\geq n$ and a nonnegative function $f\in C(\sn).$  Let $\psi\in C^1(\R_+)$ be a nonnegative, strictly decreasing function. Suppose the equation \begin{equation} 
\label{eq:monge_4}
  h^{1-p}\psi(|\nabla h|)\det(h_{ij}+h\delta_{ij})=f\end{equation}
  is solved by a positive $h\in C^{2,+}(\sn)$. Then, $h$ is unique.
	\end{lemma}
	\begin{proof}
		 Let $h_1,h_2 \in C^{2,+}(\sn)$ be two positive solutions of \eqref{eq:monge_4}, and set $G(x)=\frac{h_1(x)}{h_2(x)}$. By hypothesis $G$ is positive and $G\in C^{2}(\sn)$. Then, there exists $u_0\in \sn$ such that $G$ reaches its maximum at $u_0$. Thus,
		$$
		0=\nabla G(u_0)=\frac{(\nabla h_1(u_0))h_2(u_0)-h_1(u_0)\nabla h_2(u_0)}{h_2(u_0)^2}
		$$
		and
		$$
		0\geq\nabla^2G(u_0)=\frac{h_2(u_0)\nabla^2h_1(u_0)-h_1(u_0)\nabla^2h_2(u_0)}{h_2^2(u_0)},
		$$
		i.e.,
		$$
		\frac{\nabla^2h_1(u_0)}{h_1(u_0)}\leq\frac{\nabla^2h_2(u_0)}{h_2(u_0)}.
		$$
		Hence,
		\begin{equation*}
			\begin{split}
				f(u_0)h_1(u_0)^{p-1}\psi(|\nabla h_1(u_0)|)^{-1}&=h_1(u_0)^{n-1}\det\left(\frac{\nabla^2h_1(u_0)}{h_1(u_0)}+I\right)\\
				&\leq h_1(u_0)^{n-1}\det\left(\frac{\nabla^2h_2(u_0)}{h_2(u_0)}+I\right)\\
				&=\frac{h_1(u_0)^{n-1}}{h_2(u_0)^{n-1}}f(u_0)h_2(u_0)^{p-1}\psi(|\nabla h_2(u_0)|)^{-1},
			\end{split}
		\end{equation*}
		i.e.,
		$$
		h_1^{p-n}(u_0)\psi(|\nabla h_1(u_0)|)^{-1}\leq h_2(u_0)^{p-n}\psi(|\nabla h_2(u_0)|)^{-1}.
		$$
		Set $\frac{|\nabla h_1(u_0)|}{h_1(u_0)}=\frac{|\nabla h_2(u_0)|}{h_2(u_0)}=c,$ and then
		$$
		h_1^{p-n}(u_0)\psi(c h_1(u_0))^{-1}\leq h_2^{p-n}(u_0)\psi(c h_2(u_0))^{-1},
		$$
		Observe that the function $t\mapsto t^{p-n}\psi(c t)^{-1}$ is increasing, since  $p\geq n$ and $\psi$ is decreasing. Hence
		$$
		h_1(u_0)\leq h_2(u_0),
		$$
		which implies
		$
		h_1\leq h_2
		$ on $\sn$. Similarly, we get $h_1\geq h_2$ on $\sn$. Thus $h_1=h_2$.
	\end{proof}
Then, we turn to the existence. Our first step is again a $C^{2,\alpha}$-estimate.

\begin{lemma}\label{a priori estimate in section 6}
   For $p\geq n$, let $f\in C^\alpha(\sn)$ be such that  $\frac{1}{\tau}<f<\tau$ for some positive constant $\tau$ when $p>n$, and $f<\psi(0)$ when $p=n$. Let $\psi\in C^1(\R_+\setminus\{0\})$ be a function that is strictly decreasing. Suppose $h$ solves \eqref{eq:Monge-Ampere equation_2} for the triple $(f,p,\psi)$. Then, there exists a constant $\tau'>0$, depending only on $\tau$ and $\psi$, such that
		\[
		\frac{1}{\tau'}<h<\tau'
		\quad \text{ and } \quad
    \|h\|_{C^{2,\alpha}}<\tau'.
  \]
	\end{lemma}
\begin{proof}
    The upper bound is the same as Lemma \ref{C^0 estimate}. We only need to prove the positive lower bound. Let $\{v_k\}$ be a sequence in $\sn$ such that $\lim_{k\rightarrow\infty}h(v_k)=\inf_{v\in\sn}h(v)$. Then from \eqref{eq:Monge-Ampere equation_2} and \eqref{eq:monge}, 
    \begin{align}\label{lower bound for p>=n}
        \tau&>f(v_k)=h^{1-p}(v_k)\psi(|\nabla h(v_k)|)\det(h_{ij}+h\delta_{ij})|_{v_k}\notag\\
        &\geq h^{n-p}(v_k)\psi(|\nabla h({v_k})|) = h^{n-p}(v_k)\psi(h(v_k))\frac{\psi(|\nabla h({v_k})|)}{\psi(h(v_k))}.
    \end{align}
By way of contradiction, suppose $h(v_k)\rightarrow 0$. From \eqref{eq:grad_relates}, $\frac{\psi(|\nabla h({v_k})|)}{\psi(h(v_k))} \to 1$ as $k\to \infty$. On the other hand, $h(v_k)^{n-p}\psi(h(v_k))\rightarrow\infty$ as $k\rightarrow \infty$ when $p>n$, and $h(v_k)^{n-p}\psi(h(v_k))\rightarrow\psi(0)$ when $p=n$; this contradicts \eqref{lower bound for p>=n}. Thus $h<\tau'$ for some positive constant $\tau'$ that depends only on $\tau$ and $\psi$.

 Moreover, by Lemma \ref{l:regularity}, the $C^{2,\alpha}$ a priori estimate is also established.
\end{proof}
In the next lemma, we introduce the linear operator of \eqref{eq:monge_4} and show it is invertible.
  \begin{lemma} \label{linear operator is invertible}Fix $p\geq n$. Let $\psi\in C^1(\R_+)$ be a strictly decreasing nonnegative function.  Fix a positive function  $f\in C^\alpha(\sn)$. Suppose that $h\in C^+(\sn)$ is a convex function solving  \eqref{eq:monge_4}. Then the linearized operator of \eqref{eq:monge_4} at $h$, defined in \eqref{eq:linear_op_2} below, is invertible.
  
  \end{lemma}
\begin{proof}
Let $\Gamma\in C^{2}(\sn)$  and set $h_\epsilon=he^{\epsilon \Gamma}$. Then, the linear operator $L_h$ of \eqref{eq:monge_4} at $h$ acting on $\Gamma$, $L_h(\Gamma)\in C(\sn)$, is given by
\begin{equation}
\begin{split}
\label{eq:linear_op_2}
			L_h(\Gamma)&:=\frac{d}{d\epsilon}\det((h_\epsilon)_{ij}+h_\epsilon\delta_{ij})|_{\epsilon=0}-f\frac{d}{d\epsilon}[h_\epsilon^{p-1}\psi(|\nabla h_\epsilon|)^{-1}]|_{\epsilon=0}\\
			&=\sum_{i,j}(\omega^{ij}(h_{ij}\Gamma+h_j\Gamma_i+h_i\Gamma_j+h\Gamma_{ij}+h\Gamma\delta_{ij}))\\
			&-f\left[(p-1)h^{p-1}\Gamma\psi(|\nabla h|)^{-1}+h^{p-1}(\psi(|\nabla h|)^{-1})'\frac{1}{|\nabla h|}\left(\left(h^2+\sum_{i}h_i^2\right)\Gamma+\sum_{i}h(h_i\Gamma_i)\right)\right]
		\end{split}
  \end{equation}
		where $\omega^{ij}$ is the cofactor matrix of $(h_{ij}+h\delta_{ij})_{i,j}$. Since $p\geq n$ and $\psi$ is strictly decreasing, we deduce 
\begin{align*}
    L_h(1)
    &=\sum_{i,j}\omega^{ij}(h_{ij}+h\delta_{ij})\\
    &-f\left[(p-1)h^{p-1}\psi(|\nabla h|)^{-1}+h^{p-1}(\psi(|\nabla h|)^{-1})'\frac{1}{|\nabla h|}\left(h^2+\sum_{i}h_i^2\right)\right]\\
    &=f\left[(n-p)h^{p-1}\psi(|\nabla h|)^{-1}-h^{p-1}(\psi(|\nabla h|)^{-1})'|\nabla h|\right]<0.
\end{align*}

  Observe that we can write
  \begin{equation}
\label{eq:linear_op_new}	
 L_h(\Gamma)=\Gamma L_h(1)+\sum_{i,j}\omega^{ij}(h_j\Gamma_i+h_i\Gamma_j+h\Gamma_{ij})-\sum_{i}fh^{p-1}hh_i\Gamma_i.
   \end{equation}
  Suppose $\Gamma$ is a function so that $L_h(\Gamma)=0$. Let $u_0\in \sn$ be a point where $\Gamma$ obtains its maximum, which implies $\nabla_s \Gamma(u_0) = 0$ (so $\Gamma_i(u_0) = 0$). Then, \eqref{eq:linear_op_new} becomes
  \begin{equation*}
  \Gamma(u_0) L_h(1)[u_0]+\sum_{i,j}\omega^{ij}(u_0) h(u_0)\Gamma_{ij}(u_0)=0.
  \end{equation*}
 
  Combining $L_h(1)<0$ and the fact that $\omega^{ij}$ is positive definite (since $h$ is convex), with $\Gamma_{ij}$ being semi-negative definite at $u_0$, we obtain 
  \begin{equation*}
  \Gamma(u_0) \leq0,
  \end{equation*}
and therefore $\Gamma \leq 0$. Similarly, by repeating the above procedure at a point where $\Gamma$ obtains its minimum, we obtain $ \Gamma\geq0.$ Thus, 
  \begin{equation*}
  \Gamma\equiv0.
  \end{equation*}
 This indicates \eqref{eq:linear_op_2} is invertible.
 \end{proof} 

Finally, we use convex combinations and the continuity of the Monge-Amp\`ere equation to establish existence.
\begin{theorem}
\label{continuity method_local}
 Fix $\alpha\in(0,1)$ and $p\geq n$. Let $\psi\in C^1(\R_+)$ be a strictly decreasing nonnegative function. Let $f\in C^\alpha(\sn)$ be such that
 \begin{enumerate}
     \item when $p>n$: $\frac{1}{\tau}<f<\tau$ for some positive constant $\tau$.
     \item when $p=n$: $f < \psi(0)$.
 \end{enumerate}
 Then, there exists a unique, positive, convex solution $h\in C^{2,\alpha}(\sn)$ to \eqref{eq:monge_4} for the triple $(f,p,\psi)$. 
\end{theorem}
\begin{proof}
    We first show that there exists a constant solution to \eqref{eq:monge_4}. From Proposition~\ref{p:uniqueness}, we see that the condition $\lim_{t\to\infty}t^{n-p}\psi(t)=0$ yields such a solution when $f=c_0$ for $c_0$ small enough. Since $c_0$ is strictly positive, we can find a constant $0<\tau$ such that $\frac{1}{\tau} <c_0 < \tau$. Notice that, when $p=n$, the condition reduces to $\psi$ being decreasing. Since $\psi$ has its maximum at the origin, we are free to take $\tau = \psi(0)$ when $p=n$.
    
    Now, consider a family of functions:
 \begin{equation}\label{t monge ampere equation}
 h^{1-p}\psi(|\nabla h|)\det(h_{ij}+h\delta_{ij})=f_t,
 \end{equation}
 where $f_t=(1-t)c_0+tf$, $t\in[0,1]$ and $c_0,f$ such that $\frac{1}{\tau}<c_0,f<\tau$. Thus $\frac{1}{\tau}<f_t<\tau$. For ease, set $\mathcal{S}= C^{2,\alpha}(\sn)\cap \{h\in C^+(\sn): h\text{ is convex}\}$. Next, define
 \[
 O=\{t\in[0,1]:\eqref{t monge ampere equation}\ \text{admits a solution}\ h_t\in \mathcal{S}\}.
 \]
 We have shown that $0\in O$, and thus $O$ is not empty. Next, we prove that $O$ is an open set. Set $ F(t,h):=h^{1-p}\psi(|\nabla h|)\det(h_{ij}+h\delta_{ij})-f_t$. For any $t^\star\in O$, by the definition of $O$, there exists a solution $h_{t^\star}\in \mathcal{S}$ that satisfies equation \eqref{t monge ampere equation}, i.e., $F(t^\star,h_{t^\star})=0$. From Lemma \ref{linear operator is invertible}, we may use the Implicit Function Theorem to obtain $r,r_1>0$ such that, if we define the following metric balls, $B_r(t^\star)\subset [0,1]$ and $B_{r_1}(h_{t^\star})\subset \mathcal{S}$, then there exists a unique $u\in C^1(B_r(t^\star),B_{r_1}(h_{t^\star}))$ with the following property: $u(t^\star)=h_{t^\star}$ and $F(t,u(t))=0$ for any $t\in B_r(t^\star)$. Thus, $O$ is an open set.
 
 Next, we prove that $O$ is a closed set. Let $\{t_k\}_{k=1}^\infty$ be a sequence in $O$ and suppose $t_k\to t_0$ as $k\to\infty$. We will show that $t_0\in O$. From the definition of $O$, for every $t_k$, there exists a solution $h^{(k)}:=h_{t_k}\in\mathcal{S}$ that satisfies \eqref{t monge ampere equation}. Notice that $\frac{1}{\tau}<f_{t_k}<\tau$. Thus, by Lemma \ref{a priori estimate in section 6}, there exists a constant $\tau'$, which only depends on $\tau$ and $\psi$, such that $\|h^{(k)}\|_{C^{2,\alpha}}<\tau'$. Then, by the Arzel\`a-Ascoli Theorem, there exists a subsequence $\{h^{(k_\ell)}\}\subset\{h^{(k)}\}$ and a $h^{(0)}\in \mathcal{S}$ satisfying $h^{{(k_\ell)}}\rightarrow h^{(0)}$ as $\ell\rightarrow\infty$. Moreover,
  \[
  (h^{(k_\ell)})^{1-p}\psi(|\nabla h^{(k_\ell)}|)\det(h^{(k_\ell)}_{ij}+h^{(k_\ell)}\delta_{ij})=f_{t_{k_\ell}}.
  \]
  Taking the limit as $\ell\rightarrow\infty$ in the above equation, we deduce
  \[
  (h^{(0)})^{1-p}\psi(|\nabla h^{(0)}|)\det(h^{(0)}_{ij}+h^{(0)}\delta_{ij})=f_{t_0},
  \]
  which means $h_{t_0}:=h^{(0)}$ solves \eqref{t monge ampere equation} when $t=t_0$. By definition, $t_0\in O$. Thus, $O$ is a closed set.

  Consequently, $O$ is both open and closed in $[0,1]$; since $[0,1]$ is connected, we must have $O=[0,1]$. In particular, when $t=1$, there exists a solution $h:=h_1\in \mathcal{S}$ to equation \eqref{eq:monge_4}. From Lemma~\ref{a priori estimate in section 6}, this solution is positive. The uniqueness of the solution then follows from Lemma~\ref{Uniqueness for p>n}.
\end{proof}

We deduce the following from Theorem~\ref{continuity method_local} via a standard approximation argument.
\begin{theorem}\label{continuity method}
		Fix $n\in \mathbb N$ and $p\geq n$. Suppose $\psi \in C^1(\R_+)$ is strictly decreasing on its support. Let $\mu$ be the rotationally invariant measure on $\R^n$ with density $\psi(|\cdot|)$.  Let $\nu$ be a finite Borel measure on $\sn$, not concentrated on any great hemisphere.

  If $p>n$, then, there exists $K\in\kno$ such that $S^\mu_{K,p} = \nu.$
    If $p=n$, and we additionally assume that $\nu(\sn) \leq \psi(0)n\kappa_n$, then, there exists $K\in\kno$ such that $S^\mu_{K,n} = \nu.$
	\end{theorem}

\appendix

\section{Isoperimetric Inequalities}
\label{sec:iso}
This section is dedicated to discussing the isoperimetric function of a measure on $\R^n$, which we recall is the largest function satisfying the inequality $\mu^+(\partial A) \geq I_\mu(\mu(A)).$
 
 Let us discuss a few examples; the first three can be found in, say, the work by Bakry and Ledoux \cite{BL96}, and the last is from Barthe and Maurey \cite{BM00}:
\begin{enumerate}
    \item The classic isoperimetric inequality on $\R^n$ is precisely that $I_{\vol_n}(t)=n\vol_n(B_2^n)^\frac{1}{n}t^{\frac{n-1}{n}}$, and the extremal sets are the Euclidean balls.
    \item Set $\Phi(t)=\gamma_1((-\infty,t))$. Then, $I_{\gamma_n}=\Phi^{\prime}\circ\Phi^{-1}$. Here, $\Phi^\prime$ denotes the derivative of $\Phi$ in $t$, $\Phi^{-1}$ denotes the inverse of $\Phi$, and $\circ$ denotes function composition. Here, the extremal sets are half-spaces.
    \item Let $\mu$ be a probability measure on $\R^n$ with smooth, strictly positive density $e^{-V}$ such that $\hess{V(x)}\geq \rho \mathrm{Id}_n$. Here, $\rho>0$, $\mathrm{Id}_n$ is the $n\times n$ identity matrix, and the inequality is with respect to symmetric matrices uniformly for every $x\in\R^n$. Then, $I_\mu \geq \sqrt{\rho}I_{\gamma_n}$.
    \item Let $\mu$ be the uniform measure on the cube $[-1,1]^n$. Then, $I_\mu \geq \sqrt{2\pi}I_{\gamma_n}$.
\end{enumerate}
All of the above examples are $\log$-concave measures. Bobkov \cite{BS99} completed the picture for this class of measures:
if $\mu$ is a log-concave probability measure on $\R^n$, then, for every Borel set $A$, every point $x_0\in\R^n$ and every $r>0$, one has
$$2r\mu^+(\partial A)\geq -\mu(A)\log\mu(A)-(1-\mu(A))\log(1-\mu(A))+\log(\mu(rB_2^n+x_0)).$$
Since we are free to pick $x_0$, we can set $x_0 = \int_{\R^n}xd\mu(x)$. Then, we can find an $r=r(\mu(A))$ sufficiently large so that the sum of the last two terms is zero, i.e., $r(t)$ on $[0,1]$ is implicitly defined by
$$\mu\left(rB_2^n+\int_{\R^n}xd\mu(x)\right) = (1-\mu(A))^{(1-\mu(A))}.$$
Then, one obtains
$\mu^+(\partial A) \geq \ell(\mu(A)),$
i.e., $I_\mu \geq \ell.$ Here, $\ell(t)=-\frac{t}{2r(t)}\log t$ is a function on $[0,1]$. 

Bobkov \cite{BS07} also studied $s$-concave measures (in the sense of Borell). Let $\mu$ be $s$-concave. Then,
$$I_\mu \geq \frac{c(s)}{m(\mu)}\left(\min\{\mu(A),1-\mu(A)\}\right)^{1-s},$$
where $m(\mu)$ is so that $\mu\left(\{x:|x|\leq m(\mu)\}\right)=\frac{1}{2}$ and
\begin{equation}c(s)=
    \begin{cases}
        \frac{2^{-s}-1}{40(2^{s^2-2s}-1)} & s<0,
        \\
        \frac{1}{80} & s=0,
        \\
        \frac{1}{16} & s\in (0,1].
    \end{cases}
\end{equation}

From these results, we deduce that any finite, rotationally invariant $s$-concave measure $\mu$ whose density is $C^1$, has its maximum at the origin, and satisfies Property~$\mathbf{(D)_p}$ (for a fixed $p \in (-n-1,n)$) is valid for Theorem~\ref{t:small_mass_best}. In fact, we see that, for $\log$-concavity ($s=0$), we have two different choices of isoperimetric function. But we also have, from the ground-breaking work by Milman and Rotem \cite{MR14}, two different choices of isoperimetric functions for $s>0$ under an additional homogeneity assumption: Let $\mu$ be $s$-concave and $1/s$-homogeneous with $s\in (0,\frac{1}{n}]$. Then, 
$$I_\mu(t) = \frac{1}{q}\mu(B_2^n)^qt^{1-q}, \quad \text{where } \frac{1}{q}=\frac{1}{s}+n,$$
and the extremal sets are precisely dilates of $B_2^n$.

We hope we have provided ample examples illustrating Theorem~\ref{t:small_mass_best}. We wrap up with discussing Cheeger constants. We say that a probability measure $\mu$ satisfies an isoperimetric inequality of Cheeger-type if there exists a constant $c$ such that
\begin{equation}I_\mu(t) \geq c \min\{t,1-t\},
\label{eq:cheeger}
\end{equation}
i.e., $c$ satisfies
$$\mu^+(\partial A) \geq c \min\{\mu(A),1-\mu(A)\}.$$
The largest constant $c$ where \eqref{eq:cheeger} holds is called the isoperimetric, or Cheeger, constant of $\mu$, and this will be denoted as: $Is(\mu)=\inf \frac{\mu^+(\partial A)}{\min\{\mu(A),1-\mu(A)\}},$
where the infimum runs over all open sets with smooth boundary such that $\mu(A)\in (0,1)$. 

Similar to the case of the isoperimetric function, one usually shows a constant $c$ satisfies \eqref{eq:cheeger}, thus bounding $Is(\mu)$ from below. The isoperimetric constant was introduced in the context of spectral gaps and isoperimetric inequalities on compact manifolds by Cheeger \cite{JC70}.

Let $\mu$ be the Borel measure on $\R$ with density $\chi_{[0,\infty)}(t)e^{-t}$. Then, Talagrand \cite{TM91} showed that $I_\mu(t)=\min\{t,1-t\}$, which also yields that $Is(\mu)=1$. This is a rare instance where the optimal isoperimetric inequality of the form \eqref{eq:Iso_eq} coincides with the optimal Cheeger-type isoperimetric inequality \eqref{eq:cheeger}. It is natural to try and extend this result to higher dimensions. Bobkov and Houdr\'e \cite{BH97} showed that, if $\mu$ is a probability measure on $\R$ and $\mu^n$ is the product measure on $\R^n$ given by the $n$-fold product of $\mu$, then 
$$Is(\mu^n)\geq \frac{1}{2\sqrt{6}}Is(\mu).$$

It behooves us to mention the following concerning the isoperimetric constant. A natural question is to bound $Is(\mu)$ from below independently of $\mu$ and the dimension. Firstly, set $Is(n)=\inf Is(\mu),$ where the infimum runs over all isotropic, log-concave measures on $\R^n$ (recall a measure is isotropic if $\int_{\R^n}x_id\mu(x)=0$ and $\int_{\R^n}x_ix_jd\mu(x)=\delta_{i,j}$ for $i,j=1,\dots,n$). The KLS conjecture \cite{KLS95} suggests that there exists a constant $C$ independent of the dimension such that $Is(n) \geq C$. This conjecture is still open, and implies the famous hyperplane conjecture by Bourgain. As of writing, the most recent progress is by Klartag \cite{BK23}: $Is(n) \geq C(\log n)^{-\frac{1}{2}}$ for some absolute constant $C$.

We turn to the $L^p$ isoperimetric function from Definition~\ref{def: Lp_iso}. We first recall the following from \cite[Theorem 6.1]{KL23}.
\begin{proposition}[Lp Minkowski's First Inequality for F-concave measures]
\label{t:LpFcon}
Fix $p\in \R, p\neq 0$. Let $\mu$ be a Borel measure on $\R^{n}$, such that $\mu$ is $L^p$ $F$-concave and $F$ is differentiable, with respect to a class of Borel sets $\mathcal{C}$. Then, for every $K,L\in\mathcal{C}$, one has that:
$$
\mu_p(K; L) \geq \mu_p(K ; K)+\frac{F(\mu(L))-F(\mu(K))}{F^{\prime}(\mu(K))}.
$$
\end{proposition}
\noindent Next, notice that from \eqref{eq:Lp_mixed_measures},
\begin{align*}
\frac{r^{1/p}}{p}S^\mu_{p}(K) &=   r^{1/p}\lim_{\epsilon \to 0}\frac{\mu(K+_p \epsilon \cdot B_2^n)-\mu(K)}{\epsilon} =  \lim_{\epsilon \to 0}\frac{\mu(K+_p \epsilon \cdot rB_2^n)-\mu(K)}{\epsilon} = \mu_p(K;rB_2^n).
\end{align*}
Therefore, for $p>0$ and $r>0$, we obtain from Proposition~\ref{t:LpFcon} that
\begin{equation}
\label{eq:Lpiso}
S^\mu_{p}(K) \geq r^{-1/p} \left[\mu(K ; K)+p\frac{F(\mu(rB_2^n))-F(\mu(K))}{F^{\prime}(\mu(K))}\right],
\end{equation}
where we used that $p\mu_p(K;K)=\mu(K;K)$. Since $\mu(K;K) \geq 0$, we see that \eqref{eq:Lpiso} can always be made positive. Indeed, define $r=r(\mu(K))$ as follows: pick $r$ so that $$\text{sgn}(F^\prime(\mu(K))) F(\mu(rB_2^n)) >\text{sgn}(F^\prime(\mu(K))) F(\mu(K)).$$

Therefore, for $p>0$, \eqref{eq:Lpiso} reduces the question of establishing an $L^p$ isoperimetric inequality to establishing a measure is $L^p$ $F$-concave. Following Hosle, Kolesnikov, and Livshyts \cite{HKL21}, let us consider the case where $F(x)=x^\frac{q}{n}$ for some $q>0$, i.e., $\mu$ satisfies the following $(p,q)$-Brunn-Minkowski-type inequality
\begin{equation}
\label{eq:BMLPQ}
\mu((1-\lambda)\cdot K +_p \lambda \cdot L) \geq \left((1-\lambda)\mu(K)^\frac{q}{n}+\lambda \mu(L)^\frac{q}{n}\right)^\frac{n}{q}.\end{equation}
Then,  \eqref{eq:Lpiso} becomes
\begin{equation}
\label{eq:LP_isoperimetric_inequality}
S^\mu_{p}(K) \geq r^{-\frac{1}{p}}\left(\mu(K;K) + \frac{np}{q}\left(\mu(rB_2^n)^\frac{q}{n}-\mu(K)^\frac{q}{n}\right)\mu(K)^\frac{n-q}{q}\right).\end{equation}

We have essentially covered the known examples of \eqref{eq:BMLPQ} when $p\geq 1$ in Section~\ref{sec:concavity}. As for $p <1$, this is still an active area of research (even in the case of the Lebesgue measure). For example, in \cite{BLYZ12} it was shown that the case $(0,0)$ for volume and symmetric convex bodies implies the case $(p,p)$ for all $p>0$; the case $(0,0)$ was then verified for $n=2$ in the same work, when both bodies are balls \cite{CLM17}, and also when $K$ is a ball and $L$ is close to $K$ \cite{CL20}. Saroglou \cite{CS15} established the $(0,0)$ case for volume when all bodies are symmetric under orthogonal coordinate hyperplanes, and showed that the case $(0,0)$ for symmetric convex bodies and volume implies the same for all even log-concave measures. As for other positive results, the $(p,p)$-inequality for volume has been shown to hold for $p$ close enough to $1$ (see, e.g., Kolesnikov and Milman, \cite{KM22}, Putterman \cite{EP21}, and Chen, Huang, Li and Liu \cite{CHLL20}). For measures besides volume, Livshyts \cite{GL23} recently established that every even log-concave measure satisfies a $(1,n^{-3-o(1)})$ inequality over symmetric convex bodies. The reader is encouraged to see Hosle, Kolesnikov and Livshyts \cite{HKL21} for a thorough overview of other known results and the establishment of more partial cases, including results for measures besides Lebesgue.

For $\mu\in\mathcal{M}_n^\infty$, when $\mathcal{C}=\kne$, \eqref{eq:LP_isoperimetric_inequality} is valid when $q=1$ and $p\geq 1$. The quantity $\mu(K;K)$ seems somewhat mysterious. We conclude with a conjecture that would supply a more quantifiable isoperimetric inequality for measures in $\mathcal{M}_n^\infty$ using the $(1/n)$-concavity. We follow a schema by Bobkov \cite{BS07}. Firstly, one can verify, if $\mu$ is a probability measure on $\R^n$, $K$ is a convex set and $r>0$,
that \begin{equation}
\label{eq:new_equation_mu_bd}
    \mu^{+}(\partial K)=\lim _{\epsilon \to 0^+} \frac{\mu\left((1-\epsilon) K+\epsilon rB_2^n\right)+\mu\left((1-\epsilon) \left(\R^n\setminus \text{int}(K)\right)+\epsilon rB_2^n\right)-1}{2 r \epsilon}.
\end{equation}
Indeed, this follows from the fact that $\mu(K;rB_2^n)=r\mu(K;B_2^n)$ and \eqref{eq:mixed_form}, which actually holds in this instance; $\partial \left(\R^n\setminus \text{int}(K)\right)$ is the same as $\partial K$, but with opposite outer-unit normal. 

Suppose also $\mu$ is an $F$-concave probability measure, in the sense of \eqref{eq:concave}, over a collection of sets $\mathcal{C}$, such that $F$ is $C^1$ smooth, $\{rB_2^n\}_{r>0}\subset \mathcal{C}$, and, if $K\in\mathcal{C}$, then so too is $(\R^n\setminus \text{int}(K))$. Then, setting $I=(F^{-1})^\prime\circ F$, from \eqref{eq:new_equation_mu_bd} one has that, for every $r>0$ and $K\in\mathcal{C}$,
\begin{equation}
\label{eq:F_iso}
\begin{split}
2r\mu^+(\partial K)&\geq I(\mu(K))(F(\mu(rB_2^n))-F(\mu(K)))
\\
&\quad+I(1-\mu(K))(F(\mu(rB_2^n))-F(1-\mu(K))).
\end{split}
\end{equation}
The result on $s$-concave probability measures by Bobkov followed by setting $F(t)=t^s$ in \eqref{eq:F_iso}, and some further analysis.

If $F$ satisfies $F(t)=-F(1-t)$, i.e., $F$ is an odd function about $t=\frac{1}{2}$, then \eqref{eq:F_iso} reduces to
\begin{equation}
\label{eq:F_iso_odd}
\mu^+(\partial A)\geq \frac{I(\mu(A))(F(\mu(rB_2^n)))}{r}.\end{equation}
The Ehrhard inequality is precisely the statement that $\gamma_n$ is $\Phi^{-1}$-concave. Furthermore, $\Phi^{-1}$ is odd about $\frac{1}{2}$, and so \eqref{eq:F_iso_odd} holds with $\mu=\gamma_n$ and $F=\Phi^{-1}$. Sending $r\to \infty$, and noting that $\lim_{r\to\infty}\frac{\Phi^{-1}(\mu(rB_2^n))}{r}=1$, recovers $I_{\gamma_n}=\Phi^{\prime}\circ\Phi^{-1}$. 
\\
\textbf{Conjecture} $\,$ Let $\mu$ be a probability measure over $\R^n$ that is $(1/n)$-concave over $\kne$. Then, for every $K,L\in \kne$ and $t\in (0,1),$
    $$\mu((1-t) (\R^n\setminus \text{int}(K))+ tL)\geq \left((1-t)\mu((\R^n\setminus \text{int}(K)))^\frac{1}{n}+t\mu(L)^\frac{1}{n}\right)^n.$$
\\
Suppose that the conjecture is true. Then, we could consider the case where $\mathcal{C}\supset \kne$ and $F(x)=x^{\frac{1}{n}}$ in \eqref{eq:F_iso}. Then $(F^{-1})^\prime(x)= nx^{n-1}$ and $I(x)=nx^\frac{n-1}{n}$. Consequently, \eqref{eq:F_iso} becomes, for every $r>0$ and $K\in\kne$,
\begin{equation}
\label{eq:iso_1_n}
\mu^+(\partial K) \geq \frac{n}{2r}\left[\mu(rB_2^n)^\frac{1}{n}\left(\mu(K)^\frac{n-1}{n}+\left(1-\mu(K)\right)^\frac{n-1}{n}\right) 
-1\right].
\end{equation}
Suppose $\mu$ has continuous density that contains the origin in the interior of its connected support. We then claim that we can define a function $r=r(\mu(K))$ so that the right-hand side of \eqref{eq:iso_1_n} is strictly positive. Indeed, we would need
$$\mu(rB_2^n) > \left(\mu(K)^\frac{n-1}{n}+\left(1-\mu(K)\right)^\frac{n-1}{n}\right)^{-n}=z(\mu(K)),$$ where
$$z(t):= \left(t^\frac{n-1}{n}+\left(1-t\right)^\frac{n-1}{n}\right)^{-n}.$$
Observe that $z(t)$ maps $[0,1]$ onto $[\frac{1}{2},1]$, with maximum $z(0)=z(1)=1$ and minimum $z\left(\frac{1}{2}\right)=\frac{1}{2}$. The function $r \to \mu(rB_2^n)$ is strictly increasing on its support. Thus, it has an inverse function, say $\mathcal{R}_\mu$. Therefore, the function we need is, for $\epsilon >0$, 
\begin{equation} 
\label{eq:small_iso_1_n}
r_\mu(t,\epsilon)=\mathcal{R}_\mu(z(t))+\epsilon,
\end{equation} as this yields $\mu(r_{\mu}(\mu(K),\epsilon)B_2^n) > \mu(\mathcal{R}_\mu(z(\mu(K)))B_2^n)=z(\mu(K))$. We make this formal:
    Let $\mu\in\mathcal{M}_n^\infty$. Then, under the condition that the above conjecture is true, one has, for every $K\in \kne$ and $\epsilon >0$
    $$\mu^+(\partial K) \geq \mathcal{I}_\mu(\mu(K),\epsilon),$$
    where
    \[\mathcal{I}_\mu(t,\epsilon) = \frac{n}{2r_\mu(t,\epsilon)}\left(\mu(r_\mu(t,\epsilon)B_2^n)^\frac{1}{n}\left(t^\frac{n-1}{n}+\left(1-t\right)^\frac{n-1}{n}\right) 
-1\right),\]
and $r_\mu(t,\epsilon)$ is the function defined in \eqref{eq:small_iso_1_n}.

 {\bf Acknowledgments: }We thank Professors Yong Huang and Deping Ye for taking a look at a draft of this work. We also thank the anonymous referees for their detailed comments, which vastly improved the presentation of the results herein.

{\bf Funding: } The first named author was funded by the Fondation Sciences Math\'ematiques de Paris Postdoctoral Program and the U.S. National Science Foundation's MSPRF fellowship via NSF grant DMS-2502744. The second named author was funded by the China National Postdoctoral Program for Innovative Talents of CPSF (BX20240102) and by the National Natural Science Foundation of China (12401252).

{\bf Conflicts of interest, Ethics Approval and Data Availability: } The authors have no competing interests to declare that are relevant to the content of this article and no ethics approval was needed for this research. No research data were used for this work.

\bibliographystyle{siam}

\end{document}